\documentclass[a4,10pt]{article}
\usepackage{amsmath}
\usepackage{amssymb}
\usepackage{amsthm}
\usepackage{graphicx}
\usepackage{amscd}
\usepackage{ulem}
\usepackage{makeidx}
\usepackage{fancybox}

\newcommand{\ccc}{{\mathbf C}}

\newcommand{\nnn}{{\mathbf N}}

\newcommand{\zzz}{{\mathbf Z}}
\renewcommand{\ggg}{{\frak{g}}}
\newcommand{\hhh}{{\frak{h}}}

\newtheorem{thm}{Theorem}[section]
\newtheorem{prop}{Proposition}[section]
\newtheorem{lemma}{Lemma}[section]
\newtheorem{cor}{Corollary}[section]

\newtheorem{conj}{Conjecture}[section]

\newtheorem{rem}{Remark}[section]

\numberwithin{equation}{section}

\begin{document}

\title{
Mock theta functions and characters of N=3 superconformal modules}

\author{\footnote{12-4 Karato-Rokkoudai, Kita-ku, Kobe 651-1334, Japan, \,\ 
wakimoto@r6.dion.ne.jp, \,\ wakimoto.minoru.314@m.kyushu-u.ac.jp
}{ Minoru Wakimoto}}

\date{\empty}

\maketitle

\begin{center}
Abstract
\end{center}

In this paper we study the characters of N=3 superconformal modules 
by using the Zwegers' theory on modification of mock theta functions.

\tableofcontents

\section{Introduction}

As it is known well, the N=3 superconformal algebra is constructed as the W-algebra 
associated to the affine superalgebra $\widehat{B}(1,1)= \widehat{osp}(3|2)$. 
In this paper, we consider the N=3 modules obtained by the quantum Hamiltonian 
reduction of $\widehat{B}(1,1)$-modules which are integrable with respect to 
the roots $\alpha_2$ and $\delta-\alpha_2$ and atypical with respect to 
$\alpha_1$, Among these modules, we consider  the highest weight $\widehat{B}(1,1)$-module 
$L(\Lambda^{[K(m), m_2]})$ where 
$m$ is a positive integer and $m_2$ is a non-negative integer such that 
$m_2 \leq m$ and $K(m) := -\frac{m+2}{4}$ and 
$\Lambda^{[K(m), m_2]}:= K(m)\Lambda_0-\frac{m_2}{2}\alpha_1$. 

The characters of these N=3 modules are Appell's functions 
({\it cf}. \cite{KW2001}) 
and their modular properties are usually unclear and difficult.

The case $m=1$ (i.e., $K=-\frac34$) was discussed in \S 5 of \cite{KW2017b}.

The purpose of the present paper is to study the characters in the case $m=2$ 
(i.e., $K=-1$) by using the Zwegers' theory on mock modular forms. 
In this case, we will show that the space of characters 
is not $SL_2(\zzz)$-invariant but the $SL_2(\zzz)$-invariance is achieved 
by collaboration of honest characters and modified characters. 
Also we show that the modified characters are explicitly written 
by the Dedekind's $\eta$-function and the Mumford's theta 
functions $\vartheta_{ab}(\tau, z)$ $(a, b \in \{0,1\})$.
The method to compute the characters of N=3 modules in this paper is as follows:
\begin{itemize}
\item First, compute all modified characters by using 
the action of $SL_2(\zzz)$.
\item Next, compute the Zwegers' correction terms.
\end{itemize}
These are possible in the case $m=2$ namely $K=-1$, and then we 
obtain the explicit formulas for both honest and modified characters.
The results obtained in the case $m=2$ enable us to study also 
the case $m=4$.

\medskip

This paper is organized as follows.

In section \ref{sec:Phi} we review the Zwegers' modification theory 
and collect formulas which are necessary to work out our calculation.

In section \ref{sec:B11}, as preparation for calculation in this paper, 
we give formulas for (super)characters and twisted characters of 
integrable $\widehat{B}(1,1)$-modules.

In section \ref{sec:quantum} we compute the characters of the 
N=3 modules obtained from the quantum Hamiltonian reduction of 
integrable $\widehat{B}(1,1)$-modules.
For the $SL_2(\zzz)$-invariance of the space of modified characters, 
we consider 2 kinds of twisted characters.

In section \ref{sec:m2:modified} we compute modified characters in the case 
$m=2$ by using their modular properties. 

In section \ref{sec:m2:honest}, as continuation from section 
\ref{sec:m2:modified}, we compute the Zwegers' correction terms and obtain 
the explicit formulas for the honest characters in the case $m=2$. 
By this calculation it turns out that the modified character, in this case, 
is nothing else but the sum of two honest characters.

In section \ref{sec:asymptotics} we compute the asymptotic behavior 
of (super)characters as $\tau \downarrow 0$ in the case $m=2$ 
(i.e., $K=-1$). We see that 
the asymptotics of supercharacters remain unchanged under the modification, 
whereas the asymptotics of characters behave in somehow strange way.

In sections \ref{sec:m4:modified} and \ref{sec:m4:honest} 
we discuss about relations between characters of \\
$H(\Lambda^{[K(m), m_2]})$ and $H(\Lambda^{[K(2m), m_2]})$ and, 
as its application, we deduce explicit formulas for some of modified
and honest characters in the case $m=4$.

Finally in section \ref{sec:theta-relation} we show some theta-relations 
of the Mumford's theta functions $\vartheta_{ab}(\tau, z)$ obtained 
from the modification of characters.

Through the study developped in this paper, we are led to a conjecture 
about the relation between honest characters and modified characters.
If this conjecture is true, we can expect that the Zwegers' modification 
theory of mock theta functions will provide very powerful tools 
for the study of characters of W-algebras.

The author is grateful to Professor Victor Kac for useful discussions and 
interest in this work.

\section{Functions $\Phi^{[m,s]}$ and $\widetilde{\Phi}^{[m,s]}$}
\label{sec:Phi}
%(line=177)

We begin this section with brief and quick review on the the Zwegers' 
modification theory of mock theta functions which are used in this paper.
For $m \in \frac12 \nnn$ and $s \in \frac12\zzz$, we define the 
functions $\Phi^{[m,s]}_i$ \,\ $(i=1,2)$ by
%(label=n3:eqn:2022-111a)
%(label=n3:eqn:2022-111b)
%(label=n3:eqn:2022-111c)
\begin{subequations}
\begin{eqnarray}
\Phi^{[m,s]}_1(\tau, z_1, z_2,t) &:=& e^{-2\pi imt} 
\sum_{j \in \zzz}
\frac{e^{2\pi imj(z_1+z_2)+2\pi isz_1} q^{mj^2+sj}}{1-e^{2\pi iz_1}q^j}
\label{n3:eqn:2022-111a}
\\[1mm]
\Phi^{[m,s]}_2(\tau, z_1, z_2,t) &:=& e^{-2\pi imt} 
\sum_{j \in \zzz}
\frac{e^{-2\pi imj(z_1+z_2)-2\pi isz_2} q^{mj^2+sj}}{1-e^{-2\pi iz_2}q^j}
\label{n3:eqn:2022-111b}
\end{eqnarray}
\end{subequations}
where $q:=e^{2\pi i\tau} \,\ (\tau \in \ccc_+)$ and $z_1, z_2, t \in \ccc$, 
and put 
\begin{equation}
\Phi^{[m,s]}(\tau, z_1, z_2,t) \,\ := \,\ 
\Phi^{[m,s]}_1(\tau, z_1, z_2,t) -\Phi^{[m,s]}_2(\tau, z_1, z_2,t) 
\label{n3:eqn:2022-111c}
\end{equation}

These functions appear usually in the characters of affine superalgebras 
and superconformal algebras, but do not have good modular properties.
This situation is improved by the Zwegers' modification, using the 
functions $R_{j;m}(\tau,w)$ for $(m, j) \in \frac12\nnn \times \frac12\zzz$ 
defined by 
%(label=n3:eqn:2022-111h)
\begin{eqnarray}
& &
R_{j;m}(\tau,w) := \hspace{-6mm}
\sum_{\substack{ \\[1mm] n \, \equiv \, j \, {\rm mod} \, 2m}}
\hspace{-6mm}
\bigg\{{\rm sgn}\Big(n-\frac12-j+2m\Big)
-E\bigg(\Big(n-2m\frac{{\rm Im}(w)}{{\rm Im}(\tau)}\Big)
\sqrt{\frac{{\rm Im}(\tau)}{m}}\bigg)\bigg\}
\nonumber
\\[0mm]
& & \hspace{40mm}
\times \,\ e^{-\frac{\pi in^2\tau}{2m} +2\pi inw}
\label{n3:eqn:2022-111h}
\end{eqnarray}
where $(\tau,w) \in \ccc_+ \times \ccc$ and 
$E(z):= 2\int_0^z e^{-\pi t^2}dt$.
For $m \in \frac12 \nnn$ 
and $s \in \frac12 \zzz$, we put 
%(label=n3:eqn:2022-111j)
%(label=n3:eqn:2022-111k)
\begin{subequations}
\begin{eqnarray}
& &\hspace{-10mm}
\Phi^{[m;s]}_{\rm add}(\tau, z_1,z_2,t) 
\nonumber
\\[1mm]
&:=&
-\frac12 e^{-2\pi imt} \hspace{-5mm}
\sum_{\substack{\\ k \in s+\zzz \\[1mm] 
s \leq k < s+2m }} \hspace{-5mm}
R_{k;m}\Big(\tau, \frac{z_1-z_2}{2}\Big) 
\big[\theta_{k;m}-\theta_{-k;m}\big](\tau, \, z_1+z_2) \hspace{10mm}
\label{n3:eqn:2022-111j}
\\[0mm]
& & \hspace{-10mm}
\widetilde{\Phi}^{[m;s]}(\tau, z_1,z_2,t) \, := \, 
\Phi^{[m;s]}(\tau, z_1,z_2,t) 
+
\Phi^{[m;s]}_{\rm add}(\tau, z_1,z_2,t) 
\label{n3:eqn:2022-111k}
\end{eqnarray}
\end{subequations}
and call $\widetilde{\Phi}^{[m;s]}$ the \lq \lq modification" of $\Phi^{[m;s]}$
and $\Phi^{[m;s]}_{\rm add}$ the \lq \lq correction term". In the formula 
\eqref{n3:eqn:2022-111j}, $\theta_{k,m}$ is the Jacobi's theta function :
%(label=n3:eqn:2022-120a)
\begin{equation}
\theta_{k,m}(\tau, z) \,\ := \,\ \sum_{j \in \zzz}
e^{2\pi im(j+\frac{k}{2m})z}q^{m(j+\frac{k}{2m})^2}
\label{n3:eqn:2022-120a}
\end{equation}

The properties of these functions were studied in 
\cite{KW2014}, \cite{KW2016a}, \cite{KW2016b}, \cite{KW2017a}, \cite{KW2017b} 
and \cite{Z}. Among them we collect here only the formulas which are necessary 
for our purpose in this paper.

%(line=258)
%(label=n3:lemma:2022-116c)
\begin{lemma}
\label{n3:lemma:2022-116c}
If $m \in \nnn$ and $j \in \frac12 \zzz$, then 
$R_{j,m}(\tau, w \pm \frac12) = e^{\pm \pi ij}R_{j,m}(\tau, w)$
\end{lemma}

%(line=266)
%(label=n3:lemma:2022-111a)
\begin{lemma}
\label{lemman3:n3:2022-111a}
\label{n3:lemma:2022-111a}
Let $m\in \frac12 \nnn$ and $s, s' \in \frac12 \zzz$. Then 

\begin{enumerate}
\item[{\rm 1)}] if $s-s'\in \zzz$, then 
$\widetilde{\Phi}^{[m;s]}=\widetilde{\Phi}^{[m;s']}$.

\item[{\rm 2)}] $\Phi^{[m,s]}(\tau, z_2, z_1,t) 
=\Phi^{[m,1-s]}(\tau, z_1, z_2,t)$. 

\item[{\rm 3)}] $\widetilde{\Phi}^{[m,s]}(\tau, z_2, z_1,t) 
=\widetilde{\Phi}^{[m,s]}(\tau, z_1, z_2,t)$. 

\item[{\rm 4)}] $\Phi^{[m,s]}_2(\tau, z_1, z_2,t) 
=\Phi^{[m,s]}_1(\tau, -z_2, -z_1,t)$.
\end{enumerate}
\end{lemma}

\vspace{1mm}
%(line=300)
%(label=n3:lemma:2022-111b)
\begin{lemma}
\label{n3:lemma:2022-111b}
\label{lemma:n3:2022-111b}
Let $m \in \frac12\nnn$ and $s \in \frac12 \zzz$. Then
\begin{enumerate}
\item[{\rm 1)}] if \, $s \in \zzz$, \quad 
$\widetilde{\Phi}^{[m;s]}
\Big(-\dfrac{1}{\tau}, \dfrac{z_1}{\tau},\dfrac{z_1}{\tau},t\Big) 
=
\tau e^{\frac{2\pi im}{\tau}z_1z_2}
\widetilde{\Phi}^{[m;s]}(\tau, z_1, z_2, t)$
\item[{\rm 2)}] if \, $m+s \in \zzz$, \, then 
\begin{enumerate}
\item[{\rm (i)}] $\widetilde{\Phi}^{[m;s]}(\tau+1, z_1, z_2, t) 
=\widetilde{\Phi}^{[m;s]}(\tau, z_1, z_2, t)$
\item[{\rm (ii)}] $\Phi^{[m;s]}(\tau+1, z_1, z_2, t) 
=\Phi^{[m;s]}(\tau, z_1, z_2, t)$
\end{enumerate}
\end{enumerate}
and similar formula holds for $\widetilde{\Phi}^{[m,s]}_i$ 
\,\ $(i \in \{1,2\})$.
\end{lemma}

\vspace{1mm}
%(line=327)
%(label=n3:lemma:2022-202b)
\begin{lemma}
\label{n3:lemma:2022-202b}
Let $m \in \nnn$ and $s \in \frac12 \zzz$ and $a, b \in \zzz$ such that 
$a+b \in 2\zzz$. Then
\begin{enumerate}
\item[{\rm 1)}] \,\ $\Phi^{[m,s]}(\tau, z_1+a, z_2+b, t) 
\, = \, 
(-1)^{2sa} \, \Phi^{[m,s]}(\tau, z_1, z_2, t)$ 
\item[{\rm 2)}] \,\ In the case $s \in \zzz$,
$$
\Phi^{[m,s]}(\tau, z_1+a, z_2+b, t) 
\, = \, 
\Phi^{[m,s]}(\tau, z_1, z_2, t) \hspace{5mm}
\text{for} \,\ {}^{\forall}a, \, {}^{\forall}b \, \in \, \zzz
$$ 
\end{enumerate}
and similar formula holds for $\Phi^{[m,s]}_i$, 
$\Phi_{i, {\rm add}}^{[m,s]}$ and $\widetilde{\Phi}^{[m,s]}_i$ 
\,\ $(i \in \{1,2\})$.
\end{lemma}

\vspace{1mm}
%(line=350)
%(label=n3:lemma:2022-111c)
\begin{lemma}
\label{n3:lemma:2022-111c}
\label{lemma:n3:2022-111c}
For $m \in \frac12 \nnn$ and $s \in \frac12\zzz$,
\begin{eqnarray*}
& & \hspace{-10mm}
\Phi^{[m,s]}(2\tau, z_1, z_2, t)
\\[1mm]
&=& \frac12 \left\{
\Phi^{[2m,2s]}
\left(\tau, \frac{z_1}{2}, \frac{z_2}{2}, \frac{t}{2}\right) 
+ 
e^{-2\pi is} \Phi^{[2m,2s]}
\Big(\tau, \frac{z_1+1}{2}, \frac{z_2-1}{2}, \frac{t}{2}\Big)
\right\}
\end{eqnarray*}
and similar formula holds for $\Phi^{[m,s]}_i$, 
$\Phi_{i, {\rm add}}^{[m,s]}$ and $\widetilde{\Phi}^{[m,s]}_i$ 
\,\ $(i \in \{1,2\})$.
\end{lemma}

\vspace{1mm}
%(line=373)
%(label=n3:lemma:2022-111d)
%(label=note:2021-1212a)
\begin{lemma} 
\label{n3:lemma:2022-111d}
\label{lemma:n3:2022-111d}
Let $m \in \frac12 \nnn$, $s \in \frac12\zzz$ and $a, b \in \ccc$. Then
$$
\widetilde{\Phi}^{[m,s]}\left(\tau,  
z+a+\dfrac{\tau}{2}, z+b-\dfrac{\tau}{2}, 0\right)
=
e^{2\pi im(a-b)} 
\widetilde{\Phi}^{[m,s]}\left(\tau, 
z+a-\dfrac{\tau}{2}, z+b+\dfrac{\tau}{2}, 0\right)
$$
and similar formula holds for $\widetilde{\Phi}^{[m,s]}_i$ 
\,\ $(i \in \{1,2\})$.
\end{lemma}

\vspace{1mm}
%(line=350)
%(label=n3:lemma:2022-111e)
\begin{lemma} ($\widehat{sl}(2|1)$-denominator identity) :
\label{n3:lemma:2022-111e}
\label{lemma:n3:2022-111e}
$$
\widetilde{\Phi}^{[1,s]}(\tau, z_1, z_2,t) 
= \Phi^{[1,s]}(\tau, z_1, z_2,t) 
= -i e^{-2\pi it}
\frac{
\eta(\tau)^3 \vartheta_{11}(\tau, z_1+z_2)}{
\vartheta_{11}(\tau, z_1)\vartheta_{11}(\tau, z_2)}
$$
for ${}^{\forall}s \in \zzz$, where 
$\vartheta_{ab}(\tau,z)$ are the Mumford's theta functions 
(\cite{Mum}).
\end{lemma}

\vspace{1mm}

We note also the following simple but useful formulas:

\vspace{1mm}
%(line=372)
%(label=n3:lemma:2021-1213a)
\begin{lemma} 
\label{n3:lemma:2021-1213a}
Let $m \in \frac12 \nnn$, $s \in \frac12 \zzz$ and $a \in \nnn$. Then 

\begin{enumerate}
\item[{\rm 1)}] \quad $\Phi_1^{[m,s]}(\tau, z_1, z_2,0) 
\, - \, \Phi_1^{[m,s+a]}(\tau, z_1, z_2,0) $
$$
= \,\ \sum_{k=0}^{a-1} \, 
e^{\pi i(s+k)(z_1-z_2)} \, q^{-\frac{(s+k)^2}{4m}}
\, \theta_{s+k, \, m}(\tau, \, z_1+z_2)
$$

\item[{\rm 2)}] \quad $\Phi_2^{[m,s]}(\tau, z_1, z_2,0) 
\, - \, \Phi_2^{[m,s+a]}(\tau, z_1, z_2,0) $
$$
= \,\ \sum_{k=0}^{a-1} \, 
e^{\pi i(s+k)(z_1-z_2)} \, q^{-\frac{(s+k)^2}{4m}}
\, \theta_{-(s+k), \, m}(\tau, \, z_1+z_2)
$$

\item[{\rm 3)}] \quad $\Phi^{[m,s]}(\tau, z_1, z_2,t) 
-\Phi^{[m,s+a]}(\tau, z_1, z_2,t) $
$$
= \,\ e^{-2\pi imt} \sum_{k=0}^{a-1} \, 
e^{\pi i(s+k)(z_1-z_2)} \, q^{-\frac{(s+k)^2}{4m}}
\, \big[\theta_{s+k, \, m}-\theta_{-(s+k), \, m}
\big](\tau, \, z_1+z_2)
$$
\end{enumerate}
\end{lemma}

\begin{proof} 1) Letting $s \rightarrow s+a$ \, in \eqref{n3:eqn:2022-111a},
we have
$$
\Phi_1^{[m,s+a]}(\tau, z_1, z_2,0) 
\,\ = \,\ \sum_{j \in \zzz} \,
\frac{e^{2\pi imj(z_1+z_2)} \, q^{mj^2}
\, (e^{2\pi iz_1}q^j)^{s+a}}{1-e^{2\pi iz_1}q^j} \, ,
$$
and so 
{\allowdisplaybreaks
\begin{eqnarray*}
\lefteqn{
\Phi_1^{[m,s]}(\tau, z_1, z_2,0) 
\, - \, \Phi_1^{[m,s+a]}(\tau, z_1, z_2,0)} 
\\[1mm]
&=&
\sum_{j \in \zzz}  e^{2\pi imj(z_1+z_2)} \, q^{mj^2}
\, (e^{2\pi iz_1}q^j)^s
\frac{1-(e^{2\pi iz_1}q^j)^a}{1-e^{2\pi iz_1}q^j}
\\[1mm]
&=&
\sum_{j \in \zzz} \sum_{k=0}^{a-1}e^{2\pi imj(z_1+z_2)} \, q^{mj^2}
\, (e^{2\pi iz_1}q^j)^s 
(e^{2\pi iz_1}q^j)^k
\\[1mm]
&=&
\sum_{k=0}^{a-1} e^{\pi i(s+k)(z_1-z_2)} 
q^{-\frac{(s+k)^2}{4m}}
\sum_{j \in \zzz}
e^{2\pi im(j+\frac{s+k}{2m})(z_1+z_2)} \, q^{m(j+\frac{s+k}{2m})^2}
\\[1mm]
&=&
\sum_{k=0}^{a-1} \, e^{\pi i(s+k)(z_1-z_2)} \, q^{-\frac{(s+k)^2}{4m}}
\, \theta_{s+k, \, m}(\tau, \, z_1+z_2) \, ,
\end{eqnarray*}}
proving 1). 

\medskip

\noindent
2) By Lemma \ref{n3:lemma:2022-111a}, we have  
{\allowdisplaybreaks
\begin{eqnarray*}
\lefteqn{\Phi_2^{[m,s]}(\tau, z_1, z_2,0)
\, - \, 
\Phi_2^{[m,s+a]}(\tau, z_1, z_2,0)}
\\[1mm]
&=&
\Phi_1^{[m,s]}(\tau, -z_2, -z_1,0) 
\, - \, \Phi_1^{[m,s+a]}(\tau, -z_2, -z_1,0)
\\[1mm]
&=&
\sum_{k=0}^{a-1} \, e^{\pi i(s+k)(-z_2+z_1)} \, q^{-\frac{(s+k)^2}{4m}}
\, \theta_{s+k, \, m}(\tau, \, -(z_1+z_2))
\\[1mm]
&=&
\sum_{k=0}^{a-1} \, e^{\pi i(s+k)(z_1-z_2)} \, q^{-\frac{(s+k)^2}{4m}}
\, \theta_{-(s+k), \, m}(\tau, \, z_1+z_2) \, ,
\end{eqnarray*}}
proving 2). \, 3) follows from 1) and 2).
\end{proof}

The following functions $\widetilde{A}^{[m]}_i(\tau, z)$ 
$(m \in \nnn$ and $i=1 \sim 6)$ play important roles 
to describe the characters of N=3 modules.
%(label=n3:eqn:2022-111d1) $\sim$ (label=n3:eqn:2022-111d6)
\begin{subequations}
{\allowdisplaybreaks
\begin{eqnarray}
\widetilde{A}^{[m]}_1(\tau, z) &:=&\widetilde{\Phi}^{[m,0]}\Big(
\tau, \,\ \frac{z}{2}+\frac{\tau}{4}+\frac14, \,\ 
\frac{z}{2}-\frac{\tau}{4}-\frac14, \,\ 0 \Big)
\label{n3:eqn:2022-111d1}
\\[1mm]
\widetilde{A}^{[m]}_2(\tau, z) &:=&\widetilde{\Phi}^{[m,0]}\Big(
\tau, \,\ \frac{z}{2}+\frac{\tau}{4}-\frac14, \,\ 
\frac{z}{2}-\frac{\tau}{4}+\frac14, \,\ 0 \Big)
\label{n3:eqn:2022-111d2}
\\[1mm]
\widetilde{A}^{[m]}_3(\tau, z) &:=&\widetilde{\Phi}^{[m,0]}\Big(
\tau, \,\ \frac{z}{2}+\frac{\tau}{4}+\frac12, \,\ 
\frac{z}{2}-\frac{\tau}{4}-\frac12, \,\ 0 \Big)
\label{n3:eqn:2022-111d3}
\\[1mm]
\widetilde{A}^{[m]}_4(\tau, z) &:=&\widetilde{\Phi}^{[m,0]}\Big(
\tau, \,\ \frac{z}{2}+\frac{\tau}{4}, \,\ 
\frac{z}{2}-\frac{\tau}{4}, \,\ 0 \Big)
\label{n3:eqn:2022-111d4}
\\[1mm]
\widetilde{A}^{[m]}_5(\tau, z) &:=&\widetilde{\Phi}^{[m,0]}\Big(
\tau, \,\ \frac{z}{2}+\frac14, \,\ \frac{z}{2}-\frac14, \,\ 0 \Big)
\label{n3:eqn:2022-111d5}
\\[1mm]
\widetilde{A}^{[m]}_6(\tau, z) &:=&\widetilde{\Phi}^{[m,0]}\Big(
\tau, \,\ \frac{z}{2}+\frac{\tau}{2}-\frac14, \,\ 
\frac{z}{2}-\frac{\tau}{2}+\frac14, \,\ 0 \Big)
\label{n3:eqn:2022-111d6}
\end{eqnarray}}
We note that, by Lemma \ref{n3:lemma:2022-111d}, 
$\widetilde{A}^{[m]}_6(\tau, z)$ is written as follows:
%(label=n3:eqn:2022-111e)
\begin{equation}
\widetilde{A}^{[m]}_6(\tau, z) = 
e^{-\pi im} \, \widetilde{\Phi}^{[m,0]}\Big(
\tau, \,\ \frac{z}{2}+\frac{\tau}{2}+\frac14, \,\ 
\frac{z}{2}-\frac{\tau}{2}-\frac14, \,\ 0 \Big) \, .
\label{n3:eqn:2022-111e}
\end{equation}
\end{subequations}

The modular transformation properties of these functions are easily 
computed by using Lemma \ref{n3:lemma:2022-111b} to obtain the following:

\vspace{1mm}
%(line=567)
%(label=n3:lemma:2022-108b)
\begin{lemma} \,\ 
\label{n3:lemma:2022-108b}
\begin{enumerate}
\item[{\rm 1)}] \,\ $S$-transformation \,\ :
\begin{enumerate}
\item[{\rm (i)}] \quad $\widetilde{A}^{[m]}_1
\Big(-\dfrac{1}{\tau}, \, \dfrac{z}{\tau}\Big)
\,\ = \,\ 
\tau \, e^{\frac{2\pi im}{\tau} \, 
(\frac{z}{2}-\frac{1}{4}+\frac{\tau}{4})(\frac{z}{2}+\frac{1}{4}-\frac{\tau}{4})} \, 
\widetilde{A}^{[m]}_2 (\tau, \, z)$
\item[{\rm (ii)}] \quad $\widetilde{A}^{[m]}_2
\Big(-\dfrac{1}{\tau}, \, \dfrac{z}{\tau}\Big)
\,\ = \,\ 
\tau \, e^{\frac{2\pi im}{\tau} \, 
(\frac{z}{2}-\frac{1}{4}-\frac{\tau}{4})(\frac{z}{2}+\frac{1}{4}+\frac{\tau}{4})} \, 
\widetilde{A}^{[m]}_1 (\tau, \, z)$
\item[{\rm (iii)}] \quad $\widetilde{A}^{[m]}_3
\Big(-\dfrac{1}{\tau}, \, \dfrac{z}{\tau}\Big)
\,\ = \,\ 
\tau \, e^{\frac{2\pi im}{\tau} \, 
(\frac{z}{2}-\frac{1}{4}+\frac{\tau}{2})(\frac{z}{2}+\frac{1}{4}-\frac{\tau}{2})} \, 
\widetilde{A}^{[m]}_6 (\tau, \, z)$
\item[{\rm (iv)}] \quad $\widetilde{A}^{[m]}_4
\Big(-\dfrac{1}{\tau}, \, \dfrac{z}{\tau}\Big)
\,\ = \,\ 
\tau \, e^{\frac{2\pi im}{\tau} \, 
(\frac{z}{2}-\frac{1}{4})(\frac{z}{2}+\frac{1}{4})} \, 
\widetilde{A}^{[m]}_5 (\tau, \, z)$
\item[{\rm (v)}] \quad $\widetilde{A}^{[m]}_5
\Big(-\dfrac{1}{\tau}, \, \dfrac{z}{\tau}\Big)
\,\ = \,\ 
\tau \, e^{\frac{2\pi im}{\tau} \, 
(\frac{z}{2}+\frac{\tau}{4})(\frac{z}{2}-\frac{\tau}{4})} \, 
\widetilde{A}^{[m]}_4 (\tau, \, z)$
\item[{\rm (vi)}] \quad $\widetilde{A}^{[m]}_6
\Big(-\dfrac{1}{\tau}, \, \dfrac{z}{\tau}\Big)
\,\ = \,\ 
\tau \, e^{\frac{2\pi im}{\tau} \, 
(\frac{z}{2}-\frac{1}{2}-\frac{\tau}{4})
(\frac{z}{2}+\frac{1}{2}+\frac{\tau}{4})} \, 
\widetilde{A}^{[m]}_3 (\tau, \, z)$
\end{enumerate}

\item[{\rm 2)}] \,\ $T$-transformation \,\ :
\begin{enumerate}
\item[{\rm (i)}] \quad $\widetilde{A}^{[m]}_1(\tau+1, \, z) \,\ = \,\ 
\widetilde{A}^{[m]}_3(\tau, \, z)$
\item[{\rm (ii)}] \quad $\widetilde{A}^{[m]}_2(\tau+1, \, z) \,\ = \,\ 
\widetilde{A}{}^{[m]}_4(\tau, \, z)$
\item[{\rm (iii)}] \quad $\widetilde{A}^{[m]}_3(\tau+1, \, z) \,\ = \,\ 
\widetilde{A}{}^{[m]}_2(\tau, \, z)$
\item[{\rm (iv)}] \quad $\widetilde{A}^{[m]}_4(\tau+1, \, z) \,\ = \,\ 
\widetilde{A}{}^{[m]}_1(\tau, \, z)$
\item[{\rm (v)}] \quad $\widetilde{A}^{[m]}_5(\tau+1, \, z) \,\ = \,\ 
\widetilde{A}{}^{[m]}_5(\tau, \, z)$
\item[{\rm (vi)}] \quad $\widetilde{A}^{[m]}_6(\tau+1, \, z) \,\ = \,\ 
e^{\pi im} \, \widetilde{A}^{[m]}_6(\tau, \, z)$
\end{enumerate}
\end{enumerate}
\end{lemma}

Define the functions $\overset{\circ}{A}{}^{[m]}_i(\tau, \, z)$ by
%(label=n3:eqn:2022-108b)
\begin{equation}
\overset{\circ}{A}{}^{[m]}_i(\tau, \, z) \,\ := \,\ \left\{
\begin{array}{rcl}
e^{- \, \frac{\pi im\tau}{8}} \, \widetilde{A}^{[m]}_i(\tau, \, z) & &
(1 \, \leq \, i \, \leq \, 4) \\[2mm]
\widetilde{A}^{[m]}_5(\tau, \, z) & & (i \, = \, 5) \\[2mm]
e^{- \, \frac{\pi i m\tau}{2}} \, \widetilde{A}^{[m]}_6(\tau, \, z) & & 
(i \, = \, 6)
\end{array} \right.
\label{n3:eqn:2022-108b}
\end{equation}
Then the transformation properties of these functions are obtained 
from Lemma \ref{n3:lemma:2022-108b} as follows:

%(line=626)
%(label=n3:lemma:2022-108c)
\begin{lemma} \,\ 
\label{n3:lemma:2022-108c}
\begin{enumerate}
\item[{\rm 1)}] \,\ $S$\text{-}transformation \,\ :
\begin{enumerate}
\item[{\rm (i)}] \quad $\overset{\circ}{A}{}^{[m]}_1
\Big(-\dfrac{1}{\tau}, \dfrac{z}{\tau}\Big) 
\,\ = \,\ 
e^{\frac{\pi im}{4}} \, \tau \, e^{\frac{\pi imz^2}{2\tau}} \, 
\overset{\circ}{A}{}^{[m]}_2(\tau,z)$
\item[{\rm (ii)}] \quad $\overset{\circ}{A}{}^{[m]}_2
\Big(-\dfrac{1}{\tau}, \dfrac{z}{\tau}\Big) 
\,\ = \,\ 
e^{-\frac{\pi im}{4}} \, \tau \, e^{\frac{\pi imz^2}{2\tau}} \, 
\overset{\circ}{A}{}^{[m]}_1(\tau,z)$
\item[{\rm (iii)}] \quad $\overset{\circ}{A}{}^{[m]}_3
\Big(-\dfrac{1}{\tau}, \dfrac{z}{\tau}\Big) 
\,\ = \,\ 
e^{\frac{\pi im}{2}} \, \tau \, e^{\frac{\pi imz^2}{2\tau}} \, 
\overset{\circ}{A}{}^{[m]}_6(\tau,z)$
\item[{\rm (iv)}] \quad $\overset{\circ}{A}{}^{[m]}_4
\Big(-\dfrac{1}{\tau}, \dfrac{z}{\tau}\Big) 
\,\ = \,\ 
\tau \, e^{\frac{\pi imz^2}{2\tau}} \, \overset{\circ}{A}{}^{[m]}_5(\tau,z)$
\item[{\rm (v)}] \quad $\overset{\circ}{A}{}^{[m]}_5
\Big(-\dfrac{1}{\tau}, \dfrac{z}{\tau}\Big) 
\,\ = \,\ 
\tau \, e^{\frac{\pi imz^2}{2\tau}} \, \overset{\circ}{A}{}^{[m]}_4(\tau,z)$
\item[{\rm (vi)}] \quad $\overset{\circ}{A}{}^{[m]}_6
\Big(-\dfrac{1}{\tau}, \dfrac{z}{\tau}\Big) 
\,\ = \,\ 
e^{-\frac{\pi im}{2}} \, \tau \, e^{\frac{\pi imz^2}{2\tau}} \, 
\overset{\circ}{A}{}^{[m]}_3(\tau,z)$
\end{enumerate}

\item[{\rm 2)}] \,\ $T$\text{-}transformation \,\ :
\begin{enumerate}
\item[{\rm (i)}] \quad $\overset{\circ}{A}{}^{[m]}_1(\tau+1, \, z) \,\ = \,\ 
e^{-\frac{\pi im}{8}} \, \overset{\circ}{A}{}^{[m]}_3(\tau, \, z)$
\item[{\rm (ii)}] \quad $\overset{\circ}{A}{}^{[m]}_2(\tau+1, \, z) \,\ = \,\ 
e^{-\frac{\pi im}{8}} \, \overset{\circ}{A}{}^{[m]}_4(\tau, \, z)$
\item[{\rm (iii)}] \quad $\overset{\circ}{A}{}^{[m]}_3(\tau+1, \, z) \,\ = \,\ 
e^{-\frac{\pi im}{8}} \, \overset{\circ}{A}{}^{[m]}_2(\tau, \, z)$
\item[{\rm (iv)}] \quad $\overset{\circ}{A}{}^{[m]}_4(\tau+1, \, z) \,\ = \,\ 
e^{-\frac{\pi im}{8}} \, \overset{\circ}{A}{}^{[m]}_1(\tau, \, z)$
\item[{\rm (v)}] \quad $\overset{\circ}{A}{}^{[m]}_5(\tau+1, \, z) \,\ = \,\ 
\overset{\circ}{A}{}^{[m]}_5(\tau, \, z)$
\item[{\rm (vi)}] \quad $\overset{\circ}{A}{}^{[m]}_6(\tau+1, \, z) \,\ = \,\ 
e^{\frac{\pi im}{2}} \, \overset{\circ}{A}{}^{[m]}_6(\tau, \, z)$
\end{enumerate}
\end{enumerate}
\end{lemma}

We note that, by Lemma \ref{n3:lemma:2022-111c}, these functions 
$\overset{\circ}{A}{}^{[m]}_j(\tau,z)$ are connected to 
$\widetilde{\Phi}^{[\frac{m}{2},s]}$ by the following relations:
%(label=n3:eqn:2022-117d1) $\sim$ (label=n3:eqn:2022-117d6)
%(line=681)
\begin{subequations}
{\allowdisplaybreaks
\begin{eqnarray}
& &
\widetilde{\Phi}^{[\frac{m}{2},0]}
\left(2\tau, z+\frac{\tau}{2}-\frac12,z-\frac{\tau}{2}+\frac12,\frac{\tau}{8}\right)
= \frac12 \big\{
\overset{\circ}{A}{}^{[m]}_2(\tau, z)+\overset{\circ}{A}{}^{[m]}_1(\tau, z)\big\}
\nonumber
\\[0mm]
& &
\label{n3:eqn:2022-117d1}
\\[0mm]
& &
\widetilde{\Phi}^{[\frac{m}{2},\frac12]}
\left(2\tau, z+\frac{\tau}{2}-\frac12,z-\frac{\tau}{2}+\frac12,\frac{\tau}{8}\right)
= \frac12 \big\{
\overset{\circ}{A}{}^{[m]}_2(\tau, z)-\overset{\circ}{A}{}^{[m]}_1(\tau, z)\big\}
\nonumber
\\[0mm]
& &
\label{n3:eqn:2022-117d2}
\\[0mm]
& &
\widetilde{\Phi}^{[\frac{m}{2},0]}
\left(2\tau, z+\frac{\tau}{2},z-\frac{\tau}{2},\frac{\tau}{8}\right)
= \frac12 \big\{
\overset{\circ}{A}{}^{[m]}_4(\tau, z)+\overset{\circ}{A}{}^{[m]}_3(\tau, z)\big\}
\label{n3:eqn:2022-117d3}
\\[1mm]
& &
\widetilde{\Phi}^{[\frac{m}{2},\frac12]}
\left(2\tau, z+\frac{\tau}{2},z-\frac{\tau}{2},\frac{\tau}{8}\right)
= \frac12 \big\{
\overset{\circ}{A}{}^{[m]}_4(\tau, z)-\overset{\circ}{A}{}^{[m]}_3(\tau, z)\big\}
\label{n3:eqn:2022-117d4}
\\[1mm]
& &
\widetilde{\Phi}^{[\frac{m}{2},0]}
\left(2\tau, z-\frac12,z+\frac12,0\right)
= \overset{\circ}{A}{}^{[m]}_5(\tau, z)
\label{n3:eqn:2022-117d5}
\\[1mm]
& &
\widetilde{\Phi}^{[\frac{m}{2},0]}
\left(2\tau, z+\tau-\frac12,z-\tau+\frac12,\frac{\tau}{2}\right)
= \frac12 \big\{1+(-1)^{m+2s}\big\} \, \overset{\circ}{A}{}^{[m]}_6(\tau, z)
\nonumber
\\[-1mm]
& &
\label{n3:eqn:2022-117d6}
\end{eqnarray}}
\end{subequations}

\section{Integrable $\widehat{osp}(3|2)$-modules and their characters}
\label{sec:B11}
%(line=765)

We consider the Dynkin diagram of $\widehat{B}(1,1) = \widehat{osp}(3|2)$ 
%$\widehat{B}(1,1) \,\ : $ \hspace{-3mm}
\setlength{\unitlength}{1mm}
\begin{picture}(36,9)
\put(6,1){\circle{3}}
\put(17,1){\circle{3}}
\put(28,1){\circle{3}}
\put(7.5,1.5){\vector(1,0){8.1}}
\put(7.5,0.5){\vector(1,0){8.1}}
\put(18.5,1.5){\vector(1,0){8.1}}
\put(18.5,0.5){\vector(1,0){8.1}}
\put(17,1){\makebox(0,0){$\times$}}
\put(6,5){\makebox(0,0){$\alpha_0$}}
\put(17,5){\makebox(0,0){$\alpha_1$}}
\put(28,5){\makebox(0,0){$\alpha_2$}}
%\put(11.5,-2.5){\makebox(0,0){$-1$}}
%\put(22.5,-2.5){\makebox(0,0){$\frac12$}}
\end{picture}
with the inner product $( \,\ | \,\ )$ such that 
$\Big((\alpha_i|\alpha_j)\Big)_{i,j=0,1,2} = \, 
\left(
\begin{array}{ccrc}
2 & -1 & 0 \\[0mm]
-1 & 0 & \frac12 \\[1mm]
0 & \frac12 & -\frac12
\end{array}\right)$. 
Then the dual Coxeter number of $\widehat{B}(1,1)$ is $h^{\vee}=\frac12$.
Let $\Lambda_0$ be the element in $\hhh^*$ satisfying the conditions
$(\Lambda_0|\alpha_j) =\delta_{j,0}$ and $(\Lambda_0|\Lambda_0)=0$.
Let $\delta=\alpha_0+2\alpha_1+2\alpha_2$ be the primitive imaginary root
and $\rho := \frac12 \Lambda_0-\frac12 \alpha_1$ be the Weyl vector.

We put 
%(label=n3:eqn:2022-110a)
%(label=n3:eqn:2022-110b)
\begin{subequations}
\begin{eqnarray}
K(m) &:=& -\dfrac{m+2}{4} \qquad \text{so that} \quad
m=-4\left(K(m)+\frac12\right)
\label{n3:eqn:2022-110a}
\\[1mm]
\Lambda^{[K(m), m_2]} &:=& K(m)\Lambda_0-\dfrac{m_2}{2} \alpha_1
\label{n3:eqn:2022-110b}
\end{eqnarray}
\end{subequations}
Note that the weight $\Lambda^{[K(m), m_2]}$ is atypical with respect to 
$\alpha_1$, \\
namely $(\Lambda^{[K(m), m_2]}+\rho|\alpha_1)=0$.

\vspace{1mm}
%(line=817)
%(label=n3:lemma:2022-108a)
\begin{lemma}
\label{n3:lemma:2022-108a}
The weight $\Lambda^{[K(m), m_2]}$ is integrable with respect to 
$\alpha_2$ and $\delta-\alpha_2$ if and only if 
$m$ and $m_2$ are non-negative integers satisfying $m_2 \leq m$.
\end{lemma}

In this paper,  a $\widehat{B}(1,1)$-module $L(\Lambda)$ which is integrable 
with respect to $\alpha_2$ and $\delta-\alpha_2$ is called simply 
an \lq \lq integrable" $\widehat{B}(1,1)$-module, and $\Lambda$ is called 
simply an \lq \lq integrable" weight. 

\medskip

Define the coordinates on the Cartan subalgebra $\hhh$ of $\widehat{B}(1,1)$ by 
\begin{equation}
(\tau, z_1, z_2,t) := 2\pi i \left\{-\tau \Lambda_0-
z_1(\alpha_1+2\alpha_2)-z_2\alpha_1+t\delta\right\}
\label{n3:eqn:2022-110g}
\end{equation}

For an integrable weight $\Lambda$ which is atypical with respect to 
$\alpha_1$, the supercharacter ${\rm ch}^{(-)}_{\Lambda}$ 
and the character ${\rm ch}^{(+)}_{\Lambda}$ of $L(\Lambda)$ 
are obtained by the formulas
%(label=n3:eqn:2022-110d)
\begin{subequations}
\begin{eqnarray}
& & \hspace{-15mm}
R^{(-)}{\rm ch}^{(-)}_{\Lambda} := 
\sum_{w \in \langle r_{\alpha_2}, r_{\delta-\alpha_2}\rangle }
\varepsilon(w) \, w 
\left(\frac{e^{\Lambda+\rho}}{1-e^{-\alpha_1}}\right)
\nonumber
\\[1mm]
&=&
\sum_{j \in \zzz}t_{4j\alpha_2}\left(\frac{e^{\Lambda+\rho}}{1-e^{-\alpha_1}}\right)
- 
\sum_{j \in \zzz}r_{\alpha_2}t_{4j\alpha_2}
\left(\frac{e^{\Lambda+\rho}}{1-e^{-\alpha_1}}\right)
\label{n3:eqn:2022-110c}
\\[1mm]
& &\hspace{-15mm}
(R^{(+)}{\rm ch}^{(+)}_{\Lambda})(\tau, z_1, z_2,t) :=
(R^{(-)}{\rm ch}^{(-)}_{\Lambda})\left(\tau, z_1-\frac12, z_2-\frac12,t\right)
\label{n3:eqn:2022-110d}
\end{eqnarray}
\end{subequations}

\noindent
where $t_{\alpha} \,\ (\alpha \in \hhh)$ is the linear automorphism of $\hhh$
defined, in \cite{K1},  by 
%(label=n3:eqn:2022-110h)
\begin{equation}
t_{\alpha}(\lambda) := \lambda + (\lambda|\delta)\alpha 
-\left\{
\frac{(\alpha|\alpha)}{2}(\lambda|\delta)+(\lambda|\alpha)\right\}\delta
\label{n3:eqn:2022-110h}
\end{equation}
and $R^{(\pm)}$ are (super)denominator of $\widehat{B}(1,1)$.

%(line=898)
%(label=n3:lemma:2022-108a)
\begin{lemma} 
\label{n3:lemma:2022-108a}
For $m \in \nnn$, the (super)character of integrable $\widehat{B}(1,1)$-module 
$L(\Lambda^{[K(m),m_2]})$ is given by the following formulas:

\begin{enumerate}
\item[{\rm 1)}] $\Big(R^{(+)} {\rm ch}^{(+)}_{\Lambda^{[K(m),m_2]}}\Big)
(\tau, z_1, z_2, t) $
$$= \,\ 
\Phi^{(+)[\frac{m}{2}, \, \frac{m_2+1}{2}]}\Big(
2\tau, \,\ z_1-\dfrac12, \,\ -z_2+\dfrac12, \,\ \dfrac{t}{2}\Big)
$$

\item[{\rm 2)}] $\Big(R^{(-)} {\rm ch}^{(-)}_{\Lambda^{[K(m),m_2]}}\Big)
(\tau, z_1, z_2, t) 
\,\ = \,\ 
\Phi^{(+)[\frac{m}{2}, \, \frac{m_2+1}{2}]}\Big(
2\tau, \, z_1, \, -z_2, \, \dfrac{t}{2}\Big)$
\end{enumerate}
\end{lemma}

In order to have $SL_2(\zzz)$-invariance of the space of characters, we need 
\lq \lq twisted" characters. For this sake, we define the 
character twisted by $\sigma_{j,k}:=t_{j\alpha_1+k (\alpha_1+2\alpha_2)}$ 
as follows:
%(label=n3:eqn:2022-110j)
\begin{equation}
(R^{(+){\rm tw}} 
{\rm ch}^{(+){\rm tw}(\sigma_{j,k})}_{\Lambda})(h) 
:= (R^{(+)} {\rm ch}^{(+)}_{\Lambda})(\sigma_{j,k}(h)) 
\qquad (h \in \hhh)
\label{n3:eqn:2022-110j}
\end{equation}
where $R^{(+){\rm tw}}$ is the twisted denominator, but 
we will not need here to know its explicit formula. The twisted numerator 
$(R^{(+){\rm tw}} {\rm ch}^{(+){\rm tw}(\sigma_{j,k})}_{\Lambda^{[K(m), m_2]}})(h)$
is computed easily as follows:

%(line=932)
%(label=n3:lemma:2022-111a)
\begin{lemma}
\label{n3:lemma:2022-112a}
$(R^{(+){\rm tw}} 
{\rm ch}^{(+){\rm tw}(\sigma_{j,k})}_{\Lambda^{[K(m), m_2]}})(\tau, z_1, z_2,t)$
$$
= \Phi^{[\frac{m}{2}, \frac{m_2+1}{2}]}\left(2\tau, 
z_1+k\tau-\frac12, -z_2-j\tau+\frac12, \frac12(t+jz_1+kz_2+jk\tau)\right)
$$
\end{lemma}

\begin{proof} By \eqref{n3:eqn:2022-110h} we have 
$$
\left\{
\begin{array}{lcl}
\sigma_{j,k}(\Lambda_0) &=& \Lambda_0+j\alpha_1+k(\alpha_1+2\alpha_2)-jk\delta
\\[1mm]
\sigma_{j,k}(\alpha_1) &=& \alpha_1-k\delta
\\[1mm]
\sigma_{j,k}(\alpha_1+2\alpha_2) &=&\alpha_1+2\alpha_2-j\delta .
\end{array}\right.
$$
Then we have
\begin{eqnarray*}
& &
\sigma_{j,k}(\tau, z_1, z_2, t) 
=2\pi i \left\{-\tau \sigma_{j,k}(\Lambda_0)
-z_1\sigma_{j,k}(\alpha_1+2\alpha_2)
-z_2\sigma_{j,k}(\alpha_1)+t\delta\right\}
\\[2mm]
&=&
2\pi i\left\{
-\tau \Lambda_0
-(z_1+k\tau)(\alpha_1+2\alpha_2)
-(z_2+j\tau)\alpha_1
+(t+jz_1+kz_2+jk\tau)\delta\right\}
\\[2mm]
&=&
(\tau, \,\ z_1+k\tau, \,\ z_2+j\tau, \,\ t+jz_1+kz_2+jk\tau),
\end{eqnarray*}
so
\begin{eqnarray*}
& &
(R^{(+){\rm tw}} 
{\rm ch}^{(+){\rm tw}(\sigma_{j,k})}_{\Lambda^{[K(m), m_2]}})(\tau, z_1, z_2,t)
\\[1mm]
&=& (R^{(+)} {\rm ch}^{(+)}_{\Lambda^{[K(m), m_2]}})
(\tau, z_1+k\tau, z_2+j\tau, t+jz_1+kz_2+jk\tau)
\\[1mm]
&=&
\Phi^{[\frac{m}{2}, \frac{m_2+1}{2}]}\left(2\tau, 
z_1+k\tau-\frac12, -z_2-j\tau+\frac12, \frac12(t+jz_1+kz_2+jk\tau)\right)
\end{eqnarray*}
by Lemma \ref{n3:lemma:2022-108a}, proving the claim.
\end{proof}

\section{Characters of quantum reduction}
\label{sec:quantum}
%(line=962)

We now consider the quantum Hamiltonian reduction associated to the pair 
$(x=\frac12 \theta, \, f=e_{-\theta})$, 
where $\theta := 2(\alpha_1+\alpha_2)$ 
is the highest root of the finite-dimensional Lie superalgebra 
$\overline{\ggg} := osp(3|2)$. The Cartan subalgebra of $\overline{\ggg}$ is 
$\overline{\hhh}:= \ccc \alpha_1 \oplus \ccc \alpha_2$. 
Taking a basis $J_0 := -2 \alpha_2$ of $\overline{\hhh}^f$, we have  
%(label=n3:eqn:2022-111f)
\begin{eqnarray}
& & 
2\pi i \left\{-\tau\Lambda_0-\tau x+zJ_0 + \frac{\tau}{2}(x|x) \delta\right\}
\nonumber
\\[0mm]
&=& 
2\pi i \left\{-\tau \Lambda_0-\left(z+\frac{\tau}{2}\right)(\alpha_1+2\alpha_2)
-\left(-z+\frac{\tau}{2}\right)\alpha_1 + \frac{\tau}{4}\delta
\right\}
\nonumber
\\[0mm]
&=&
\left(\tau, \, z+\frac{\tau}{2}, \, -z+\frac{\tau}{2}, \, \frac{\tau}{4}\right)
\label{n3:eqn:2022-111f}
\end{eqnarray}
Then the (super)character of the quantum Hamiltonian reduction $H(\Lambda)$
of $\widehat{B}(1,1)$ module $L(\Lambda)$ is obtained by the formula :
%(label=n3:eqn:2022-111g)
\begin{eqnarray}
& &
(\overset{N=3}{R}{}^{(\pm)} {\rm ch}^{(\pm)}_{H(\Lambda)})(\tau, z) =
(R^{(\pm)} {\rm ch}^{(\pm)}_{\Lambda})
\left(
2\pi i \left\{-\tau\Lambda_0-\tau x+zJ_0 + \frac{\tau}{2}(x|x) \delta\right\}
\right)
\nonumber
\\[1mm]
& & \hspace{30mm}
= \,\ 
(R^{(\pm)} {\rm ch}^{(\pm)}_{\Lambda})
\left(\tau, \, z+\frac{\tau}{2}, \, -z+\frac{\tau}{2}, \, \frac{\tau}{4}\right)
\label{n3:eqn:2022-111g}
\end{eqnarray}
And also similar for the twisted characters, where 
$\overset{N=3}{R}{}^{(+)}$ and $\overset{N=3}{R}{}^{(-)}$ and 
$\overset{N=3}{R}{}^{(+){\rm tw}}$ are the denominator 
and superdenominator and twisted denominator respectively of the N=3 
superconformal algebra defined, by using the Mumford's theta functions 
$\vartheta_{ab}(\tau,z)$  $(a,b=0,1)$, as follows:
%(label=n3:eqn:2022-115a)
\begin{equation} \left\{
\begin{array}{lcl}
\overset{N=3}{R}{}^{(+)}(\tau,z) 
&:=& 
\eta(\frac{\tau}{2}) \eta(2\tau) \, 
\dfrac{\vartheta_{11}(\tau,z)}{\vartheta_{00}(\tau,z)}
\\[3mm]
\overset{N=3}{R}{}^{(-)}(\tau,z) 
&:=& 
\dfrac{\eta(\tau)^3}{\eta(\frac{\tau}{2})} \cdot 
\dfrac{\vartheta_{11}(\tau,z)}{\vartheta_{01}(\tau,z)}
\\[3mm]
\overset{N=3}{R}{}^{(+){\rm tw}}(\tau,z) 
&:=& \dfrac{1}{\sqrt{2}} \cdot
\dfrac{\eta(\tau)^3}{\eta(2\tau)} \cdot 
\dfrac{\vartheta_{11}(\tau,z)}{\vartheta_{10}(\tau,z)}
\end{array}\right.
\label{n3:eqn:2022-115a}
\end{equation}

Then, by \eqref{n3:eqn:2022-111g} and Lemma \ref{n3:lemma:2022-108a}
and Lemma \ref{n3:lemma:2022-112a}, we obtain the following : 

%(line=1071)
%(label=n3:prop:2022-111a)
\begin{prop}
\label{n3:prop:2022-111a}
For $m \in \nnn$ and $m_2 \in \zzz_{\geq 0}$ such that 
$m_2 \leq m$, the (super or twisted) characters of the N=3 module 
$H(\Lambda^{[K(m), m_2]})$ are as follows:

\begin{enumerate}
\item[{\rm 1)}] $(\overset{N=3}{R}{}^{(+)} 
{\rm ch}^{(+)}_{H(\Lambda^{[K(m), m_2]})})
(\tau, z) =
 \Phi^{[\frac{m}{2}, \frac{m_2+1}{2}]}
\Big(2\tau, z+\dfrac{\tau}{2}-\dfrac12, 
z-\dfrac{\tau}{2}+\dfrac12, \dfrac{\tau}{8}\Big)$
\item[{\rm 2)}] $(\overset{N=3}{R}{}^{(-)} 
{\rm ch}^{(-)}_{H(\Lambda^{[K(m), m_2]})})
(\tau, z) =
\Phi^{[\frac{m}{2}, \frac{m_2+1}{2}]}
\Big(2\tau, z+\dfrac{\tau}{2}, z-\dfrac{\tau}{2}, \dfrac{\tau}{8}\Big)$
\item[{\rm 3)}] $(\overset{N=3}{R}{}^{(+){\rm tw}} 
{\rm ch}^{(+){\rm tw}(\sigma_{j,k})}_{H(\Lambda^{[K(m), m_2]})})(\tau, z) $
\begin{eqnarray*}
&=&
\Phi^{[\frac{m}{2}, \frac{m_2+1}{2}]}\left(2\tau, 
z+\left(k+\frac12\right)\tau-\frac12,
z-\left(j+\frac12\right)\tau+\frac12, \right.
\\[1mm]
& & \hspace{20mm} \left.
\frac{(j-k)z}{2}+\frac12 \left(j+\frac12\right)\left(k+\frac12\right)\tau
\right)
\end{eqnarray*}
\item[{\rm 4)}] Let $\sigma_{\pm}$ be the linear automorphisms of $\hhh$ 
defined by 
$$
\left\{
\begin{array}{lclcl}
\sigma_+ &:=& \sigma_{\frac12, \frac12} &=& t_{\alpha_1+\alpha_2}
\\[2mm]
\sigma_- &:=& \sigma_{-\frac12, -\frac12} &=& t_{-(\alpha_1+\alpha_2)}
\end{array}\right. \, .
$$ 
Then
\begin{enumerate}
\item[{\rm (i)}] $(\overset{N=3}{R}{}^{(+){\rm tw}} 
{\rm ch}^{(+){\rm tw}(\sigma_+)}_{H(\Lambda^{[K(m), m_2]})})(\tau, z)$ 
$$
= \Phi^{[\frac{m}{2}, \frac{m_2+1}{2}]}
\Big(2\tau, z+\tau-\dfrac12, z-\tau+\dfrac12, \dfrac{\tau}{2}\Big)
$$
\item[{\rm (ii)}] $(\overset{N=3}{R}{}^{(+){\rm tw}(} 
{\rm ch}^{(+){\rm tw}(\sigma_-)}_{H(\Lambda^{[K(m), m_2]})})
(\tau, z) =
\Phi^{[\frac{m}{2}, \frac{m_2+1}{2}]}
\Big(2\tau, z-\dfrac12, z+\dfrac12, 0\Big)$
\end{enumerate}
\end{enumerate}
\end{prop}

\begin{proof} 1) and 2) follow from \eqref{n3:eqn:2022-111g}
and Lemma \ref{n3:lemma:2022-108a}. 3) is shown as follows.
By the twisted version of \eqref{n3:eqn:2022-111g}, we have
%(label=n3:eqn:2022-112b)
\begin{equation}
(\overset{N=3}{R}{}^{(+){\rm tw}} 
{\rm ch}^{(+){\rm tw}(\sigma_{j,k})}_{H(\Lambda^{[K(m), m_2]})})(\tau, z) 
=
(R^{(+){\rm tw}}{\rm ch}^{(+){\rm tw}(\sigma_{j,k})}_{\Lambda^{[K(m), m_2]}})
\left(\tau, z+\frac{\tau}{2}, -z+\frac{\tau}{2}, \frac{\tau}{4}\right) 
\label{n3:eqn:2022-112b}
\end{equation}
The RHS of this equation \eqref{n3:eqn:2022-112b} is rewitten, 
by using Lemma \ref{n3:lemma:2022-112a}, as follows : 
\begin{eqnarray*}
& & \text{RHS of \eqref{n3:eqn:2022-112b}}
\\[0mm]
&=&
\Phi^{[\frac{m}{2}, \frac{m_2+1}{2}]}\left(2\tau, \,\ 
\left(z+\frac{\tau}{2}\right)+k\tau-\frac12, \,\ 
-\left(-z+\frac{\tau}{2}\right)-j\tau+\frac12, 
\right.
\\[0mm]
& & \hspace{30mm} \left.
\frac12 \left(\frac{\tau}{4}+j\left(z+\frac{\tau}{2}\right)
+k\left(-z+\frac{\tau}{2}\right)+jk\tau\right)\right)
\\[0mm]
&=&
\Phi^{[\frac{m}{2}, \frac{m_2+1}{2}]}\left(2\tau, \,\ 
z+\left(k+\frac12\right)\tau-\frac12, \,\ 
z-\left(j+\frac12\right)\tau+\frac12, \right.
\\[1mm]
& & \hspace{30mm} \left.
\frac{(j-k)z}{2}+\frac12 \left(j+\frac12\right)\left(k+\frac12\right)\tau
\right) \, ,
\end{eqnarray*}
proveng 3).  4) follows from 3) immediately.
\end{proof}

We call these characters \lq \lq honest" characters. 
\lq \lq Modified" characters $\widetilde{\rm ch}_{H(\Lambda)}$ are 
defined by replacing $\Phi$ with $\widetilde{\Phi}$ in the formulas 
in Proposition \ref{n3:prop:2022-111a}, namely
%(label=n3:eqn:2022-112c)
{\allowdisplaybreaks
\begin{eqnarray}
& &
(\overset{N=3}{R}{}^{(+)} 
\widetilde{\rm ch}^{(+)}_{H(\Lambda^{[K(m), m_2]})})
(\tau, z) :=
\widetilde{\Phi}^{[\frac{m}{2}, \frac{m_2+1}{2}]}
\left(2\tau, z+\frac{\tau}{2}-\frac12, z-\frac{\tau}{2}+\frac12, 
\frac{\tau}{8}\right)
\nonumber
\\[1mm]
& &
(\overset{N=3}{R}{}^{(-)} 
\widetilde{\rm ch}^{(-)}_{H(\Lambda^{[K(m), m_2]})})
(\tau, z) :=
\widetilde{\Phi}^{[\frac{m}{2}, \frac{m_2+1}{2}]}
\left(2\tau, z+\frac{\tau}{2}, z-\frac{\tau}{2}, \frac{\tau}{8}\right)
\nonumber
\\[1mm]
& &
(\overset{N=3}{R}{}^{(+){\rm tw}} 
\widetilde{\rm ch}^{(+){\rm tw}(\sigma_+)}_{H(\Lambda^{[K(m), m_2]})})
(\tau, z) :=
\widetilde{\Phi}^{[\frac{m}{2}, \frac{m_2+1}{2}]}
\left(2\tau, z+\tau-\frac12, z-\tau+\frac12, \frac{\tau}{2}\right)
\nonumber
\\[1mm]
& &
(\overset{N=3}{R}{}^{(+){\rm tw}(} 
\widetilde{\rm ch}^{(+){\rm tw}(\sigma_-)}_{H(\Lambda^{[K(m), m_2]})})
(\tau, z) :=
\widetilde{\Phi}^{[\frac{m}{2}, \frac{m_2+1}{2}]}
\left(2\tau, z-\frac12, z+\frac12, 0\right)
\label{n3:eqn:2022-112c}
\end{eqnarray}}

By Lemma \ref{n3:lemma:2022-111a} and Lemma \ref{n3:lemma:2022-111c} and 
the formula \eqref{n3:eqn:2022-111e}, these formulas are rewritten 
as follows :

%(line=1260)
%(label=n3:prop:2022-112a)
\begin{prop} \,\ 
\label{n3:prop:2022-112a}
\begin{enumerate}
\item[{\rm 1)}] $(\overset{N=3}{R}{}^{(+)} 
\widetilde{\rm ch}^{(+)}_{H(\Lambda^{[K(m), m_2]})})(\tau, z) 
= \dfrac12 \left\{
\overset{\circ}{A}{}^{[m]}_2(\tau,z)
+e^{\pi i(m_2+1)} \overset{\circ}{A}{}^{[m]}_1(\tau,z)\right\}$

\item[{\rm 2)}] $(\overset{N=3}{R}{}^{(-)} 
\widetilde{\rm ch}^{(-)}_{H(\Lambda^{[K(m), m_2]})})(\tau, z) 
= \dfrac12 \left\{
\overset{\circ}{A}{}^{[m]}_4(\tau,z)
+e^{\pi i(m_2+1)} \overset{\circ}{A}{}^{[m]}_3(\tau,z)\right\}$

\item[{\rm 3)}] $(\overset{N=3}{R}{}^{(+){\rm tw}} 
\widetilde{\rm ch}^{(+){\rm tw}(\sigma_+)}_{H(\Lambda^{[K(m), m_2]})})
(\tau, z) 
= \dfrac12 \left\{1+e^{\pi i(m+m_2+1)}\right\} 
\overset{\circ}{A}{}^{[m]}_6(\tau,z)$

\item[{\rm 4)}] $(\overset{N=3}{R}{}^{(+){\rm tw}} 
\widetilde{\rm ch}^{(+){\rm tw}(\sigma_-)}_{H(\Lambda^{[K(m), m_2]})})
(\tau, z) 
= \dfrac12 \left\{1+e^{\pi i(m_2+1)}\right\} 
\overset{\circ}{A}{}^{[m]}_5(\tau,z)$
\end{enumerate}
\end{prop}

\medskip

From these formulas we see that 
%(label=n3:eqn:2022-112d) 
%(label=n3:eqn:2022-112d2) 
%(label=n3:eqn:2022-112d3) 
\begin{equation}
\begin{array}{lcl}
\widetilde{\rm ch}^{(+){\rm tw}(\sigma_+)}_{H(\Lambda^{[K(m), m_2]})})
(\tau,z)
&=& 0 \qquad \text{if} \quad m+m_2 \in 2\zzz
\\[3mm]
\widetilde{\rm ch}^{(+){\rm tw}(\sigma_-)}_{H(\Lambda^{[K(m), m_2]})})
(\tau,z)
&=& 0 \qquad \text{if} \quad m_2 \in 2\zzz
\end{array}
\label{n3:eqn:2022-112d}
\end{equation}
Then, putting 
\begin{equation}
p(m) \,\ := \,\ \left\{
\begin{array}{ccl}
1 & & {\rm if} \,\ m \in 2 \zzz \\[1mm]
0 & & {\rm if} \,\ m \in \zzz_{\rm odd}
\end{array}\right. \,\ ,
\label{n3:eqn:2022-112d2}
\end{equation}
we have
\begin{equation}\left\{
\begin{array}{lcl}
(\overset{N=3}{R}{}^{(+){\rm tw}} 
\widetilde{\rm ch}^{(+){\rm tw}(\sigma_+)}_{H(\Lambda^{[K(m), p(m)]})})
(\tau, z) &=&
\overset{\circ}{A}{}^{[m]}_6(\tau,z)
\\[2mm]
(\overset{N=3}{R}{}^{(+){\rm tw}} 
\widetilde{\rm ch}^{(+){\rm tw}(\sigma_-)}_{H(\Lambda^{[K(m), 1]})})
(\tau, z) &=&
\overset{\circ}{A}{}^{[m]}_5(\tau,z)
\end{array} \right. \,\ .
\label{n3:eqn:2022-112d3}
\end{equation}
Note also that
%(label=n3:eqn:2022-112e)
\begin{equation}\left\{
\begin{array}{lcc}
\widetilde{\rm ch}^{(\pm)}_{H(\Lambda^{[K(m), m_2]})}(\tau, z) 
&=&\widetilde{\rm ch}^{(\pm)}_{H(\Lambda^{[K(m), m'_2]})}(\tau, z)
\\[3mm]
\widetilde{\rm ch}^{(+){\rm tw}(\sigma_{\pm})}_{H(\Lambda^{[K(m), m_2]})}
(\tau,z) &=&
\widetilde{\rm ch}^{(+){\rm tw}(\sigma_{\pm})}_{H(\Lambda^{[K(m), m'_2]})}
(\tau,z)
\end{array}\right. \,\  \text{if} \,\ m_2-m'_2 \in 2\zzz \, .
\label{n3:eqn:2022-112e}
\end{equation}
So, among modified characters, we need to consider only the following 6 characters:
%(label=n3:eqn:2022-112f)
\begin{equation}\left\{
\begin{array}{l}
\widetilde{\rm ch}^{(\pm)}_{H(\Lambda^{[K(m), m_2]})}(\tau, z) \qquad 
m_2 \in \{0, 1\} \\[2mm]
\widetilde{\rm ch}^{(+){\rm tw}(\sigma_+)}_{H(\Lambda^{[K(m), p(m)]})}
(\tau,z), \quad 
\widetilde{\rm ch}^{(+){\rm tw}(\sigma_-)}_{H(\Lambda^{[K(m), 1]})}
(\tau,z)
\end{array}\right.
\label{n3:eqn:2022-112f}
\end{equation}
Then, by Proposition \ref{n3:prop:2022-112a}, we have

%(line=1366)
%(label=n3:cor:2022-112a)
\begin{cor} \,\ 
\label{n3:cor:2022-112a}
\begin{enumerate}
\item[{\rm 1)}] $(\overset{N=3}{R}{}^{(+)} 
\widetilde{\rm ch}^{(+)}_{H(\Lambda^{[K(m), 0]})})(\tau, z) 
= \dfrac12 \left\{
\overset{\circ}{A}{}^{[m]}_2(\tau,z)
- \overset{\circ}{A}{}^{[m]}_1(\tau,z)\right\}$
\item[{\rm 2)}] $(\overset{N=3}{R}{}^{(+)} 
\widetilde{\rm ch}^{(+)}_{H(\Lambda^{[K(m), 1]})})(\tau, z) 
= \dfrac12 \left\{
\overset{\circ}{A}{}^{[m]}_2(\tau,z)
+ \overset{\circ}{A}{}^{[m]}_1(\tau,z)\right\}$
\item[{\rm 3)}] $(\overset{N=3}{R}{}^{(-)} 
\widetilde{\rm ch}^{(-)}_{H(\Lambda^{[K(m), 0]})})(\tau, z) 
= 
\dfrac12 \left\{
\overset{\circ}{A}{}^{[m]}_4(\tau,z)
- \overset{\circ}{A}{}^{[m]}_3(\tau,z)\right\}$
\item[{\rm 4)}] $(\overset{N=3}{R}{}^{(-)} 
\widetilde{\rm ch}^{(-)}_{H(\Lambda^{[K(m), 1]})})(\tau, z) 
= 
\dfrac12 \left\{
\overset{\circ}{A}{}^{[m]}_4(\tau,z)
+ \overset{\circ}{A}{}^{[m]}_3(\tau,z)\right\}$
\item[{\rm 5)}] $(\overset{N=3}{R}{}^{(+){\rm tw}} 
\widetilde{\rm ch}^{(+){\rm tw}(\sigma_+)}_{H(\Lambda^{[K(m), p(m)]})})
(\tau, z) 
= 
\overset{\circ}{A}{}^{[m]}_6(\tau,z)$
\item[{\rm 6)}] $(\overset{N=3}{R}{}^{(+){\rm tw}} 
\widetilde{\rm ch}^{(+){\rm tw}(\sigma_-)}_{H(\Lambda^{[K(m), 1]})})
(\tau, z) = \overset{\circ}{A}{}^{[m]}_5(\tau,z)$
\end{enumerate}
\end{cor}

\medskip

The modular transformation properties of these functions are obtained 
by Lemma \ref{n3:lemma:2022-108c} as follows:

%(line=1412)
%(label=n3:lemma:2022-115a)
\begin{lemma} \,\ 
\label{n3:lemma:2022-115a}
\begin{enumerate}
\item[{\rm 1)}] \,\ $S$-transformation :
\begin{enumerate}
\item[{\rm (i)}] \,\ $\big(\overset{N=3}{R}{}^{(+)} 
\widetilde{\rm ch}^{(+)}_{H(\Lambda^{[K(m), 0]})}\big)
\Big(-\dfrac{1}{\tau}, \dfrac{z}{\tau}\Big) 
\,\ = \,\ 
\dfrac{\tau}{2} \, e^{-\frac{\pi im}{4}}e^{\frac{\pi imz^2}{2\tau}} \, 
\Big\{$
$$\hspace{-10mm}
(1-e^{\frac{\pi im}{2}})
\big(\overset{N=3}{R}{}^{(+)} 
\widetilde{\rm ch}^{(+)}_{H(\Lambda^{[K(m), 1]})}\big)(\tau,z)
-
(e^{\frac{\pi im}{2}}+1)
\big(\overset{N=3}{R}{}^{(+)} 
\widetilde{\rm ch}^{(+)}_{H(\Lambda^{[K(m), 0]})}\big)(\tau,z)
\Big\}
$$
\item[{\rm (ii)}] \,\ $\big(\overset{N=3}{R}{}^{(+)} 
\widetilde{\rm ch}^{(+)}_{H(\Lambda^{[K(m), 1]})}\big)
\Big(-\dfrac{1}{\tau}, \dfrac{z}{\tau}\Big) 
\,\ = \,\ 
\dfrac{\tau}{2} \, e^{-\frac{\pi im}{4}} e^{\frac{\pi imz^2}{2\tau}} \, 
\Big\{ $
$$\hspace{-10mm}
(1+e^{\frac{\pi im}{2}})
\big(\overset{N=3}{R}{}^{(+)} 
\widetilde{\rm ch}^{(+)}_{H(\Lambda^{[K(m), 1]})}\big)(\tau,z)
+
(e^{\frac{\pi im}{2}}-1)
\big(\overset{N=3}{R}{}^{(+)} 
\widetilde{\rm ch}^{(+)}_{H(\Lambda^{[K(m), 0]})}\big)(\tau,z)
\Big\}
$$
\item[{\rm (iii)}] \,\ $\big(\overset{N=3}{R}{}^{(-)} 
\widetilde{\rm ch}^{(-)}_{H(\Lambda^{[K(m), 0]})}\big)
\Big(-\dfrac{1}{\tau}, \dfrac{z}{\tau}\Big) 
\,\ = \,\ 
\dfrac{\tau}{2} \, e^{\frac{\pi imz^2}{2\tau}} \, \Big\{ $
$$
\big(\overset{N=3}{R}{}^{(+){\rm tw}} 
\widetilde{\rm ch}^{(+){\rm tw}(\sigma_-)}_{H(\Lambda^{[K(m), 1]})}\big)
(\tau, z)
\, - \, e^{\frac{\pi im}{2}}
\big(\overset{N=3}{R}{}^{(+){\rm tw}} 
\widetilde{\rm ch}^{(+){\rm tw}(\sigma_+)}_{H(\Lambda^{[K(m), 1]})}\big)
(\tau, z)
\Big\}
$$
\item[{\rm (iv)}] \,\ $\big(\overset{N=3}{R}{}^{(-)} 
\widetilde{\rm ch}^{(-)}_{H(\Lambda^{[K(m), 1]})}\big)
\Big(-\dfrac{1}{\tau}, \dfrac{z}{\tau}\Big) 
\,\ = \,\ 
\dfrac{\tau}{2} \, e^{\frac{\pi imz^2}{2\tau}} \, \Big\{ $
$$
\big(\overset{N=3}{R}{}^{(+){\rm tw}} 
\widetilde{\rm ch}^{(+){\rm tw}(\sigma_-)}_{H(\Lambda^{[K(m), 1]})}\big)
(\tau, z)
\, + \, e^{\frac{\pi im}{2}}
\big(\overset{N=3}{R}{}^{(+){\rm tw}} 
\widetilde{\rm ch}^{(+){\rm tw}(\sigma_+)}_{H(\Lambda^{[K(m), 1]})}\big)
(\tau, z)
\Big\}
$$
\item[{\rm (v)}] \,\ $\big(\overset{N=3}{R}{}^{(+){\rm tw}} 
\widetilde{\rm ch}^{(+){\rm tw}(\sigma_+)}_{H(\Lambda^{[K(m), p(m)]})}\big)
\Big(-\dfrac{1}{\tau}, \dfrac{z}{\tau}\Big)$
$$
= \,\ 
\tau \, e^{\frac{\pi imz^2}{2\tau}} \, \Big\{
\big(\overset{N=3}{R}{}^{(-)} 
\widetilde{\rm ch}^{(-)}_{H(\Lambda^{[K(m), 1]})}\big)(\tau,z)
-
\big(\overset{N=3}{R}{}^{(-)} 
\widetilde{\rm ch}^{(-)}_{H(\Lambda^{[K(m), 0]})}\big)(\tau,z)
\Big\}
$$
\item[{\rm (vi)}] \,\ $\big(\overset{N=3}{R}{}^{(+){\rm tw}} 
\widetilde{\rm ch}^{(+){\rm tw}(\sigma_-)}_{H(\Lambda^{[K(m), 1]})}\big)
\Big(-\dfrac{1}{\tau}, \dfrac{z}{\tau}\Big)$
$$
= \,\ 
\tau \, e^{\frac{\pi imz^2}{2\tau}} \, \Big\{
\big(\overset{N=3}{R}{}^{(-)} 
\widetilde{\rm ch}^{(-)}_{H(\Lambda^{[K(m), 1]})}\big)(\tau,z)
+
\big(\overset{N=3}{R}{}^{(-)} 
\widetilde{\rm ch}^{(-)}_{H(\Lambda^{[K(m), 0]})}\big)(\tau,z)
\Big\}
$$
\end{enumerate}

\item[{\rm 2)}] \,\ $T$-transformation :
\begin{enumerate}
\item[{\rm (i)}] \,\ $\big(\overset{N=3}{R}{}^{(+)} 
\widetilde{\rm ch}^{(+)}_{H(\Lambda^{[K(m), 0]})}\big)(\tau+1,z)
\,\ = \,\ 
e^{-\frac{\pi im}{8}} \, 
(\overset{N=3}{R}{}^{(-)} 
\widetilde{\rm ch}^{(-)}_{H(\Lambda^{[K(m), 0]})})(\tau,z)$
\item[{\rm (ii)}] \,\ $\big(\overset{N=3}{R}{}^{(+)} 
\widetilde{\rm ch}^{(+)}_{H(\Lambda^{[K(m), 1]})}\big)(\tau+1,z)
\,\ = \,\ 
e^{-\frac{\pi im}{8}} \, 
(\overset{N=3}{R}{}^{(-)} 
\widetilde{\rm ch}^{(-)}_{H(\Lambda^{[K(m), 1]})})(\tau,z)$
\item[{\rm (iii)}] \,\ $\big(\overset{N=3}{R}{}^{(-)} 
\widetilde{\rm ch}^{(-)}_{H(\Lambda^{[K(m), 0]})}\big)(\tau+1,z)
\,\ = \,\ 
- \, e^{-\frac{\pi im}{8}} \, 
(\overset{N=3}{R}{}^{(+)} 
\widetilde{\rm ch}^{(+)}_{H(\Lambda^{[K(m), 0]})})(\tau,z)$
\item[{\rm (iv)}] \,\ $\big(\overset{N=3}{R}{}^{(-)} 
\widetilde{\rm ch}^{(-)}_{H(\Lambda^{[K(m), 1]})}\big)(\tau+1,z)
\,\ = \,\ 
e^{-\frac{\pi im}{8}} \, 
(\overset{N=3}{R}{}^{(+)} 
\widetilde{\rm ch}^{(+)}_{H(\Lambda^{[K(m), 1]})})(\tau,z)$
\item[{\rm (v)}] \,\ $\big(\overset{N=3}{R}{}^{(+){\rm tw}} 
\widetilde{\rm ch}^{(+){\rm tw}(\sigma_+)}_{H(\Lambda^{[K(m), p(m)]})}\big)
(\tau+1,z)$
$$
= \,\ 
e^{\frac{\pi im}{2}} \, 
\big(\overset{N=3}{R}{}^{(+){\rm tw}} 
\widetilde{\rm ch}^{(+){\rm tw}(\sigma_+)}_{H(\Lambda^{[K(m), p(m)]})}\big)
(\tau,z)
$$
\item[{\rm (vi)}] \,\ $\big(\overset{N=3}{R}{}^{(+){\rm tw}} 
\widetilde{\rm ch}^{(+){\rm tw}(\sigma_-)}_{H(\Lambda^{[K(m), 1]})}\big)
(\tau+1,z)
\,\ = \,\ 
\big(\overset{N=3}{R}{}^{(+){\rm tw}} 
\widetilde{\rm ch}^{(+){\rm tw}(\sigma_-)}_{H(\Lambda^{[K(m), 1]})}\big)
(\tau,z)$
\end{enumerate}
\end{enumerate}
\end{lemma}

\medskip

We note that the denominators of the N=3 SCA defined by 
\eqref{n3:eqn:2022-115a} satisfy the following modular 
transformation properties:
%(label=n3:eqn:2022-115b)
%(label=n3:eqn:2022-115c)
\begin{subequations}
{\allowdisplaybreaks
\begin{eqnarray}
& & \left\{
\begin{array}{lcl}
\overset{N=3}{R}{}^{(+)}\Big(-\dfrac{1}{\tau}, \dfrac{z}{\tau}\Big) &=& 
- \, \tau \, \overset{N=3}{R}{}^{(+)}(\tau, z)
\\[3mm]
\overset{N=3}{R}{}^{(-)}\Big(-\dfrac{1}{\tau}, \dfrac{z}{\tau}\Big) 
&=& 
- \, \tau \, \overset{N=3}{R}{}^{(+){\rm tw}}(\tau, z)
\\[3mm]
\overset{N=3}{R}{}^{(+){\rm tw}}\Big(-\dfrac{1}{\tau}, \dfrac{z}{\tau}\Big) 
&=&
- \, \tau \, \overset{N=3}{R}{}^{(-)}(\tau, z)
\end{array}\right.
\label{n3:eqn:2022-115b}
\\[1mm]
& & \left\{
\begin{array}{lcl}
\overset{N=3}{R}{}^{(+)}(\tau+1, \, z) 
&=& 
e^{\frac{11\pi i}{24}} \, \overset{N=3}{R}{}^{(-)}(\tau, z)
\\[1mm]
\overset{N=3}{R}{}^{(-)}(\tau+1, \, z)
&=& 
e^{\frac{11\pi i}{24}} \, \overset{N=3}{R}{}^{(+)}(\tau,z)
\\[1mm]
\overset{N=3}{R}{}^{(+){\rm tw}}(\tau+1, \, z)
&=& 
e^{\frac{\pi i}{12}} \, \overset{N=3}{R}{}^{(+){\rm tw}}(\tau,z)
\end{array}\right.
\label{n3:eqn:2022-115c}
\end{eqnarray}}
\end{subequations}
Then by these formulas \eqref{n3:eqn:2022-115b} and \eqref{n3:eqn:2022-115c}
and by Lemma \ref{n3:lemma:2022-115a}, we obtain the modular transformation 
properties of the characters as follows : 

%(line=1512)
%(label=n3:prop:2022-115a)
\begin{prop} \,\ 
\label{n3:prop:2022-115a}
\begin{enumerate}
\item[{\rm 1)}] \,\ $S$-transformation  :
\begin{enumerate}
\item[{\rm (i)}] \,\ $
\widetilde{\rm ch}^{(+)}_{H(\Lambda^{[K(m), 0]})}
\Big(-\dfrac{1}{\tau}, \dfrac{z}{\tau}\Big)
\,\ = \,\ 
- \dfrac12 \, e^{-\frac{\pi im}{4}} \, e^{\frac{\pi imz^2}{2\tau}} 
\, \Big\{ $
$$
(1-e^{\frac{\pi im}{2}}) \, 
\widetilde{\rm ch}^{(+)}_{H(\Lambda^{[K(m), 1]})}(\tau,z)
\, - \, 
(e^{\frac{\pi im}{2}}+1) \,  \, 
\widetilde{\rm ch}^{(+)}_{H(\Lambda^{[K(m), 0]})}(\tau,z)\Big\}
$$
\item[{\rm (ii)}] \,\ $
\widetilde{\rm ch}^{(+)}_{H(\Lambda^{[K(m), 1]})}
\Big(-\dfrac{1}{\tau}, \dfrac{z}{\tau}\Big)
\,\ = \,\ 
- \dfrac12 \, e^{-\frac{\pi im}{4}} \, e^{\frac{\pi imz^2}{2\tau}} 
\, \Big\{ $
$$
(1+e^{\frac{\pi im}{2}}) \, 
\widetilde{\rm ch}^{(+)}_{H(\Lambda^{[K(m), 1]})}(\tau,z)
\, + \, 
(e^{\frac{\pi im}{2}}-1) \,  \, 
\widetilde{\rm ch}^{(+)}_{H(\Lambda^{[K(m), 0]})}(\tau,z)\Big\}
$$
\item[{\rm (iii)}] \,\ $
\widetilde{\rm ch}^{(-)}_{H(\Lambda^{[K(m), 0]})}
\Big(-\dfrac{1}{\tau}, \dfrac{z}{\tau}\Big)$
$$
= \,\ 
- \, \dfrac12 \, e^{\frac{\pi imz^2}{2\tau}} \, \Big\{
\widetilde{\rm ch}^{(+){\rm tw}(\sigma_-)}_{H(\Lambda^{[K(m), 1]})}(\tau, z)
\, - \, e^{\frac{\pi im}{2}} \, 
\widetilde{\rm ch}^{(+){\rm tw}(\sigma_+)}_{H(\Lambda^{[K(m), p(m)]})}(\tau, z)
\Big\} 
$$
\item[{\rm (iv)}] \,\ $
\widetilde{\rm ch}^{(-)}_{H(\Lambda^{[K(m), 1]})}
\Big(-\dfrac{1}{\tau}, \dfrac{z}{\tau}\Big)$
$$
= \,\ 
- \, \dfrac12 \, e^{\frac{\pi imz^2}{2\tau}} \, \Big\{
\widetilde{\rm ch}^{(+){\rm tw}(\sigma_-)}_{H(\Lambda^{[K(m), 1]})}(\tau, z)
\, + \, e^{\frac{\pi im}{2}} \, 
\widetilde{\rm ch}^{(+){\rm tw}(\sigma_+)}_{H(\Lambda^{[K(m), p(m)]})}(\tau, z)
\Big\}
$$
\item[{\rm (v)}] \,\ $
\widetilde{\rm ch}^{(+){\rm tw}(\sigma_+)}_{H(\Lambda^{[K(m), p(m)]})}
\Big(-\dfrac{1}{\tau}, \dfrac{z}{\tau}\Big)$
$$
= \,\ 
e^{\frac{\pi imz^2}{2\tau}} \, \Big\{
\widetilde{\rm ch}^{(-)}_{H(\Lambda^{[K(m), 0]})}(\tau,z)
\,\ - \,\ 
\widetilde{\rm ch}^{(-)}_{H(\Lambda^{[K(m), 1]})}(\tau,z)
\Big\}
$$
\item[{\rm (vi)}] \,\ $
\widetilde{\rm ch}^{(+){\rm tw}(\sigma_-)}_{H(\Lambda^{[K(m), 1]})}
\Big(-\dfrac{1}{\tau}, \dfrac{z}{\tau}\Big)$
$$
= \,\ 
- \, e^{\frac{\pi imz^2}{2\tau}} \, \Big\{
\widetilde{\rm ch}^{(-)}_{H(\Lambda^{[K(m), 0]})}(\tau,z)
\,\ + \,\ 
\widetilde{\rm ch}^{(-)}_{H(\Lambda^{[K(m), 1]})}(\tau,z)
\Big\}
$$
\end{enumerate}

\item[{\rm 2)}] \,\ $T$-transformation  :
\begin{enumerate}
\item[{\rm (i)}] \quad $\widetilde{\rm ch}^{(+)}_{H(\Lambda^{[K(m), 0]})}
(\tau+1, \, z)
\,\ = \,\ 
e^{-(\frac{m}{8}+ \frac{11}{24}) \, \pi i} \, 
\widetilde{\rm ch}^{(-)}_{H(\Lambda^{[K(m), 0]})}(\tau,z)$
\item[{\rm (ii)}] \quad $\widetilde{\rm ch}^{(+)}_{H(\Lambda^{[K(m), 1]})}
(\tau+1, \, z)
\,\ = \,\ 
e^{-(\frac{m}{8}+ \frac{11}{24}) \, \pi i} \, 
\widetilde{\rm ch}^{(-)}_{H(\Lambda^{[K(m), 1]})}(\tau,z)$
\item[{\rm (iii)}] \quad $\widetilde{\rm ch}^{(-)}_{H(\Lambda^{[K(m), 0]})}
(\tau+1, \, z)
\,\ = \,\ 
- \, e^{-(\frac{m}{8}+ \frac{11}{24}) \, \pi i} \, 
\widetilde{\rm ch}^{(+)}_{H(\Lambda^{[K(m), 0]})}(\tau,z)$
\item[{\rm (iv)}] \quad $\widetilde{\rm ch}^{(-)}_{H(\Lambda^{[K(m), 1]})}
(\tau+1, \, z)
\,\ = \,\ 
e^{-(\frac{m}{8}+ \frac{11}{24}) \, \pi i} \, 
\widetilde{\rm ch}^{(+)}_{H(\Lambda^{[K(m), 1]})}(\tau,z)$
\item[{\rm (v)}] \quad 
$\widetilde{\rm ch}^{(+){\rm tw} (\sigma_+)}_{H(\Lambda^{[K(m), p(m)]})}
(\tau+1, \, z)
\,\ = \,\ 
e^{(\frac{m}{2}- \frac{1}{12}) \, \pi i} \, 
\widetilde{\rm ch}^{(+){\rm tw}(\sigma_+)}_{H(\Lambda^{[K(m), p(m)]})}(\tau,z)$
\item[{\rm (vi)}] \quad 
$\widetilde{\rm ch}^{(+){\rm tw} (\sigma_-)}_{H(\Lambda^{[K(m), 1]})}
(\tau+1, \, z)
\,\ = \,\ 
e^{-\frac{\pi i}{12}} \, 
\widetilde{\rm ch}^{(+){\rm tw}(\sigma_-)}_{H(\Lambda^{[K(m), 1]})}(\tau,z)$
\end{enumerate}
\end{enumerate}
\end{prop}

\begin{rem}
\label{n3:rem:2022-625a}
From this Proposition \ref{n3:prop:2022-115a} we can know 
all of modified (super)characters if we know only one of them, 
except in the case $m \in 4\nnn$. 
\end{rem}

We note that the cantral charge $c(m)$ of the N=3 module 
$H(\Lambda^{[K(m), m_2]})$ is 
%(label=n3:eqn:2022-122a)
\begin{equation}
c(m) \, = \,\ -6K(m)-\frac72 \,\ = \, \frac{3m-1}{2}
\label{n3:eqn:2022-122a}
\end{equation}

\section{Modified characters in the case $m=2$}
\label{sec:m2:modified}
%(line=1648)

In this section, we deduce the explicit formulas 
for the modified characters in the case $m=2$, namely 
for the N=3 modules $H(\Lambda^{[-1,m_2]})$ \,\ $(m_2=0,1,2)$ 
of central charage $c(2)=\frac52$.
By \eqref{n3:eqn:2022-112f}, we need to consider
the following 6 characters in this case:
$$
\widetilde{\rm ch}^{(\pm)}_{H(\Lambda^{[-1, 0]})}, \,\ 
\widetilde{\rm ch}^{(\pm)}_{H(\Lambda^{[-1, 1]})}, \,\ 
\widetilde{\rm ch}^{(+){\rm tw}(\sigma_{\pm})}_{H(\Lambda^{[-1, 1]})}
$$
The numerators of these characters are written by $
\overset{\circ}{A}{}^{[2]}_j(\tau, z)$ which, 
by Proposition \ref{n3:prop:2022-115a}, 
satisfy the following modular transformation properties:

%(line=1757)
%(label=n3:lemma:2022-113a)
\begin{lemma}
\label{n3:lemma:2022-113a}
\begin{enumerate}
\item[{\rm 1)}] \,\ $S$-transformation \,\ :
\begin{enumerate}
\item[{\rm (i)}] \quad $\overset{\circ}{A}{}^{[2]}_1\Big(-\dfrac{1}{\tau}, \dfrac{z}{\tau}\Big) 
\,\ = \,\ 
i \, \tau \, e^{\frac{\pi iz^2}{\tau}} \, \overset{\circ}{A}{}^{[2]}_2(\tau,z)$
\item[{\rm (ii)}] \quad $\overset{\circ}{A}{}^{[2]}_2\Big(-\dfrac{1}{\tau}, \dfrac{z}{\tau}\Big) 
\,\ = \,\ 
- \, i \, \tau \, e^{\frac{\pi iz^2}{\tau}} \, \overset{\circ}{A}{}^{[2]}_1(\tau,z)$
\item[{\rm (iii)}] \quad $\overset{\circ}{A}{}^{[2]}_3\Big(-\dfrac{1}{\tau}, \dfrac{z}{\tau}\Big) 
\,\ = \,\ 
- \, \tau \, e^{\frac{\pi iz^2}{\tau}} \, \overset{\circ}{A}{}^{[2]}_6(\tau,z)$
\item[{\rm (iv)}] \quad $\overset{\circ}{A}{}^{[2]}_4\Big(-\dfrac{1}{\tau}, \dfrac{z}{\tau}\Big) 
\,\ = \,\ 
\tau \, e^{\frac{\pi iz^2}{\tau}} \, \overset{\circ}{A}{}^{[2]}_5(\tau,z)$
\item[{\rm (v)}] \quad $\overset{\circ}{A}{}^{[2]}_5\Big(-\dfrac{1}{\tau}, \dfrac{z}{\tau}\Big) 
\,\ = \,\ 
\tau \, e^{\frac{\pi iz^2}{\tau}} \, \overset{\circ}{A}{}^{[2]}_4(\tau,z)$
\item[{\rm (vi)}] \quad $\overset{\circ}{A}{}^{[2]}_6\Big(-\dfrac{1}{\tau}, \dfrac{z}{\tau}\Big) 
\,\ = \,\ 
- \, \tau \, e^{\frac{\pi iz^2}{\tau}} \, \overset{\circ}{A}{}^{[2]}_3(\tau,z)$
\end{enumerate}

\item[{\rm 2)}] \,\ $T$-transformation \,\ :

\begin{enumerate}
\item[{\rm (i)}] \quad $\overset{\circ}{A}{}^{[2]}_1(\tau+1, \, z) \,\ = \,\ 
e^{-\frac{\pi i}{4}} \, \overset{\circ}{A}{}^{[2]}_3(\tau, \, z)$
\item[{\rm (ii)}] \quad $\overset{\circ}{A}{}^{[2]}_2(\tau+1, \, z) \,\ = \,\ 
e^{-\frac{\pi i}{4}} \, \overset{\circ}{A}{}^{[2]}_4(\tau, \, z)$
\item[{\rm (iii)}] \quad $\overset{\circ}{A}{}^{[2]}_3(\tau+1, \, z) \,\ = \,\ 
e^{-\frac{\pi i}{4}} \, \overset{\circ}{A}{}^{[2]}_2(\tau, \, z)$
\item[{\rm (iv)}] \quad $\overset{\circ}{A}{}^{[2]}_4(\tau+1, \, z) \,\ = \,\ 
e^{-\frac{\pi i}{4}} \, \overset{\circ}{A}{}^{[2]}_1(\tau, \, z)$
\item[{\rm (v)}] \quad $\overset{\circ}{A}{}^{[2]}_5(\tau+1, \, z) \,\ = \,\ 
\overset{\circ}{A}_5(\tau, \, z)$
\item[{\rm (vi)}] \quad $\overset{\circ}{A}{}^{[2]}_6(\tau+1, \, z) \,\ = \,\ 
- \, \overset{\circ}{A}{}^{[2]}_6(\tau, \, z)$
\end{enumerate}
\end{enumerate}
\end{lemma}

\medskip

In order to write these functions $\overset{\circ}{A}{}^{[2]}_j(\tau,  z)$ 
explicitly, we consider the following functions $g_i^{(\pm)}(\tau,z)$ :
%(label=n3:eqn:2022-119c1) $\sim$ (label=n3:eqn:2022-119c6)
\begin{subequations}
{\allowdisplaybreaks
\begin{eqnarray}
g_1^{(+)}(\tau, z) &:=& \dfrac{\eta(2\tau)^2}{\eta(\tau)} \cdot \dfrac{
\vartheta_{11}(\tau,z) \, \vartheta_{10}(\tau,z)}{
\vartheta_{01}(\tau,z)} 
\label{n3:eqn:2022-119c1}
\\[1mm]
g_1^{(-)}(\tau, z) &:=& \dfrac{\eta(2\tau)^2}{\eta(\tau)} \cdot \dfrac{
\vartheta_{11}(\tau,z) \, \vartheta_{10}(\tau,z)}{
\vartheta_{00}(\tau,z)}
\label{n3:eqn:2022-119c2}
\\[1mm]
g_2^{(+)}(\tau, z) &:=& \frac{\eta(\frac{\tau}{2})^2}{\eta(\tau)} \cdot 
\frac{\vartheta_{11}(\tau,z) \, \vartheta_{01}(\tau,z)}{
\vartheta_{10}(\tau,z)}
\label{n3:eqn:2022-119c3}
\\[1mm]
g_2^{(-)}(\tau, z) &:=& \frac{\eta(\frac{\tau}{2})^2}{\eta(\tau)} \cdot 
\frac{\vartheta_{11}(\tau,z) \, \vartheta_{01}(\tau,z)}{
\vartheta_{00}(\tau,z)}
\label{n3:eqn:2022-119c4}
\\[1mm]
g_3^{(+)}(\tau, z) &:=& 
\frac{\eta(\tau)^5}{\eta(\frac{\tau}{2})^2 \eta(2\tau)^2} 
\cdot 
\frac{\vartheta_{11}(\tau,z) \, \vartheta_{00}(\tau,z)}{\vartheta_{10}(\tau,z)}
\label{n3:eqn:2022-119c5}
\\[1mm]
g_3^{(-)}(\tau, z) &:=& 
\frac{\eta(\tau)^5}{\eta(\frac{\tau}{2})^2 \eta(2\tau)^2} 
\cdot 
\frac{\vartheta_{11}(\tau,z) \, \vartheta_{00}(\tau,z)}{\vartheta_{01}(\tau,z)}
\label{n3:eqn:2022-119c6}
\end{eqnarray}}
\end{subequations}

%(line=1844)
%(label=n3:lemma:2022-108d)
\begin{lemma}
\label{n3:lemma:2022-108d}
The functions $g^{(\pm)}_i$ defined above satisfy the following 
transformation properties:

\begin{enumerate}
\item[{\rm 1)}] \,\ $S$-transformation :
\begin{enumerate}
\item[{\rm (i)}] \quad $g_1^{(+)}\Big(-\dfrac{1}{\tau}, \dfrac{z}{\tau}\Big) 
\,\ = \,\ 
- \, \frac12 \, \tau \, e^{\frac{\pi i z^2}{\tau}} \, g_2^{(+)}(\tau,z)$
\item[{\rm (ii)}] \quad $g_1^{(-)}\Big(-\dfrac{1}{\tau}, \dfrac{z}{\tau}\Big) 
\,\ = \,\ 
- \, \frac12 \, \tau \, e^{\frac{\pi i z^2}{\tau}} \, g_2^{(-)}(\tau,z)$
\item[{\rm (iii)}] \quad $g_2^{(+)}\Big(-\dfrac{1}{\tau}, \dfrac{z}{\tau}\Big) 
\,\ = \,\ 
- \, 2 \, \tau \, e^{\frac{\pi i z^2}{\tau}} \, g_1^{(+)}(\tau,z)$
\item[{\rm (iv)}] \quad $g_2^{(-)}\Big(-\dfrac{1}{\tau}, \dfrac{z}{\tau}\Big) 
\,\ = \,\ 
- \, 2 \, \tau \, e^{\frac{\pi i z^2}{\tau}} \, g_1^{(-)}(\tau,z)$
\item[{\rm (v)}] \quad $g_3^{(+)}\Big(-\dfrac{1}{\tau}, \dfrac{z}{\tau}\Big) 
\,\ = \,\ 
- \, \tau \, e^{\frac{\pi i z^2}{\tau}} \, g_3^{(-)}(\tau,z)$
\item[{\rm (vi)}] \quad $g_3^{(-)}\Big(-\dfrac{1}{\tau}, \dfrac{z}{\tau}\Big) 
\,\ = \,\ 
- \, \tau \, e^{\frac{\pi i z^2}{\tau}} \, g_3^{(+)}(\tau,z)$
\end{enumerate}

\item[{\rm 2)}] \,\ $T$-transformation :
\begin{enumerate}
\item[{\rm (i)}] \quad $g_1^{(+)} (\tau+1, \, z) 
\,\ = \,\ e^{\frac{3\pi i}{4}} \, g_1^{(-)} (\tau, z)$
\item[{\rm (ii)}] \quad $g_1^{(-)} (\tau+1, \, z) 
\,\ = \,\ e^{\frac{3\pi i}{4}} \, g_1^{(+)} (\tau, z)$
\item[{\rm (iii)}] \quad $g_2^{(+)} (\tau+1, \, z) 
\,\ = \,\ \hspace{6mm} g_3^{(+)} (\tau, z)$
\item[{\rm (iv)}] \quad $g_2^{(-)} (\tau+1, \, z) 
\,\ = \,\ e^{\frac{\pi i}{4}} \, g_3^{(-)} (\tau, z)$
\item[{\rm (v)}] \quad $g_3^{(+)} (\tau+1, \, z) 
\,\ = \,\ \hspace{6mm} g_2^{(+)} (\tau, z)$
\item[{\rm (vi)}] \quad $g_3^{(-)} (\tau+1, \, z) 
\,\ = \,\ e^{\frac{\pi i}{4}} \, g_2^{(-)} (\tau, z)$
\end{enumerate}
\end{enumerate}
\end{lemma}

Now we see that the functions $\overset{\circ}{A}{}^{[2]}_j(\tau,z)$ 
and $g_j^{(\pm)}(\tau,z)$ are related as follows:

%(line=1895)
%(label=n3:lemma:2022-108e)
\begin{lemma} \,\ 
\label{n3:lemma:2022-108e}
\begin{enumerate}
\item[{\rm 1)}]
\begin{enumerate}
\item[{\rm (i)}] \quad $\overset{\circ}{A}{}^{[2]}_1(\tau, z) 
\,\ + \,\ \overset{\circ}{A}{}^{[2]}_2(\tau, z)
\,\ = \,\ 2 \, i \, g^{(-)}_1(\tau, z)$
\item[{\rm (ii)}] \quad $\overset{\circ}{A}{}^{[2]}_3(\tau, z) 
\,\ + \,\ \overset{\circ}{A}{}^{[2]}_4(\tau, z)
\,\ = \,\ - \, 2 \, i \, g^{(+)}_1(\tau, z)$
\item[{\rm (iii)}] \quad $\overset{\circ}{A}{}^{[2]}_1(\tau, z) 
\,\ - \,\ \overset{\circ}{A}{}^{[2]}_2(\tau, z) 
\,\ = \,\ g^{(-)}_2(\tau, z)$
\item[{\rm (iv)}] \quad $\overset{\circ}{A}{}^{[2]}_5(\tau, z) 
\,\ - \,\ \overset{\circ}{A}{}^{[2]}_6(\tau, z) 
\,\ = \,\ i \, g^{(+)}_2(\tau, z)$
\item[{\rm (v)}] \quad $\overset{\circ}{A}{}^{[2]}_3(\tau, z) 
\,\ - \,\ \overset{\circ}{A}{}^{[2]}_4(\tau, z) 
\,\ = \,\ i \, g^{(-)}_3(\tau, z) $
\item[{\rm (vi)}] \quad $\overset{\circ}{A}{}^{[2]}_5(\tau, z) 
\,\ + \,\ \overset{\circ}{A}{}^{[2]}_6(\tau, z) 
\,\ = \,\ i \, g^{(+)}_3(\tau, z)$
\end{enumerate}

\item[{\rm 2)}]
\begin{enumerate}
\item[{\rm (i)}] \quad $2 \, \overset{\circ}{A}{}^{[2]}_1(\tau, z) 
\,\ = \,\ 
2 \, i \, g^{(-)}_1(\tau, z) \,\ + \,\ g^{(-)}_2(\tau, z)$
\item[{\rm (ii)}] \quad $2 \, \overset{\circ}{A}{}^{[2]}_2(\tau, z) 
\,\ = \,\ 
2 \, i \, g^{(-)}_1(\tau, z) \,\ - \,\ g^{(-)}_2(\tau, z)$
\item[{\rm (iii)}] \quad $2 \, \overset{\circ}{A}{}^{[2]}_3(\tau, z) 
\,\ = \,\ i \, \big\{
- \, 2 \, g^{(+)}_1(\tau, z) \,\ + \,\ g^{(-)}_3(\tau, z)\big\}$
\item[{\rm (iv)}] \quad $2 \, \overset{\circ}{A}{}^{[2]}_4(\tau, z) 
\,\ = \,\ - \, i \, \big\{
2 \, g^{(+)}_1(\tau, z) \,\ + \,\ g^{(-)}_3(\tau, z)\big\}$
\item[{\rm (v)}] \quad $2 \, \overset{\circ}{A}{}^{[2]}_5(\tau, z) 
\,\ = \,\ i \, \big\{
g^{(+)}_3(\tau, z) \,\ + \,\ g^{(+)}_2(\tau, z)\big\}$
\item[{\rm (vi)}] \quad $2 \, \overset{\circ}{A}{}^{[2]}_6(\tau, z) 
\,\ = \,\ i \, \big\{
g^{(+)}_3(\tau, z) \,\ - \,\ g^{(+)}_2(\tau, z)\big\}$
\end{enumerate}
\end{enumerate}
\end{lemma}

\begin{proof} 1) In order to prove (i) and (ii), we let $m=2$ in the formulas 
\eqref{n3:eqn:2022-117d1} and \eqref{n3:eqn:2022-117d3} :
\begin{subequations}
{\allowdisplaybreaks
\begin{eqnarray}
& &
\widetilde{\Phi}^{[1,0]}
\left(2\tau, z+\frac{\tau}{2}-\frac12,z-\frac{\tau}{2}+\frac12,\frac{\tau}{8}\right)
= \frac12 \big\{
\overset{\circ}{A}{}^{[2]}_1(\tau, z)+\overset{\circ}{A}{}^{[2]}_2(\tau, z)\big\}
\nonumber
\\[0mm]
& &
\label{n3:eqn:2022-119h1}
\\[0mm]
& &
\widetilde{\Phi}^{[1,0]}
\left(2\tau, z+\frac{\tau}{2},z-\frac{\tau}{2},\frac{\tau}{8}\right)
= \frac12 \big\{
\overset{\circ}{A}{}^{[2]}_3(\tau, z)+\overset{\circ}{A}{}^{[2]}_4(\tau, z)\big\}
\label{n3:eqn:2022-119h2}
\end{eqnarray}}
\end{subequations}
The LHS of these equations can be computed by using Lemma \ref{n3:lemma:2022-111e}
as follows:
\begin{subequations}
{\allowdisplaybreaks
\begin{eqnarray}
& & \hspace{-10mm}
\widetilde{\Phi}^{[1,0]}
\left(2\tau, \, z+\frac{\tau}{2}-\frac12, \, z-\frac{\tau}{2}+\frac12, \, 
\frac{\tau}{8}\right)
\nonumber
\\[1mm]
&=&
-i q^{-\frac18}
\frac{\eta(\tau)^3 \, \vartheta_{11}(2\tau, 2z)}{
\vartheta_{11}(2\tau, z+\frac{\tau}{2}-\frac12) 
\vartheta_{11}(2\tau, z-\frac{\tau}{2}+\frac12) }
\nonumber
\\[1mm]
&=&
i q^{-\frac18}
\frac{\eta(\tau)^3 \, \vartheta_{11}(2\tau, 2z)}{
\vartheta_{10}(2\tau, z+\frac{\tau}{2}) 
\vartheta_{10}(2\tau, z-\frac{\tau}{2})}
\nonumber
\\[1mm]
&=&
i \frac{\eta(2\tau)^2}{\eta(\tau)} \cdot 
\frac{\vartheta_{11}(\tau, z)\vartheta_{10}(\tau, z)
}{\vartheta_{00}(\tau, z)}
\,\ = \,\ 
i g^{(-)}_1(\tau, z)
\label{n3:eqn:2022-119j1}
\\[2mm]
& & \hspace{-10mm}
\widetilde{\Phi}^{[1,0]}
\left(2\tau, z+\frac{\tau}{2},z-\frac{\tau}{2},\frac{\tau}{8}\right)
\, = \, 
-i q^{-\frac18}
\frac{\eta(\tau)^3 \, \vartheta_{11}(2\tau, 2z)}{
\vartheta_{11}(2\tau, z+\frac{\tau}{2}) 
\vartheta_{11}(2\tau, z-\frac{\tau}{2}) }
\nonumber
\\[1mm]
&=&
-i \frac{\eta(2\tau)^2}{\eta(\tau)} \cdot 
\frac{\vartheta_{11}(\tau, z)\vartheta_{10}(\tau, z)
}{\vartheta_{01}(\tau, z)}
\,\ = \,\ 
-i g^{(+)}_1(\tau, z)
\label{n3:eqn:2022-119j2}
\end{eqnarray}}
\end{subequations}
where we used the following formulas:
%(label=n3:eqn:2022-120g1) $\sim$ (label=n3:eqn:2022-120g3)
\begin{subequations}
{\allowdisplaybreaks
\begin{eqnarray}
& &
\vartheta_{11}(2\tau, 2z)
= \frac{\eta(2\tau)}{\eta(\tau)^2}
\vartheta_{11}(\tau,z)\vartheta_{10}(\tau,z)
\label{n3:eqn:2022-120g1}
\\[1mm]
& &
\vartheta_{10}\Big(2\tau, z+\frac{\tau}{2}\Big) 
\vartheta_{10}\Big(2\tau, z-\frac{\tau}{2}\Big) 
= q^{-\frac18} \frac{\eta(2\tau)^2}{\eta(\tau)}
\vartheta_{00}(\tau,z)
\label{n3:eqn:2022-120g2}
\\[1mm]
& &
\vartheta_{11}\Big(2\tau, z+\frac{\tau}{2}\Big) 
\vartheta_{11}\Big(2\tau, z-\frac{\tau}{2}\Big) 
= q^{-\frac18} \frac{\eta(2\tau)^2}{\eta(\tau)}
\vartheta_{01}(\tau,z) \, ,
\label{n3:eqn:2022-120g3}
\end{eqnarray}}
\end{subequations}
Then, by \eqref{n3:eqn:2022-119h1} and \eqref{n3:eqn:2022-119j1}
we obtain (i), and 
by \eqref{n3:eqn:2022-119h2} and \eqref{n3:eqn:2022-119j2} we obtain (ii). 
Since (i) and (ii) are thus established, 
(iii) is obtained by applying $S$-transformation to (i),
(iv) is obtained by applying $S$-transformation to (ii),
(v) is obtained by applying $T$-transformation to (iii),
(vi) is obtained by applying $T$-transformation to (iv), 
which complete proof of 1).  2) follows from 1) immediately.
\end{proof}

%(line=2058)
%(label=n3:thm:2022-108a)
\begin{thm} 
\label{n3:thm:2022-108a}
In the case $m=2$, the modified characters are as follows:

\begin{enumerate}
\item[{\rm 1)}] $\widetilde{\rm ch}{}^{(+)}_{H(\Lambda^{[-1, 0]})}(\tau,z) 
= 
- \dfrac12 \cdot 
\dfrac{\eta(\frac{\tau}{2})}{\eta(2\tau)\eta(\tau)}
\cdot \vartheta_{01}(\tau,z)$ 
\item[{\rm 2)}] ${\rm ch}{}^{(+)}_{H(\Lambda^{[-1, 1]})}(\tau,z) 
= 
i \, \dfrac{\eta(2\tau)}{\eta(\frac{\tau}{2})\eta(\tau)} 
\cdot \vartheta_{10}(\tau,z)$
\item[{\rm 3)}] $\widetilde{\rm ch}{}^{(-)}_{H(\Lambda^{[-1, 0]})}(\tau,z) 
= 
- \dfrac{i}{2} \cdot 
\dfrac{\eta(\tau)^2}{\eta(\frac{\tau}{2})\eta(2\tau)^2} 
\cdot \vartheta_{00}(\tau,z)$
\item[{\rm 4)}]  ${\rm ch}{}^{(-)}_{H(\Lambda^{[-1, 1]})}(\tau,z) 
= 
- i \, \dfrac{\eta(\frac{\tau}{2})\eta(2\tau)^2}{\eta(\tau)^4} \cdot 
\vartheta_{10}(\tau,z)$
\item[{\rm 5)}] ${\rm ch}{}^{(+) {\rm tw} (\sigma_-)
}_{H(\Lambda^{[-1, 1]})}(\tau,z) 
= 
\dfrac{i}{\sqrt{2}} \bigg\{
\dfrac{\eta(\tau)^2}{\eta(\frac{\tau}{2})^2 \eta(2\tau)} 
\vartheta_{00}(\tau,z)
+ 
\dfrac{\eta(\frac{\tau}{2})^2\eta(2\tau)}{\eta(\tau)^4} \, \vartheta_{01}(\tau,z)
\bigg\} $
\item[{\rm 6)}] ${\rm ch}{}^{(+) {\rm tw} (\sigma_+)
}_{H(\Lambda^{[-1, 1]})}(\tau,z) 
= 
\dfrac{i}{\sqrt{2}} \bigg\{
\dfrac{\eta(\tau)^2}{\eta(\frac{\tau}{2})^2 \eta(2\tau)} 
\vartheta_{00}(\tau,z)
- 
\dfrac{\eta(\frac{\tau}{2})^2\eta(2\tau)}{\eta(\tau)^4} \vartheta_{01}(\tau,z)
\bigg\} $
\end{enumerate}
\end{thm}

\begin{proof} This theorem is obtained from 
Corollary \ref{n3:cor:2022-112a} and Lemma \ref{n3:lemma:2022-108c}
and Lemma \ref{n3:lemma:2022-108e} and the formula \eqref{n3:eqn:2022-115a} 
as follows:
{\allowdisplaybreaks
\begin{eqnarray*}
& & \hspace{-10mm}
\widetilde{\rm ch}{}^{(+)}_{H(\Lambda^{[-1, 0]})}(\tau,z) 
= \frac{1}{\overset{N=3}{R}{}^{(+)}(\tau,z)} 
\big(\overset{N=3}{R}{}^{(+)} \widetilde{\rm ch}{}^{(+)}_{H(\Lambda^{[-1, 0]})}\big)
(\tau,z) 
\\[1mm]
&=&
\frac{1}{\overset{N=3}{R}{}^{(+)}(\tau, z)} \cdot 
\Big[\overset{N=3}{R}{}^{(+)} \, 
\widetilde{\rm ch}{}^{(+)}_{H(\Lambda^{[m, 0]})}\Big](\tau,z) 
\\[1mm]
&=&
- \, \frac12 \, 
\frac{1}{\eta(\frac{\tau}{2}) \eta(2\tau)} \cdot 
\frac{\vartheta_{00}(\tau, z)}{\vartheta_{11}(\tau, z)}
\cdot \frac{\eta(\frac{\tau}{2})^2}{\eta(\tau)} \cdot 
\frac{\vartheta_{11}(\tau,z) \, \vartheta_{01}(\tau,z)}{
\vartheta_{00}(\tau,z)}
\\[1mm]
&=&
\,\ = \,\ 
- \, \frac12 \, 
\frac{\eta(\frac{\tau}{2})}{\eta(\tau)\eta(2\tau)} \cdot \vartheta_{01}(\tau,z)
% (+)2
\\[2mm]
& & \hspace{-10mm}
\widetilde{\rm ch}{}^{(+)}_{H(\Lambda^{[-1, 1]})}(\tau,z) 
= \frac{1}{\overset{N=3}{R}{}^{(+)}(\tau,z)} 
\big(\overset{N=3}{R}{}^{(+)} \widetilde{\rm ch}{}^{(+)}_{H(\Lambda^{[-1,1]})}\big)
(\tau,z) 
\\[1mm]
&=&
i \, 
\frac{1}{\eta(\frac{\tau}{2}) \eta(2\tau)} \cdot 
\frac{\vartheta_{00}(\tau, z)}{\vartheta_{11}(\tau, z)}
\cdot 
\frac{\eta(2\tau)^2}{\eta(\tau)} \cdot \dfrac{
\vartheta_{11}(\tau,z) \, \vartheta_{10}(\tau,z)}{
\vartheta_{00}(\tau,z)}
\,\ = \,\ 
i \, 
\frac{\eta(2\tau)}{\eta(\frac{\tau}{2})\eta(\tau)} \cdot \vartheta_{10}(\tau,z)
% (-)1
\\[2mm]
& & \hspace{-10mm}
\widetilde{\rm ch}{}^{(-)}_{H(\Lambda^{[-1, 0]})}(\tau,z) 
= \frac{1}{\overset{N=3}{R}{}^{(-)}(\tau,z)} 
\big(\overset{N=3}{R}{}^{(-)} \widetilde{\rm ch}{}^{(+)}_{H(\Lambda^{[-1, 0]})}\big)
(\tau,z) 
\\[1mm]
&=&
- \, i \cdot 
\frac{\eta(\frac{\tau}{2})}{ \eta(\tau)^3} \cdot 
\frac{\vartheta_{01}(\tau, z)}{\vartheta_{11}(\tau, z)} \cdot 
\frac{\eta(\tau)^5}{\eta(\frac{\tau}{2})^2 \eta(2\tau)^2} 
\cdot 
\frac{\vartheta_{11}(\tau,z) \, \vartheta_{00}(\tau,z)}{\vartheta_{01}(\tau,z)}
\\[1mm]
&=&
- \, \frac{i}{2} \cdot 
\frac{\eta(\tau)^2}{\eta(\frac{\tau}{2})\eta(2\tau)^2} \cdot \vartheta_{00}(\tau,z)
% (-)2
\\[2mm]
& & \hspace{-10mm}
\widetilde{\rm ch}{}^{(-)}_{H(\Lambda^{[-1, 1]})}(\tau,z) 
= \frac{1}{\overset{N=3}{R}{}^{(-)}(\tau,z)} 
\big(\overset{N=3}{R}{}^{(-)} \widetilde{\rm ch}{}^{(+)}_{H(\Lambda^{[-1,1]})}\big)
(\tau,z) 
\\[1mm]
&=&
- \, i \, \cdot 
\frac{\eta(\frac{\tau}{2})}{ \eta(\tau)^3} \cdot 
\frac{\vartheta_{01}(\tau, z)}{\vartheta_{11}(\tau, z)} \cdot 
\frac{\eta(2\tau)^2}{\eta(\tau)} \cdot \frac{
\vartheta_{11}(\tau,z) \, \vartheta_{10}(\tau,z)}{\vartheta_{01}(\tau,z)} 
\\[1mm]
&=&
- \, i \, \cdot 
\frac{\eta(\frac{\tau}{2})\eta(2\tau)^2}{\eta(\tau)^4} \cdot 
\vartheta_{10}(\tau,z)
% tw 1
\\[2mm]
& & \hspace{-10mm}
\widetilde{\rm ch}{}^{(+) \, {\rm tw}(\sigma_-)}_{H(\Lambda^{[-1, 1]})}(\tau,z) 
\,\ = \,\ 
\frac{1}{\overset{N=3}{R}{}^{(+) \, {\rm tw}}(\tau, z)} \cdot 
\Big[\overset{N=3}{R}{}^{(+) \, {\rm tw}} \cdot 
\widetilde{\rm ch}{}^{(+) \, {\rm tw}(\sigma_-)}_{H(\Lambda^{[-1, 1]})}\Big](\tau,z) 
\\[1mm]
&=&
\frac{i}{2} \, \cdot \sqrt{2} \cdot 
\frac{\eta(2\tau)}{ \eta(\tau)^3} \cdot 
\frac{\vartheta_{10}(\tau, z)}{\vartheta_{11}(\tau, z)} 
\\[1mm]
& & \hspace{3mm}
\times \,\ \bigg\{
\frac{\eta(\frac{\tau}{2})^2}{\eta(\tau)} \cdot 
\frac{\vartheta_{11}(\tau,z) \, \vartheta_{01}(\tau,z)}{
\vartheta_{10}(\tau,z)}
\,\ + \,\ 
\frac{\eta(\tau)^5}{\eta(\frac{\tau}{2})^2 \eta(2\tau)^2} 
\cdot 
\frac{\vartheta_{11}(\tau,z) \, \vartheta_{00}(\tau,z)}{\vartheta_{10}(\tau,z)}
\bigg\}
\\[1mm]
&=&
\frac{i}{\sqrt{2}} \,\ \bigg\{
\frac{\eta(\frac{\tau}{2})^2\eta(2\tau)}{\eta(\tau)^4} \, \vartheta_{01}(\tau,z)
\,\ + \,\ 
\frac{\eta(\tau)^2}{\eta(\frac{\tau}{2})^2 \eta(2\tau)} 
\, \vartheta_{00}(\tau,z)
\bigg\}
% tw 2
\\[2mm]
& & \hspace{-10mm}
\widetilde{\rm ch}{}^{(+) \, {\rm tw}(\sigma_+)}_{H(\Lambda^{[-1, 1]})}(\tau,z) 
\,\ = \,\ 
\frac{1}{\overset{N=3}{R}{}^{(+) \, {\rm tw}}(\tau, z)} \cdot 
\Big[\overset{N=3}{R}{}^{(+) \, {\rm tw}} \cdot 
\widetilde{\rm ch}{}^{(+) \, {\rm tw}(\sigma_+)}_{H(\Lambda^{[-1, 1]})}\Big](\tau,z) 
\\[1mm]
&=&
\frac{i}{2} \, \cdot \sqrt{2} \cdot 
\frac{\eta(2\tau)}{ \eta(\tau)^3} \cdot 
\frac{\vartheta_{10}(\tau, z)}{\vartheta_{11}(\tau, z)} 
\\[3mm]
& & \hspace{3mm}
\times \,\ \bigg\{
\frac{\eta(\tau)^5}{\eta(\frac{\tau}{2})^2 \eta(2\tau)^2} 
\cdot 
\frac{\vartheta_{11}(\tau,z) \, \vartheta_{00}(\tau,z)}{\vartheta_{10}(\tau,z)}
\,\ - \,\ 
\frac{\eta(\frac{\tau}{2})^2}{\eta(\tau)} \cdot 
\frac{\vartheta_{11}(\tau,z) \, \vartheta_{01}(\tau,z)}{
\vartheta_{10}(\tau,z)}
\bigg\}
\\[1mm]
&=&
\frac{i}{\sqrt{2}} \,\ \bigg\{
\frac{\eta(\tau)^2}{\eta(\frac{\tau}{2})^2 \eta(2\tau)} 
\, \vartheta_{00}(\tau,z)
\,\ - \,\ 
\frac{\eta(\frac{\tau}{2})^2\eta(2\tau)}{\eta(\tau)^4} \, \vartheta_{01}(\tau,z)
\bigg\}
\end{eqnarray*}}
Thus the proof is completed. Since, by Lemma \ref{n3:lemma:2022-111e}, 
the modified characters are the same with the honest characters for 
$H(\Lambda^{[-1,1]})$, the formulas for the characters of $H(\Lambda^{[-1,1]})$
are exhibited by using \lq \lq ch" in place of \lq \lq $\widetilde{\rm ch}$"
in this theorem.
\end{proof}

Then by modular transformation properties of $\eta(\tau)$ and 
$\vartheta_{ab}(\tau, z)$, we obtain the following transformation 
properties of the modified characters in the case $m=2$ which, 
of course, coincide with Proposition \ref{n3:prop:2022-115a}.

%(line=2266)
%(label=n3:cor:2022-108a)
\begin{cor} \,\ 
\label{n3:cor:2022-108a}
\begin{enumerate}
\item[{\rm 1)}] \,\ $S$-transformation :
\begin{enumerate}
\item[{\rm (i)}] \quad $\widetilde{\rm ch}^{(+)}_{H(\Lambda^{[-1,0]})}
\Big(- \dfrac{1}{\tau}, \dfrac{z}{\tau}\Big) 
\,\ = \,\
i \, e^{\frac{\pi i z^2}{\tau}} \, 
{\rm ch}^{(+)}_{H(\Lambda^{[-1,1]})}(\tau, z)$
\item[{\rm (ii)}] \quad ${\rm ch}^{(+)}_{H(\Lambda^{[-1,1]})}
\Big(- \dfrac{1}{\tau}, \dfrac{z}{\tau}\Big) 
\,\ = \,\ 
- \, i \, e^{\frac{\pi i z^2}{\tau}} \, 
\widetilde{\rm ch}^{(+)}_{H(\Lambda^{[-1,0]})}(\tau, z)$
\item[{\rm (iii)}] \quad $\widetilde{\rm ch}^{(-)}_{H(\Lambda^{[-1,0]})}
\Big(- \dfrac{1}{\tau}, \dfrac{z}{\tau}\Big)$
$$
= \,\ 
- \, \frac12 \, e^{\frac{\pi iz^2}{\tau}} \, \Big\{
{\rm ch}{}^{(+) \, {\rm tw} \, (\sigma_-)}_{H(\Lambda^{[-1, 1]})}(\tau,z)
\,\ + \,\ 
\widetilde{\rm ch}{}^{(+) \, {\rm tw} \, (\sigma_+)}_{H(\Lambda^{[-1, 1]})}(\tau,z)
\Big\}
$$
\item[{\rm (iv)}] \quad ${\rm ch}^{(-)}_{H(\Lambda^{[-1,1]})}
\Big(- \dfrac{1}{\tau}, \dfrac{z}{\tau}\Big) $
$$
\,\ = \,\ 
- \, \frac12 \, e^{\frac{\pi iz^2}{\tau}} \, \Big\{
{\rm ch}{}^{(+) \, {\rm tw} \, (\sigma_-)}_{H(\Lambda^{[-1, 1]})}(\tau,z)
\,\ - \,\ 
\widetilde{\rm ch}{}^{(+) \, {\rm tw} \, (\sigma_+)}_{H(\Lambda^{[-1, 1]})}(\tau,z)
\Big\}
$$
\item[{\rm (v)}] ${\rm ch}^{(+) {\rm tw}(\sigma_-)}_{H(\Lambda^{[-1,1]})}
\Big(- \dfrac{1}{\tau}, \dfrac{z}{\tau}\Big) 
= - 
e^{\frac{\pi iz^2}{\tau}} \Big\{
\widetilde{\rm ch}^{(-)}_{H(\Lambda^{[-1,0]})}(\tau,z)
+ 
{\rm ch}^{(-)}_{H(\Lambda^{[-1,1]})}(\tau,z) \Big\}$
\item[{\rm (vi)}] $
\widetilde{\rm ch}^{(+) {\rm tw}(\sigma_+)}_{H(\Lambda^{[-1,1]})}
\Big(- \dfrac{1}{\tau}, \dfrac{z}{\tau}\Big) 
= - 
e^{\frac{\pi iz^2}{\tau}} \Big\{
\widetilde{\rm ch}^{(-)}_{H(\Lambda^{[-1,0]})}(\tau,z)
- 
{\rm ch}^{(-)}_{H(\Lambda^{[-1,1]})}(\tau,z) \Big\}$
\end{enumerate}

\item[{\rm 2)}] \,\ $T$-transformation :

\begin{enumerate}
\item[{\rm (i)}] \quad $\widetilde{\rm ch}^{(+)}_{H(\Lambda^{[-1,0]})}
(\tau+1, \, z) \,\ = \,\
- \,\ e^{ \frac{7}{24} \pi i} \, 
\widetilde{\rm ch}{}^{(-)}_{H(\Lambda^{[-1, 0]})}(\tau,z) $
\item[{\rm (ii)}] \quad ${\rm ch}^{(+)}_{H(\Lambda^{[-1,1]})}
(\tau+1, \, z) \,\ = \,\
- \,\ e^{\frac{7}{24} \pi i} \, 
{\rm ch}{}^{(-)}_{H(\Lambda^{[-1, 1]})}(\tau,z) $
\item[{\rm (iii)}] \quad $\widetilde{\rm ch}^{(-)}_{H(\Lambda^{[-1,0]})}
(\tau+1, \, z) \,\ = \,\
e^{\frac{7}{24} \pi i} \, 
\widetilde{\rm ch}{}^{(+)}_{H(\Lambda^{[-1, 0]})}(\tau,z)$
\item[{\rm (iv)}] \quad ${\rm ch}^{(-)}_{H(\Lambda^{[-1,1]})}
(\tau+1, \, z) 
\,\ = \,\ - \,\ 
e^{\frac{7}{24} \pi i} \, 
{\rm ch}{}^{(+)}_{H(\Lambda^{[-1, 1]})}(\tau,z)$
\item[{\rm (v)}] \quad 
${\rm ch}^{(+) \, {\rm tw} (\sigma_-)}_{H(\Lambda^{[-1,1]})}
(\tau+1, \, z) \,\ = \,\
e^{-\frac{\pi i}{12}} \,\ 
{\rm ch}{}^{(+) \, {\rm tw} \, (\sigma_-)}_{H(\Lambda^{[-1, 1]})}(\tau,z)$
\item[{\rm (vi)}] \quad 
$\widetilde{\rm ch}^{(+) \, {\rm tw} (\sigma_+)}_{H(\Lambda^{[-1,1]})}
(\tau+1, \, z) \,\ = \,\
- \,\ e^{-\frac{\pi i}{12}} \,\ 
\widetilde{\rm ch}{}^{(+) \, {\rm tw} \, (\sigma_+)}_{H(\Lambda^{[-1, 1]})}
(\tau,z)$
\end{enumerate}
\end{enumerate}
\end{cor}

%(line=2355)
%(label=n3:rem:2022-115a)
\begin{rem}
\label{n3:rem:2022-115a} 
As is explained in the proof of Theorem \ref{n3:cor:2022-108a},
the modified characters of $H(\Lambda^{[-1,1]})$ coinside with the honest 
characters by Lemma \ref{n3:lemma:2022-111e}, namely 
$$
\widetilde{\rm ch}^{(\pm)}_{H(\Lambda^{[-1,1]})}
={\rm ch}^{(\pm)}_{H(\Lambda^{[-1,1]})} 
\quad \text{and} \quad 
\widetilde{\rm ch}^{(+) \, {\rm tw} (\sigma_{\pm})}_{H(\Lambda^{[-1,1]})}
= {\rm ch}^{(+) \, {\rm tw} (\sigma_{\pm})}_{H(\Lambda^{[-1,1]})} \, ,
$$
so, in the above Theorem \ref{n3:cor:2022-108a} and Corollary \ref{n3:cor:2022-108a},
formulas are exposed in the form where {\rm ch} and $\widetilde{\rm ch}$ are mixed.
Corollary \ref{n3:cor:2022-108a} shows that, in the case $m=2$, the 
$SL_2(\zzz)$-invariance holds for the space in which 
honest characters and modified characters collaborate.
\end{rem}

\section{Honest characters in the case $m=2$}
\label{sec:m2:honest}

In this section, we compute the correction term $\Phi^{[1, \frac12]}_{\rm add}$
to obtain the honest (super)characters of $H(\Lambda^{[-1,m_2]})$ 
for $m_2=0,2$.

%(line=2294)
%(label=n3:lemma:2022-115b)
\begin{lemma} \quad
\label{n3:lemma:2022-115b}
\begin{enumerate}
\item[{\rm 1)}] $R_{j,1}\Big(\tau, \, \dfrac{\tau}{4}\Big) 
\,\ = \,\ \left\{
\begin{array}{ccc}
q^{\frac{1}{16}} & & \text{if} \quad j=\frac12 \\[2mm]
0 & & \text{if} \quad j=\frac32
\end{array} \right. $
\item[{\rm 2)}] $\Phi^{[1, \frac12]}_{\rm add}
\Big(\tau, z+\dfrac{\tau}{4},  z-\dfrac{\tau}{4}, 0\Big)
\,\ = \,\ 
-\dfrac12 \, q^{\frac{1}{16}} \, 
\big[\theta_{\frac12,1}-\theta_{-\frac12,1}\big](\tau, 2z)$
\item[{\rm 3)}] $\Phi^{[1, \frac12]}_{\rm add}
\Big(2\tau, z+\dfrac{\tau}{2},  z-\dfrac{\tau}{2}, 0\Big)
\,\ = \,\ 
\dfrac{i}{2} \, q^{\frac{1}{8}} \, \vartheta_{11}(\tau,z)$
\item[{\rm 4)}] $\Phi^{[1, \frac12]}_{\rm add}
\Big(2\tau, z+\dfrac{\tau}{2},  z-\dfrac{\tau}{2}, \dfrac{\tau}{8}\Big)
\,\ = \,\ 
\dfrac{i}{2} \, \vartheta_{11}(\tau,z)$
\end{enumerate}
\end{lemma}

\begin{proof} 1) \,\ Letting $m=1$ and $w=\frac{\tau}{4}$ in 
\eqref{n3:eqn:2022-111h}, we have
$$
R_{j;1}\left(\tau,\frac{\tau}{4}\right) = \hspace{-4mm}
\sum_{\substack{ \\[1mm] n \equiv j \, {\rm mod} \, 2}}
\hspace{-4mm}
\left\{{\rm sgn}\left(n+\frac32-j\right)
-E\left(\left(n-\frac12\right)
\sqrt{{\rm Im}(\tau)}\right)\right\}
e^{-\frac{\pi in^2\tau}{2} +\frac{\pi in\tau}{2}}
$$
Putting $n=j+2k$, this equation is rewitten as follows:
{\allowdisplaybreaks
\begin{eqnarray*}
& &
R_{j;1}\left(\tau,\frac{\tau}{4}\right) 
\\[1mm]
&=& \sum_{k \in \zzz}
\left\{{\rm sgn}\left(2k+\frac32\Big)
-E\left(\Big(j+2k-\frac12\right)
\sqrt{{\rm Im}(\tau)}\right)\right\}
q^{\frac14(j+2k)(1-j-2k)}
\\[1mm]
&=&
\underbrace{
\Big[\sum_{k \geq 0} \, - \, \sum_{k<0}\Big]q^{\frac14(j+2k)(1-j-2k)}}_{
\substack{\hspace{6mm} || \,\ put \\[1mm] {\rm (I)}_j
}}
\\[0mm]
& &
\underbrace{- \, \sum_{k \in \zzz}
E\left(\left(j+2k-\frac12\right)\sqrt{{\rm Im}(\tau)}\right)
q^{\frac14(j+2k)(1-j-2k)}}_{
\substack{\hspace{6mm} || \,\ put \\[1mm] {\rm (II)}_j
}}
\end{eqnarray*}}
It is easy to see that

\vspace{-5mm}

$$
{\rm (I)}_j \,\ = \,\ \left\{
\begin{array}{ccc}
q^{\frac{1}{16}} & & {\rm if} \,\ j=\frac12 \\[2mm]
0 & & {\rm if} \,\ j=\frac32
\end{array}\right. 
$$
and that ${\rm (II)}_{\frac12} = {\rm (II)}_{\frac32}=0$ since $E(-x)=-E(x)$.
Thus we have 
$$
{\rm (I)}_j +{\rm (II)}_j \,\ = \,\ \left\{
\begin{array}{ccc}
q^{\frac{1}{16}} & & {\rm if} \,\ j=\frac12 \\[2mm]
0 & & {\rm if} \,\ j=\frac32
\end{array}\right. \, ,
$$
proving 1).

\medskip

\noindent
2) By the equation \eqref{n3:eqn:2022-111j} and 1), we have 
{\allowdisplaybreaks
\begin{eqnarray*}
\lefteqn{
\Phi_{\rm add}^{[1, \frac12]}
\left(\tau, z+\frac{\tau}{4}, z-\frac{\tau}{4}, 0\right)}
\\[1mm]
&=&
-\frac12 \left\{
R_{\frac12,1}\left(\tau, \frac{\tau}{4}\right)
-
R_{\frac32,1}\left(\tau, \frac{\tau}{4}\right)\right\}
[\theta_{\frac12,1}- \theta_{-\frac12,1}](\tau, 2z)
\\[1mm]
&=&
-\frac12 q^{\frac{1}{16}}[\theta_{\frac12,1}- \theta_{-\frac12,1}](\tau, 2z)
\end{eqnarray*}}
proving 2).

\medskip

\noindent
3) Letting $\tau \rightarrow 2\tau$ in 2), we have 
{\allowdisplaybreaks
\begin{eqnarray*}
& &
\Phi^{[1, \frac12]}\left(2\tau, z+\frac{\tau}{2}, z-\frac{\tau}{2}, 0\right)
=
-\frac12 q^{\frac{1}{8}}[\theta_{\frac12,1}- \theta_{-\frac12,1}](2\tau, 2z) \, .
\end{eqnarray*}}
Since 
%(label=n3:eqn:2022-119g)
\begin{equation}
[\theta_{\frac12,1}- \theta_{-\frac12,1}](2\tau, 2z)
= 
[\theta_{1,2}- \theta_{-1,2}](\tau, z)
= -i \vartheta_{11}(\tau,z) \, ,
\label{n3:eqn:2022-119g}
\end{equation}
we obtain 3). \,\ 4) follows from 3) immediately.
\end{proof}

%(line=2516)
%(label=n3:lemma:2022-116d)
\begin{lemma} \,\
\label{n3:lemma:2022-116d}
\begin{enumerate}
\item[{\rm 1)}] $R_{j,1}\Big(\tau, \, \dfrac{\tau}{4}-\dfrac12\Big) 
\,\ = \,\ \left\{
\begin{array}{ccc}
-iq^{\frac{1}{16}} & & \text{if} \quad j=\frac12 \\[2mm]
0 & & \text{if} \quad j=\frac32
\end{array} \right. $
\item[{\rm 2)}] $\Phi^{[1, \frac12]}_{\rm add}
\Big(\tau, z+\dfrac{\tau}{4}-\dfrac12,  z-\dfrac{\tau}{4}+\dfrac12, 0\Big)
\,\ = \,\ 
\dfrac{i}{2} \, q^{\frac{1}{16}} \, 
\big[\theta_{\frac12,1}-\theta_{-\frac12,1}\big](\tau, 2z)$
\item[{\rm 3)}] $\Phi^{[1, \frac12]}_{\rm add}
\Big(2\tau, z+\dfrac{\tau}{2}-\dfrac12,  z-\dfrac{\tau}{2}+\dfrac12, 0\Big)
\,\ = \,\ 
\dfrac12 \, q^{\frac{1}{8}} \, \vartheta_{11}(\tau,z)$
\item[{\rm 4)}] $\Phi^{[1, \frac12]}_{\rm add}
\Big(2\tau, z+\dfrac{\tau}{2}-\dfrac12,  z-\dfrac{\tau}{2}+\dfrac12, 
\dfrac{\tau}{8}\Big)
\,\ = \,\ 
\dfrac12 \, \vartheta_{11}(\tau,z)$
\end{enumerate}
\end{lemma}

\begin{proof} 1) By Lemma \ref{n3:lemma:2022-116c} and 
Lemma \ref{n3:lemma:2022-115b}, we have
$$
R_{j,1}\left(\tau, \, \frac{\tau}{4}-\frac12\right) = e^{-\pi ij}
R_{j,1}\left(\tau, \, \frac{\tau}{4}\right) 
= 
\left\{
\begin{array}{ccc}
-iq^{\frac{1}{16}} & & \text{if} \quad j=\frac12 \\[2mm]
0 & & \text{if} \quad j=\frac32
\end{array} \right.
$$
proving 1). Proof of 2) and 3) and 4) is obtained by similar arguments 
as in the proof of Lemma \ref{n3:lemma:2022-115b}.
\end{proof}

Using these Lemmas, we obtain $\Phi^{[1, \frac12]}$ and $\Phi^{[1, \frac32]}$ 
as follows:

%(line=2493)
%(label=n3:prop:2022-116b)
\begin{prop} \,\
\label{n3:prop:2022-116b}
\begin{enumerate}
\item[{\rm 1)}]
\begin{enumerate}
\item[{\rm (i)}] $\Phi^{[1, \frac12]}
\Big(2\tau, z+ \dfrac{\tau}{2}, z- \dfrac{\tau}{2}, \dfrac{\tau}{8}\Big) 
\, = \,
- \, \dfrac{i}{2} \, \big\{
g_3^{(-)}(\tau, z) \, + \, \vartheta_{11}(\tau, z)\big\}$
$$
= \,\ - \, \frac{i}{2} \, \bigg\{
\frac{\eta(\tau)^5}{\eta(\frac{\tau}{2})^2 \, \eta(2\tau)^2} \cdot 
\frac{\vartheta_{11}(\tau,z) \, \vartheta_{00}(\tau,z)
}{\vartheta_{01}(\tau,z)}
\,\ + \,\ \vartheta_{11}(\tau, z)\bigg\}
$$
\item[{\rm (ii)}] $\Phi^{[1, \frac12]}
\Big(2\tau, z+ \dfrac{\tau}{2}-\dfrac12,  
z- \dfrac{\tau}{2}+\dfrac12, \dfrac{\tau}{8}\Big) 
\, = \,
- \, \dfrac{1}{2} \, \big\{
g_2^{(-)}(\tau, z) \, + \, \vartheta_{11}(\tau, z)\big\}$
$$
= \,\ - \, \frac{1}{2} \, \bigg\{
\frac{\eta(\frac{\tau}{2})^2}{\eta(\tau)} \cdot 
\frac{
\vartheta_{11}(\tau,z) \, \vartheta_{01}(\tau,z)}{\vartheta_{00}(\tau,z)}
\,\ + \,\ \vartheta_{11}(\tau, z)\bigg\}
$$
\end{enumerate}

\item[{\rm 2)}] 
\begin{enumerate}
\item[{\rm (i)}] $\Phi^{[1, \frac32]}
\Big(2\tau, z+ \dfrac{\tau}{2}, z- \dfrac{\tau}{2}, \dfrac{\tau}{8}\Big) 
\, = \,
- \, \dfrac{i}{2} \,  \big\{
g_3^{(-)}(\tau, z) \, - \, \vartheta_{11}(\tau, z)\big\}$
$$
= \,\ - \, \frac{i}{2} \,  \bigg\{
\frac{\eta(\tau)^5}{\eta(\frac{\tau}{2})^2 \, \eta(2\tau)^2} \cdot 
\frac{\vartheta_{11}(\tau,z) \, \vartheta_{00}(\tau,z)
}{\vartheta_{01}(\tau,z)}
\,\ - \,\ \vartheta_{11}(\tau, z)\bigg\}
$$

\item[{\rm (ii)}] $\Phi^{[1, \frac32]}
\Big(2\tau, z+ \dfrac{\tau}{2}-\dfrac12, 
z- \dfrac{\tau}{2}+\dfrac12, \dfrac{\tau}{8}\Big) 
\, = \,
- \dfrac{1}{2} \, \big\{
g_2^{(-)}(\tau, z) \, - \, \vartheta_{11}(\tau, z)\big\}$
$$
= \,\ - \, \frac{1}{2} \, \bigg\{
\frac{\eta(\frac{\tau}{2})^2}{\eta(\tau)} \cdot 
\frac{
\vartheta_{11}(\tau,z) \, \vartheta_{01}(\tau,z)}{\vartheta_{00}(\tau,z)}
\,\ - \,\ \vartheta_{11}(\tau, z)\bigg\}
$$
\end{enumerate}
\end{enumerate}
\end{prop}

\begin{proof} 1)  \quad By \eqref{n3:eqn:2022-111k}, we have
\begin{subequations}
{\allowdisplaybreaks
\begin{eqnarray}
& & \hspace{-5mm}
\Phi^{[1, \frac12]}
\left(2\tau, \, z+ \dfrac{\tau}{2}, \, z- \dfrac{\tau}{2}, \, \frac{\tau}{8}\right)
\nonumber
\\[1mm]
&=&
\widetilde{\Phi}^{[1, \frac12]}
\left(2\tau, z+ \dfrac{\tau}{2}, z- \dfrac{\tau}{2}, \frac{\tau}{8}\right)
- 
\Phi^{[1, \frac12]}_{\rm add}
\left(2\tau, z+ \dfrac{\tau}{2}, z- \dfrac{\tau}{2}, \frac{\tau}{8}\right)
\label{n3:eqn:2022-119a}
\\[1mm]
& & \hspace{-5mm}
\Phi^{[1, \frac12]}
\left(2\tau, \, z+ \frac{\tau}{2}-\frac12, \, 
z- \frac{\tau}{2}+\frac12, \, \frac{\tau}{8}\right)
\nonumber
\\[1mm]
&=&
\widetilde{\Phi}^{[1, \frac12]}
\left(2\tau, z+ \dfrac{\tau}{2}-\frac12, 
z- \dfrac{\tau}{2}+\frac12, \frac{\tau}{8}\right)
- 
\Phi^{[1, \frac12]}_{\rm add}
\left(2\tau, z+ \dfrac{\tau}{2}-\frac12, 
z- \dfrac{\tau}{2}+\frac12, \frac{\tau}{8}\right)
\nonumber
\\[0mm]
& &
\label{n3:eqn:2022-119b}
\end{eqnarray}}
\end{subequations}
The RHS of these equations are rewitten, by \eqref{n3:eqn:2022-117d2} 
and \eqref{n3:eqn:2022-117d4} and 
Lemma \ref{n3:lemma:2022-108e} and Lemma \ref{n3:lemma:2022-115b} 
and Lemma \ref{n3:lemma:2022-116d} and \eqref{n3:eqn:2022-119c4}
and \eqref{n3:eqn:2022-119c6}, as follows:
{\allowdisplaybreaks
\begin{eqnarray*}
& & \hspace{-5mm}
\text{RHS of \eqref{n3:eqn:2022-119a}} 
= \frac12 \big\{
\overset{\circ}{A}{}^{[2]}_4(\tau,z)-\overset{\circ}{A}{}^{[2]}_3(\tau,z)\big\}
- \frac{i}{2} \vartheta_{11}(\tau,z)
\\[0mm]
&=& -\frac{i}{2} g^{(-)}_3(\tau, z)- \frac{i}{2} \vartheta_{11}(\tau,z)
=
-\frac{i}{2} \frac{\eta(\tau)^5}{\eta(\frac{\tau}{2})^2 \eta(2\tau)^2} 
\cdot 
\frac{\vartheta_{11}(\tau,z) \, \vartheta_{00}(\tau,z)}{\vartheta_{01}(\tau,z)}
- \frac{i}{2} \vartheta_{11}(\tau,z)
\\[1mm]
& &\hspace{-5mm}
\text{RHS of \eqref{n3:eqn:2022-119b}} 
= \frac12 \big\{
\overset{\circ}{A}{}^{[2]}_2(\tau,z)-\overset{\circ}{A}{}^{[2]}_1(\tau,z)\big\}
-\frac12 \vartheta_{11}(\tau,z)
\\[0mm]
&=&-\frac12g^{(-)}_2(\tau, z)-\frac12 \vartheta_{11}(\tau,z)
= 
- \frac12 \frac{\eta(\frac{\tau}{2})^2}{\eta(\tau)} \cdot 
\frac{\vartheta_{11}(\tau,z) \, \vartheta_{01}(\tau,z)}{
\vartheta_{00}(\tau,z)}
-\frac12 \vartheta_{11}(\tau,z)
\end{eqnarray*}}
Thus we have proved 1).

\medskip
\noindent
2) Letting $m=1$, $s=\frac12$, $a=1$ and $\tau \rightarrow 2\tau$ in 
Lemma \ref{n3:lemma:2021-1213a}, we have 
$$
\Big[\Phi^{[1, \frac12]}-\Phi^{[1,\frac32]}\Big](2\tau, z_1, z_2,t)
= 
e^{-2\pi it} e^{\frac{\pi i}{2} (z_1-z_2)} 
q^{2(-\frac14) \, (\frac12)^2}
\big[\theta_{\frac12, 1}-\theta_{-\frac12, 1}\big]
(2\tau, z_1+z_2)
$$

\noindent
Letting $t=\frac{\tau}{8}$, we have
%(label=n3:eqn:2022-119d)
\begin{equation}
\Big[\Phi^{[1, \frac12]}-\Phi^{[1,\frac32]}\Big]
\left(2\tau, z_1, z_2,\frac{\tau}{8}\right)
=
e^{\frac{\pi i}{2} (z_1-z_2)} q^{-\frac14}  
\big[\theta_{\frac12, 1}-\theta_{-\frac12, 1}\big]
(2\tau, z_1+z_2)
\label{n3:eqn:2022-119d}
\end{equation}

\noindent
In this equation \eqref{n3:eqn:2022-119d}, we let 
$\left\{
\begin{array}{lcl}
z_1 &=& z+\frac{\tau}{2} \\[1mm]
z_2 &=& z-\frac{\tau}{2}
\end{array}\right.  \,\ {\rm and} \,\ \left\{
\begin{array}{lcl}
z_1 &=& z+\frac{\tau}{2}-\frac12 \\[1mm]
z_2 &=& z-\frac{\tau}{2}+\frac12
\end{array}\right. $ .

\noindent
Then, since $\left\{
\begin{array}{ccc}
z_1-z_2 &=& \tau \\[1mm]
z_1+z_2 &=& 2z
\end{array}\right. \,\ {\rm and} \,\ \left\{
\begin{array}{ccc}
z_1-z_2 &=& \tau-1 \\[1mm]
z_1+z_2 &=& 2z
\end{array}\right. $ respectively, we have 
\begin{subequations}
{\allowdisplaybreaks
\begin{eqnarray}
& &\hspace{-10mm}
\Big[\Phi^{[1, \frac12]}-\Phi^{[1,\frac32]}\Big]
\Big(2\tau, \, z+\frac{\tau}{2}, \, z-\frac{\tau}{2}, \, \frac{\tau}{8}\Big)
\nonumber
\\[1mm]
&=&
e^{\frac{\pi i}{2} \tau} q^{-\frac14}  
\big[\theta_{\frac12, 1}-\theta_{-\frac12, 1}\big](2\tau, 2z)
\,\ = \,\ -i\vartheta_{11}(\tau,z)
\label{n3:eqn:2022-119e}
\\[1mm]
& & \hspace{-10mm}
\Big[\Phi^{[1, \frac12]}-\Phi^{[1,\frac32]}\Big]
\Big(2\tau, \, z+\frac{\tau}{2}-\frac12, \, 
z-\frac{\tau}{2}+\frac12, \, \frac{\tau}{8}\Big)
\nonumber
\\[1mm]
&=&
e^{\frac{\pi i}{2} (\tau-1)} q^{-\frac14}  
\big[\theta_{\frac12, 1}-\theta_{-\frac12, 1}\big](2\tau, 2z)
\,\ = \,\ -\vartheta_{11}(\tau,z)
\label{n3:eqn:2022-119f}
\end{eqnarray}}
\end{subequations}
by \eqref{n3:eqn:2022-119g}.
Now the claim 2) follows from 1) and \eqref{n3:eqn:2022-119e}
and \eqref{n3:eqn:2022-119f}.
\end{proof}

Then we can obtain the honest characters of the N=3 modules $H(\Lambda^{[-1, 0]})$
and $H(\Lambda^{[-1, 2]})$ as follows:

%(line=2786)
%(label=n3:thm:2022-116a)
\begin{thm} \,\ 
\label{n3:thm:2022-116a}
\begin{enumerate}
\item[{\rm 1)}]
\begin{enumerate}
\item[{\rm (i)}] \, ${\rm ch}^{(+)}_{H(\Lambda^{[-1, 0]})}(\tau, z) $
$$
= \,\ 
- \, \frac12 \, \left\{
\frac{\eta(\frac{\tau}{2})}{\eta(2\tau) \, \eta(\tau)} \,
\vartheta_{01}(\tau, z)
\,\ + \,\ 
\frac{1}{\eta(\frac{\tau}{2}) \, \eta(2\tau)} \, \vartheta_{00}(\tau, z)
\right\} 
$$
\item[{\rm (ii)}] \, ${\rm ch}^{(+)}_{H(\Lambda^{[-1, 2]})}(\tau, z) $
$$
= \, 
- \, \frac12 \, \left\{
\frac{\eta(\frac{\tau}{2})}{\eta(2\tau) \, \eta(\tau)} \,
\vartheta_{01}(\tau, z)
\,\ - \,\ 
\frac{1}{\eta(\frac{\tau}{2}) \, \eta(2\tau)} \, \vartheta_{00}(\tau, z)
\right\} 
$$
\end{enumerate}

\item[{\rm 2)}]
\begin{enumerate}
\item[{\rm (i)}] \, ${\rm ch}^{(-)}_{H(\Lambda^{[-1, 0]})}(\tau, z) $
$$= \,\ 
- \, \frac{i}{2} \,\ \left\{
\frac{\eta(\tau)^2}{\eta(\frac{\tau}{2}) \, \eta(2\tau)^2} \, 
\vartheta_{00}(\tau,z)
\,\ + \,\ 
\frac{\eta(\frac{\tau}{2})}{\eta(\tau)^3} \, \vartheta_{01}(\tau, z)
\right\} 
$$
\item[{\rm (ii)}] \, ${\rm ch}^{(-)}_{H(\Lambda^{[-1, 2]})}(\tau, z) $
$$
= \,\ 
- \, \frac{i}{2} \,\ \left\{
\frac{\eta(\tau)^2}{\eta(\frac{\tau}{2}) \, \eta(2\tau)^2} \, 
\vartheta_{00}(\tau,z)
\,\ - \,\ 
\frac{\eta(\frac{\tau}{2})}{\eta(\tau)^3} \, \vartheta_{01}(\tau, z)
\right\} 
$$
\end{enumerate}
\end{enumerate}
\end{thm}

%(line=2840)
%(label=n3:cor:2022-116a)
\begin{cor} 
\label{n3:cor:2022-116a}
The modified (super)characters $\widetilde{\rm ch}^{(\pm)}_{H(\Lambda^{[-1, 0]})}$
are written by the honest (super)characters as follows:
$$
\widetilde{\rm ch}^{(\pm)}_{H(\Lambda^{[-1, 0]})}(\tau, z)
\, = \, \frac12 \, \big\{{\rm ch}^{(\pm)}_{H(\Lambda^{[-1, 0]})}(\tau, z)
+
{\rm ch}^{(\pm)}_{H(\Lambda^{[-1, 2]})}(\tau, z)\big\}
$$
\end{cor}

\begin{proof} This is obvious from Theorems 
\ref{n3:thm:2022-108a} and \ref{n3:thm:2022-116a}.
\end{proof}

Then, by the above Corollary \ref{n3:cor:2022-116a} and 
Remark \ref{n3:rem:2022-115a}, we obtain the following:

%(line=2861)
%(label=n3:cor:2022-116b)
\begin{cor} 
\label{n3:cor:2022-116b}
The linear space spanned by 
$$
{\rm ch}^{(\pm)}_{H(\Lambda^{[-1,1]})}(\tau, z) , \,\ \big[
{\rm ch}^{(\pm)}_{H(\Lambda^{[-1,0]})}+{\rm ch}^{(\pm)}_{H(\Lambda^{[-1,2]})}
\big](\tau, z), \,\ 
{\rm ch}^{(+){\rm tw}(\sigma_{\pm})}_{H(\Lambda^{[-1,1]})}(\tau, z) 
$$
is $SL_2(\zzz)$-invariant. Namely, though the space of honest characters 
is not $SL_2(\zzz)$-invariant, it contains (non-trivial) 
$SL_2(\zzz)$-invariant subspace.
\end{cor}

\section{Asymptotics of characters in the case $m=2$}
\label{sec:asymptotics}
%(line=2789)

In this section we consider the asymptotic behavior of characters 
of N=3 modules $H(\Lambda^{[-1, m_2]})$ as $\tau \downarrow 0$, 
namely $\tau=iT$  $(T>0)$ and $T \rightarrow 0$. 

%(label=n3:lemma:2022-120a)
\begin{lemma}
\label{n3:lemma:2022-120a}
For $a \in \ccc$, the asymptotics of $\vartheta_{ab}(\tau, a\tau)$
as $\tau \downarrow 0$ are as follows:
\begin{enumerate}
\item[{\rm 1)}] \,\ 
$\vartheta_{00}(\tau, \, a\tau) 
\,\ \overset{\substack{\tau \downarrow 0 \\[0.5mm] }}{\sim} \,\ 
(-i\tau)^{-\frac12}$
\item[{\rm 2)}] \,\ 
$\vartheta_{01}(\tau, \, a\tau) 
\,\ \overset{\substack{\tau \downarrow 0 \\[0.5mm] }}{\sim} \,\ 
(-i\tau)^{-\frac12}
\cdot 2 \cos (a\pi) \,\ e^{-\frac{\pi i}{4\tau}}$
\item[{\rm 3)}] \,\ 
$\vartheta_{10}(\tau, \, a\tau) 
\,\ \overset{\substack{\tau \downarrow 0 \\[0.5mm] }}{\sim} \,\ 
(-i\tau)^{-\frac12}$
\item[{\rm 4)}] \,\ 
$\vartheta_{11}(\tau, \, a\tau) 
\,\ \overset{\substack{\tau \downarrow 0 \\[0.5mm] }}{\sim} \,\ 
- \, (-i\tau)^{-\frac12}
\cdot 2i \sin (a\pi) \,\ e^{-\frac{\pi i}{4\tau}}$
\end{enumerate}
\end{lemma}

\begin{proof} These are obtained easily from the $S$-transformation 
property of $\vartheta_{ab}(\tau, z)$ :
$$
\vartheta_{ab}\Big(-\frac{1}{\tau}, \, \frac{z}{\tau}\Big)
\, = \, 
(-i)^{ab} \, (-i\tau)^{\frac12} \, e^{\frac{\pi iz^2}{\tau}} \, 
\vartheta_{ba}(\tau, z)
$$
and the power series expansion of $\vartheta_{ab}(\tau, z)$:
{\allowdisplaybreaks
\begin{eqnarray*}
\vartheta_{00}(\tau, z) &=& \sum_{n \in \zzz} \, 
e^{2\pi inz} \, q^{\frac{n^2}{2}}
\\[0mm]
\vartheta_{01}(\tau, z) &=& \sum_{n \in \zzz} \, (-1)^n \, 
e^{2\pi inz} \, q^{\frac{n^2}{2}}
\\[0mm]
\vartheta_{10}(\tau, z) &=& \sum_{n \in \zzz} \, 
e^{2\pi i(n+\frac12)z} \, q^{\frac12 (n+\frac12)^2}
\\[0mm]
\vartheta_{11}(\tau, z) &=& i \, \sum_{n \in \zzz} \, (-1)^n \, 
e^{2\pi i(n+\frac12)z} \, q^{\frac12 (n+\frac12)^2}
\end{eqnarray*}}.

\vspace{-10mm}

\end{proof}

Then, by using 
%(label=n3:eqn:2022-120b)
\begin{equation}
\eta(\tau) 
\,\ \overset{\substack{\tau \downarrow 0 \\[0.5mm] }}{\sim} \,\ 
(-i\tau)^{\frac12} \, e^{-\frac{\pi i}{12\tau}}
\label{n3:eqn:2022-120b}
\end{equation}
and Lemma \ref{n3:lemma:2022-120a}, we obtain the asymptotics of 
characters of N=3 modules as follows:

%(line=2954)
%(label=n3:prop:2022-120a)
\begin{prop}
\label{n3:prop:2022-120a}
For $a \in \ccc$, the asymptotics of the honest and modified (super)characters 
are as follows:
\begin{enumerate}
\item[{\rm 1)}] Asymptotics of characters:
\begin{enumerate}
\item[{\rm (i)}] \,\ $\widetilde{\rm ch}{}^{(+)}_{H(\Lambda^{[-1, 0]})}
(\tau, a\tau) 
\,\ \overset{\substack{\tau \downarrow 0 \\[0.5mm] }}{\sim} \,\ 
-2 \cos (a\pi) \, e^{-\frac{7}{24} \cdot \frac{\pi i}{\tau}}$
\item[{\rm (ii)}] \,\ ${\rm ch}{}^{(+)}_{H(\Lambda^{[-1, 0]})}
(\tau, a\tau) 
\,\ \overset{\substack{\tau \downarrow 0 \\[0.5mm] }}{\sim} \,\ 
- \dfrac12 \, (-i\tau)^{\frac12} \, e^{\frac{5}{24} \cdot \frac{\pi i}{\tau}}$
\item[{\rm (iii)}] \,\ ${\rm ch}{}^{(+)}_{H(\Lambda^{[-1, 2]})}
(\tau, a\tau) 
\,\ \overset{\substack{\tau \downarrow 0 \\[0.5mm] }}{\sim} \,\ 
\dfrac12 \, (-i\tau)^{\frac12} \, e^{\frac{5}{24} \cdot \frac{\pi i}{\tau}}$
\item[{\rm (iv)}] \,\ ${\rm ch}{}^{(+)}_{H(\Lambda^{[-1, 1]})}
(\tau, a\tau) 
\,\ \overset{\substack{\tau \downarrow 0 \\[0.5mm] }}{\sim} \,\ 
\dfrac{i}{2} \, e^{\frac{5}{24} \cdot \frac{\pi i}{\tau}}$
\end{enumerate}

\item[{\rm 2)}] Asymptotics of super-characters:
\begin{enumerate}
\item[{\rm (i)}] \,\ $\widetilde{\rm ch}{}^{(-)}_{H(\Lambda^{[-1, 0]})}
(\tau, a\tau) 
\,\ \overset{\substack{\tau \downarrow 0 \\[0.5mm] }}{\sim} \,\ 
- 
\dfrac{i}{\sqrt{2}} \, e^{\frac{1}{12} \cdot \frac{\pi i}{\tau}}$
\item[{\rm (ii)}] \,\ ${\rm ch}{}^{(-)}_{H(\Lambda^{[-1, m_2]})}
(\tau, a\tau) 
\,\ \overset{\substack{\tau \downarrow 0 \\[0.5mm] }}{\sim} \,\ 
- 
\dfrac{i}{\sqrt{2}} \, e^{\frac{1}{12} \cdot \frac{\pi i}{\tau}}
\qquad (m_2=0, 1, 2)$
\end{enumerate}
\end{enumerate}
\end{prop}

\begin{proof} First we note that \eqref{n3:eqn:2022-120b} gives 
the following asymptotics:
\begin{subequations}
\begin{equation}
\left\{
\begin{array}{lcl}
\dfrac{\eta(\frac{\tau}{2})}{\eta(2\tau)\eta(\tau)} 
&\overset{\substack{\tau \downarrow 0 \\[1mm]}}{\sim}& 
2 \, (-i\tau)^{\frac12} \, e^{-\frac{1}{24} \cdot \frac{\pi i}{\tau}}
% 2
\\[4mm]
\dfrac{\eta(2\tau)}{\eta(\frac{\tau}{2})\eta(\tau)} 
&\overset{\substack{\tau \downarrow 0 \\[1mm]}}{\sim}& 
\dfrac12 \, (-i\tau)^{\frac12} \, e^{\frac{5}{24} \cdot \frac{\pi i}{\tau}}
% 3
\\[4mm]
\dfrac{\eta(\tau)^2}{\eta(\frac{\tau}{2})\eta(2\tau)^2} 
&\overset{\substack{\tau \downarrow 0 \\[1mm]}}{\sim}& 
\sqrt{2} \,\ (-i\tau)^{\frac12} \, e^{\frac{1}{12} \cdot \frac{\pi i}{\tau}}
% 4
\\[4mm]
\dfrac{\eta(\frac{\tau}{2})\eta(2\tau)^2}{\eta(\tau)^4} 
&\overset{\substack{\tau \downarrow 0 \\[1mm]}}{\sim}& 
\dfrac{1}{\sqrt{2}} \,\ 
(-i\tau)^{\frac12} \, e^{\frac{1}{12} \cdot \frac{\pi i}{\tau}}
\end{array} \right.
\label{n3:eqn:2022-120c}
\end{equation}
and
\begin{equation}
\left\{
\begin{array}{lcl}
\dfrac{1}{\eta(\frac{\tau}{2}) \eta(2\tau)}
&\overset{\substack{\tau \downarrow 0 \\[0.5mm]}}{\sim}& 
-i\tau \, e^{\frac{5}{24} \cdot \frac{\pi i}{\tau}}
\\[4mm]
\dfrac{\eta(\frac{\tau}{2})}{\eta(\tau)^3} 
&\overset{\substack{\tau \downarrow 0 \\[0.5mm]}}{\sim}& 
\dfrac{1}{\sqrt{2}} \, (-i\tau) \, 
e^{\frac{1}{12} \cdot \frac{\pi i}{\tau}}
\end{array}\right.
\label{n3:eqn:2022-120d}
\end{equation}
\end{subequations}
Then, by these formulas
\eqref{n3:eqn:2022-120c} and \eqref{n3:eqn:2022-120d}
and Theorem \ref{n3:thm:2022-108a} and
Lemma \ref{n3:lemma:2022-120a}, we have
{\allowdisplaybreaks
\begin{eqnarray*}
& &\hspace{-10mm}
\widetilde{\rm ch}{}^{(+)}_{H(\Lambda^{[-1, 0]})}(\tau, a\tau) 
= 
- \frac12 \cdot 
\dfrac{\eta(\frac{\tau}{2})}{\eta(2\tau)\eta(\tau)}
\cdot \vartheta_{01}(\tau, a\tau)
\\[1mm]
&\overset{\substack{\tau \downarrow 0 \\[0.5mm] }}{\sim}& 
-\frac12 \cdot 
2 \, (-i\tau)^{\frac12} \, e^{-\frac{1}{24} \cdot \frac{\pi i}{\tau}}
\cdot 
(-i\tau)^{-\frac12} \cdot 2\cos(a\pi) e^{-\frac{\pi i}{4\tau}}
\, = \, 
-2 \cos (a\pi) \, e^{-\frac{7}{24} \cdot \frac{\pi i}{\tau}}
% (+) 1
\\[1mm]
& &\hspace{-10mm}
{\rm ch}{}^{(+)}_{H(\Lambda^{[-1, 1]})}(\tau, a\tau) 
= 
i \, \dfrac{\eta(2\tau)}{\eta(\frac{\tau}{2})\eta(\tau)} 
\cdot \vartheta_{10}(\tau, a\tau)
\\[1mm]
&\overset{\substack{\tau \downarrow 0 \\[0.5mm] }}{\sim}& 
i \cdot 
\dfrac12 \, (-i\tau)^{\frac12} \, e^{\frac{5}{24} \cdot \frac{\pi i}{\tau}}
\cdot (-i\tau)^{-\frac12}
\,\ = \,\ 
\frac{i}{2} \, e^{\frac{5}{24} \cdot \frac{\pi i}{\tau}}
% (-) 0
\\[1mm]
& &\hspace{-10mm}
\widetilde{\rm ch}{}^{(-)}_{H(\Lambda^{[-1, 0]})}(\tau, a\tau) 
= 
- \dfrac{i}{2} \cdot 
\dfrac{\eta(\tau)^2}{\eta(\frac{\tau}{2})\eta(2\tau)^2} 
\cdot \vartheta_{00}(\tau, a\tau)
\\[1mm]
&\overset{\substack{\tau \downarrow 0 \\[0.5mm] }}{\sim}& 
- \dfrac{i}{2} \cdot 
\sqrt{2} \,\ (-i\tau)^{\frac12} \, e^{\frac{1}{12} \cdot \frac{\pi i}{\tau}} 
\cdot (-i\tau)^{-\frac12}
\,\ = \,\ - 
\frac{i}{\sqrt{2}} \, e^{\frac{1}{12} \cdot \frac{\pi i}{\tau}}
% (-) 1
\\[1mm]
& &\hspace{-10mm}
{\rm ch}{}^{(-)}_{H(\Lambda^{[-1, 1]})}(\tau, a\tau) 
= 
- i \, \dfrac{\eta(\frac{\tau}{2})\eta(2\tau)^2}{\eta(\tau)^4} \cdot 
\vartheta_{10}(\tau, a\tau)
\\[1mm]
&\overset{\substack{\tau \downarrow 0 \\[0.5mm] }}{\sim}& 
-i \cdot 
\dfrac{1}{\sqrt{2}} \,\ 
(-i\tau)^{\frac12} \, e^{\frac{1}{12} \cdot \frac{\pi i}{\tau}} \cdot 
(-i\tau)^{-\frac12}
\,\ = \,\ - 
\frac{i}{\sqrt{2}} \, e^{\frac{1}{12} \cdot \frac{\pi i}{\tau}}
\end{eqnarray*}}
and 
{\allowdisplaybreaks
\begin{eqnarray*}
& & \hspace{-15mm}
\dfrac{1}{\eta(\frac{\tau}{2}) \eta(2\tau)} \, \vartheta_{00}(\tau, a\tau)
\, \overset{\substack{\tau \downarrow 0 \\[0.5mm]}}{\sim} \, 
-i\tau \, e^{\frac{5}{24} \cdot \frac{\pi i}{\tau}} \cdot (-i\tau)^{-\frac12}
\,\ = \,\ 
(-i\tau)^{\frac12} \, e^{\frac{5}{24} \cdot \frac{\pi i}{\tau}}
\\[2mm]
& & \hspace{-15mm}
\dfrac{\eta(\frac{\tau}{2})}{\eta(\tau)^3} \, \vartheta_{01}(\tau, a\tau)
\overset{\substack{\tau \downarrow 0 \\[0.5mm]}}{\sim} 
\dfrac{1}{\sqrt{2}} \, (-i\tau) \, 
e^{\frac{1}{12} \cdot \frac{\pi i}{\tau}} \cdot 
(-i\tau)^{-\frac12} \cdot 2 \cos(a\pi) \, e^{-\frac{\pi i}{4\tau}}
\\[1mm]
& &\hspace{11mm}
= \,\ 
\sqrt{2} \, \cos(a\pi) \, (-i\tau)^{\frac12} \, e^{-\frac{\pi i}{6\tau}}
\end{eqnarray*}}
Then, by Theorem \ref{n3:thm:2022-116a}, the asymptotics of honest 
(super)characters are as follows:
{\allowdisplaybreaks
\begin{eqnarray*}
& &
{\rm ch}^{(+)}_{H(\Lambda^{[-1, 0]})}(\tau, a\tau) 
\,\ \overset{\substack{\tau \downarrow 0 \\[0.5mm]}}{\sim} \,\
- \frac12 \, (-i\tau)^{\frac12} \, e^{\frac{5}{24} \cdot \frac{\pi i}{\tau}}
\\[1mm]
& &
{\rm ch}^{(+)}_{H(\Lambda^{[-1, 2]})}(\tau, a\tau) 
\,\ \overset{\substack{\tau \downarrow 0 \\[0.5mm]}}{\sim} \,\
\frac12 \, (-i\tau)^{\frac12} \, e^{\frac{5}{24} \cdot \frac{\pi i}{\tau}}
\\[1mm]
& &
{\rm ch}^{(-)}_{H(\Lambda^{[-1, m_2]})}(\tau, a\tau) 
\,\ \overset{\substack{\tau \downarrow 0 \\[0.5mm]}}{\sim} \,\
- 
\frac{i}{\sqrt{2}} \, e^{\frac{1}{12} \cdot \frac{\pi i}{\tau}}
\qquad (m_2=0, 2) .
\end{eqnarray*}}
Thus the proof of Proposition is completed.
\end{proof}

\section{Modified characters in the case $m=4$}
\label{sec:m4:modified}
%(label=sec:m4:modified)

In this section, we consider the relations of modified characters
between $m$ and $2m$.

%(line=3068)
%(label=n3:prop:2022-117a)
%(label=n3:eqn:2022-202a1) $\sim$ (label=n3:eqn:2022-202a5)
\begin{prop} 
\label{n3:prop:2022-117a}
$\overset{\circ}{A}{}^{[m]}_j(\tau,z)$'s and 
$\overset{\circ}{A}{}^{[2m]}_j(\tau,z)$'s are related to each other 
by the following formulas:
\begin{subequations}
{\allowdisplaybreaks
\begin{eqnarray}
2 \, \overset{\circ}{A}{}^{[m]}_3(2\tau, \, 2z) 
&=& 
\overset{\circ}{A}{}^{[2m]}_1(\tau, z) 
\,\ + \,\ \overset{\circ}{A}{}^{[2m]}_2(\tau, z)
\label{n3:eqn:2022-202a1}
\\[1mm]
\overset{\circ}{A}{}^{[m]}_6\Big(\dfrac{\tau}{2}, \, z\Big)
&=& 
\overset{\circ}{A}{}^{[2m]}_1(\tau,z)
\,\ + \,\ e^{- \pi im} \, \overset{\circ}{A}{}^{[2m]}_2(\tau,z)
\label{n3:eqn:2022-202a2}
\\[1mm]
2 \, \overset{\circ}{A}{}^{[m]}_4(2\tau, 2z)
&=& 
\overset{\circ}{A}{}^{[2m]}_3(\tau, z)
\,\ + \,\ 
\overset{\circ}{A}{}^{[2m]}_4(\tau, z)
\label{n3:eqn:2022-202a3}
\\[1mm]
e^{\frac{\pi im}{4}} \, 
\overset{\circ}{A}{}^{[m]}_6\Big(\dfrac{\tau+1}{2}, \, z\Big) 
&=&
\overset{\circ}{A}{}^{[2m]}_3(\tau, z)
\,\ + \,\ 
e^{-\pi im} \, \overset{\circ}{A}{}^{[2m]}_4(\tau, z)
\label{n3:eqn:2022-202a4}
\\[1mm]
\overset{\circ}{A}{}^{[m]}_5\Big(\dfrac{\tau}{2}, \, z\Big)
&=& 
\overset{\circ}{A}{}^{[2m]}_5(\tau,z)
\,\ + \,\ e^{\pi im} \, \overset{\circ}{A}{}^{[2m]}_6(\tau,z)
\label{n3:eqn:2022-202a5}
\\[1mm]
\overset{\circ}{A}{}^{[m]}_6
\Big(-\dfrac{1}{2\tau}+\dfrac12, \,\ \dfrac{z}{\tau}\Big) 
&=& \tau \, e^{\frac{\pi im}{2}} e^{\frac{\pi imz^2}{\tau}} 
%\nonumber
%\\[1mm]
%& & \hspace{-15mm}
%\times \, 
\big\{
\overset{\circ}{A}{}^{[2m]}_5(\tau,z)
+ 
\overset{\circ}{A}{}^{[2m]}_6(\tau,z)\big\}
\label{n3:eqn:2022-202a6}
\end{eqnarray}}
\end{subequations}
\end{prop}

\begin{proof} First look at the formula for 
$\overset{\circ}{A}{}^{[m]}_2(\tau,z)$ in 
\eqref{n3:eqn:2022-111d2} and \eqref{n3:eqn:2022-108b} :
$$
\overset{\circ}{A}{}^{[m]}_2(\tau,z) =
\widetilde{\Phi}^{[m,0]}\left(\tau, \,\ 
\frac{z}{2}+\frac{\tau}{4}-\frac14, \,\ 
\frac{z}{2}-\frac{\tau}{4}+\frac14, \,\ \frac{\tau}{16}\right) \, .
$$
Letting $\tau \rightarrow \tau-1$, we have
{\allowdisplaybreaks
\begin{eqnarray*}
\overset{\circ}{A}{}^{[m]}_2(\tau-1,z) 
&=&
\widetilde{\Phi}^{[m,0]}\left(\tau, \,\ 
\frac{z}{2}+\frac{\tau}{4}-\frac12, \,\ 
\frac{z}{2}-\frac{\tau}{4}+\frac12, \,\ \frac{\tau-1}{16}\right)
\\[0mm]
&=&
e^{\frac{\pi im}{8}}
\widetilde{\Phi}^{[m,0]}\left(\tau, \,\ 
\frac{z}{2}+\frac{\tau}{4}-\frac12, \,\ 
\frac{z}{2}-\frac{\tau}{4}+\frac12, \,\ \frac{\tau}{16}\right)
\end{eqnarray*}}
Letting $\tau \rightarrow 2\tau$ and $z \rightarrow 2z$, we have 
\begin{subequations}
\begin{equation}
\overset{\circ}{A}{}^{[m]}_2(2\tau-1,2z) 
= 
e^{\frac{\pi im}{8}}
\widetilde{\Phi}^{[m,0]}\left(2\tau, \,\ 
\frac{z}{2}+\frac{\tau}{2}-\frac12, \,\ 
\frac{z}{2}-\frac{\tau}{2}+\frac12, \,\ \frac{\tau}{8}\right)
\label{n3:eqn:2022-117a}
\end{equation}
By Lemma \ref{n3:lemma:2022-111c} and the formulas 
\eqref{n3:eqn:2022-111d1} and \eqref{n3:eqn:2022-111d2}, 
the RHS of this equation is written as follows:
\begin{equation}
\text{RHS of \eqref{n3:eqn:2022-117a}}= e^{\frac{\pi im}{8}} \cdot 
\frac12 \, \Big\{
\overset{\circ}{A}{}^{[2m]}_1(\tau,z) \, + \, 
\overset{\circ}{A}{}^{[2m]}_2(\tau,z)\Big\}
\label{n3:eqn:2022-117b}
\end{equation}
The LHS of \eqref{n3:eqn:2022-117a} is written, by using the $T$-transformation
properties in Lemma \ref{n3:lemma:2022-108c}, as follows:
\begin{equation}
\text{LHS of \eqref{n3:eqn:2022-117a}}= 
e^{\frac{\pi im}{8}}\overset{\circ}{A}{}^{[m]}_3(2\tau,2z) 
\label{n3:eqn:2022-117c}
\end{equation}
\end{subequations}
Then by \eqref{n3:eqn:2022-117a} and \eqref{n3:eqn:2022-117b}
and \eqref{n3:eqn:2022-117c}, we have 
$$
\overset{\circ}{A}{}^{[m]}_3(2\tau,2z) \, = \, 
\frac12 \, \Big\{
\overset{\circ}{A}{}^{[2m]}_1(\tau,z) \, + \, 
\overset{\circ}{A}{}^{[2m]}_2(\tau,z)\Big\} \, ,
$$
proving \eqref{n3:eqn:2022-202a1}.
Then \eqref{n3:eqn:2022-202a2} is obtained by applying $S$-transformation 
to \eqref{n3:eqn:2022-202a1}.
\eqref{n3:eqn:2022-202a3} is obtained by applying $T$-transformation 
to \eqref{n3:eqn:2022-202a1}.
\eqref{n3:eqn:2022-202a4} is obtained by applying $T$-transformation 
to \eqref{n3:eqn:2022-202a2}.
\eqref{n3:eqn:2022-202a5} is obtained by applying $S$-transformation 
to \eqref{n3:eqn:2022-202a3}.
\eqref{n3:eqn:2022-202a6} is obtained by applying $S$-transformation 
to \eqref{n3:eqn:2022-202a4}.
\end{proof}

Writing the formulas \eqref{n3:eqn:2022-202a1} $\sim$ 
\eqref{n3:eqn:2022-202a5} in the case $m=2$, we obtain the following:

%(line=3201)
%(label=n3:lemma:2022-202d)
\begin{lemma} 
\label{n3:lemma:2022-202d}
$\overset{\circ}{A}{}^{[4]}_j$'s are related to 
$\overset{\circ}{A}{}^{[2]}_j$'s as follows: 
\begin{enumerate}
\item[{\rm 1)}] \quad $2 \, \overset{\circ}{A}{}^{[2]}_3(2\tau, \, 2z) 
\,\ = \,\ 
\overset{\circ}{A}{}^{[4]}_1(\tau, z) 
\,\ + \,\ 
\overset{\circ}{A}{}^{[4]}_2(\tau, z)$
\item[{\rm 2)}] \quad $2 \, \overset{\circ}{A}{}^{[2]}_4(2\tau, 2z)
\,\ = \,\ 
\overset{\circ}{A}{}^{[4]}_3(\tau, z)
\,\ + \,\ 
\overset{\circ}{A}{}^{[4]}_4(\tau, z)$
\item[{\rm 3)}] \quad $\overset{\circ}{A}{}^{[2]}_5\Big(\dfrac{\tau}{2}, \, z\Big)
\,\ = \,\ 
\overset{\circ}{A}{}^{[4]}_5(\tau,z)
\,\ + \,\ 
\overset{\circ}{A}{}^{[4]}_6(\tau,z)$
\item[{\rm 4)}] \quad $2 \, \overset{\circ}{A}{}^{[4]}_3(2\tau, 2z) 
\,\ = \,\ 
\overset{\circ}{A}{}^{[2]}_6\Big(\dfrac{\tau}{2}, \, z\Big) $
\end{enumerate}
\end{lemma}

Then by Corollary \ref{n3:cor:2022-112a}, we obtain the explicit formula 
for $\widetilde{\rm ch}^{(\pm)}_{H(\Lambda^{[K(4),1]})}(\tau, z)$ as follows: 

%(line=3235)
%(label=n3:prop:2022-204b)
\begin{prop} \,\ 
\label{n3:prop:2022-204b}
\begin{enumerate}
\item[{\rm 1)}]
$\widetilde{\rm ch}^{(+)}_{H(\Lambda^{[K(4),1]})}(\tau, z)$
$$ \hspace{-8mm}
= \,\ 
\frac{i}{2} \,\ 
\frac{1}{\eta(\frac{\tau}{2})\eta(2\tau)} \cdot 
\frac{\vartheta_{10}(\tau, z)}{\vartheta_{01}(\tau, z)} \, 
\bigg\{
\frac{\eta(2\tau)^5}{\eta(\tau)^2\eta(4\tau)^2} \cdot 
\vartheta_{00}(2\tau, 2z)
\, - \, 
2 \, \frac{\eta(4\tau)^2}{\eta(2\tau)} \cdot 
\vartheta_{10}(2\tau, 2z)
\bigg\}
$$
\item[{\rm 2)}]
$\widetilde{\rm ch}^{(-)}_{H(\Lambda^{[K(4),1]})}(\tau, z)$
$$ \hspace{-8mm}
= \,\ - \, 
\frac{i}{2} \,\ 
\frac{\eta(\frac{\tau}{2})}{\eta(\tau)^3} \cdot 
\frac{\vartheta_{10}(\tau, z)}{\vartheta_{00}(\tau, z)} \, 
\bigg\{
\frac{\eta(2\tau)^5}{\eta(\tau)^2\eta(4\tau)^2} \cdot 
\vartheta_{00}(2\tau, 2z)
\, + \, 
2 \, \frac{\eta(4\tau)^2}{\eta(2\tau)} \cdot 
\vartheta_{10}(2\tau, 2z)
\bigg\}
$$
\end{enumerate}
\end{prop}

\begin{proof} 1) We compute 
$\overset{\circ}{A}{}^{[4]}_2(\tau, z)+\overset{\circ}{A}{}^{[4]}_1(\tau, z)$
by using Lemma \ref{n3:lemma:2022-202d}: 
{\allowdisplaybreaks
\begin{eqnarray*}
& & \hspace{-7mm}
\overset{\circ}{A}{}^{[4]}_2(\tau, z)+\overset{\circ}{A}{}^{[4]}_1(\tau, z)
\, = \, 
2  \overset{\circ}{A}{}^{[2]}_3(2\tau,  2z) 
=
i \, \Big\{- 2  g^{(+)}_1(2\tau, 2z) + g^{(-)}_3(2\tau, 2z)\Big\}
\\[1mm]
&=&
i \, \bigg\{-2 \, \frac{\eta(4\tau)^2}{\eta(2\tau)} \cdot 
\frac{\vartheta_{11}(2\tau, 2z)\vartheta_{10}(2\tau, 2z)
}{\vartheta_{01}(2\tau, 2z)}
\\[1mm]
& &\hspace{10mm}
+ \,\ 
\frac{\eta(2\tau)^5}{\eta(\tau)^2\eta(4\tau)^2} \cdot 
\frac{\vartheta_{11}(2\tau, 2z)\vartheta_{00}(2\tau, 2z)
}{\vartheta_{01}(2\tau, 2z)}
\bigg\}
\\[2mm]
&=&
i \, \frac{\vartheta_{11}(2\tau, 2z)}{\vartheta_{01}(2\tau, 2z)} \, 
\bigg\{
-2 \, \frac{\eta(4\tau)^2}{\eta(2\tau)} \cdot 
\vartheta_{10}(2\tau, 2z)
\, + \, 
\frac{\eta(2\tau)^5}{\eta(\tau)^2\eta(4\tau)^2} \cdot 
\vartheta_{00}(2\tau, 2z)
\bigg\}
\\[2mm]
&=&
i \frac{\vartheta_{11}(\tau, z) \vartheta_{10}(\tau, z)
}{\vartheta_{01}(\tau, z) \vartheta_{00}(\tau, z)} 
\bigg\{
-2 \frac{\eta(4\tau)^2}{\eta(2\tau)} \cdot 
\vartheta_{10}(2\tau, 2z)
+ 
\frac{\eta(2\tau)^5}{\eta(\tau)^2\eta(4\tau)^2} \cdot 
\vartheta_{00}(2\tau, 2z)
\bigg\}
\end{eqnarray*}}

\noindent
Then, by Corollary \ref{n3:cor:2022-112a}, we have
{\allowdisplaybreaks
\begin{eqnarray*}
& & \hspace{-10mm}
\widetilde{\rm ch}^{(+)}_{H(\Lambda^{[K(4),1]})}(\tau, z) \,\ = \,\ 
\frac{1}{\overset{N=3}{R}{}^{(+)}(\tau, z)} \cdot 
\frac12 \cdot \big[
\overset{\circ}{A}{}^{[4]}_2 \, + \, \overset{\circ}{A}{}^{[4]}_1\big](\tau, z)
\\[1mm]
&=&
\frac{i}{2} \,\ 
\frac{1}{\eta(\frac{\tau}{2})\eta(2\tau)} \cdot 
\frac{\vartheta_{00}(\tau, z)}{\vartheta_{11}(\tau, z)} \cdot
\frac{\vartheta_{11}(\tau, z) \, \vartheta_{10}(\tau, z)
}{\vartheta_{01}(\tau, z) \, \vartheta_{00}(\tau, z)}
\\[2mm]
& &
\times \,\ \bigg\{
-2 \, \frac{\eta(4\tau)^2}{\eta(2\tau)} \cdot 
\vartheta_{10}(2\tau, 2z)
\,\ + \,\ 
\frac{\eta(2\tau)^5}{\eta(\tau)^2\eta(4\tau)^2} \cdot 
\vartheta_{00}(2\tau, 2z)
\bigg\}
\\[1mm]
&=&
\frac{i}{2} \,\ 
\frac{1}{\eta(\frac{\tau}{2})\eta(2\tau)} \cdot 
\frac{\vartheta_{10}(\tau, z)}{\vartheta_{01}(\tau, z)} 
\\[1mm]
& & \hspace{10mm}
\times \,\ \bigg\{
-2 \, \frac{\eta(4\tau)^2}{\eta(2\tau)} \cdot 
\vartheta_{10}(2\tau, 2z)
\,\ + \,\ 
\frac{\eta(2\tau)^5}{\eta(\tau)^2\eta(4\tau)^2} \cdot 
\vartheta_{00}(2\tau, 2z)
\bigg\} ,
\end{eqnarray*}}
proving 1).

\medskip

\noindent
2) Next we compute 
$\overset{\circ}{A}{}^{[4]}_3(\tau, z)+\overset{\circ}{A}{}^{[4]}_4(\tau, z)$
by using Lemma \ref{n3:lemma:2022-202d}: 
{\allowdisplaybreaks
\begin{eqnarray*}
& &\hspace{-10mm}
\overset{\circ}{A}{}^{[4]}_3(\tau, z)
+\overset{\circ}{A}{}^{[4]}_4(\tau, z)
\,\ = \,\ 
2 \, \overset{\circ}{A}{}^{[2]}_4(2\tau, \, 2z) 
\, = \, 
- i \, \Big\{2 \, g^{(+)}_1(2\tau, 2z) \, + \, g^{(-)}_3(2\tau, 2z)\Big\}
\\[2mm]
&=&
-i \, \bigg\{2 \, \frac{\eta(4\tau)^2}{\eta(2\tau)} \cdot 
\frac{\vartheta_{11}(2\tau, 2z)\vartheta_{10}(2\tau, 2z)
}{\vartheta_{01}(2\tau, 2z)}
\, + \, 
\frac{\eta(2\tau)^5}{\eta(\tau)^2\eta(4\tau)^2} \cdot 
\frac{\vartheta_{11}(2\tau, 2z)\vartheta_{00}(2\tau, 2z)
}{\vartheta_{01}(2\tau, 2z)}
\bigg\}
\\[2mm]
&=&
-i \,\ \frac{\vartheta_{11}(2\tau, 2z)}{\vartheta_{01}(2\tau, 2z)} \, 
\bigg\{
2 \, \frac{\eta(4\tau)^2}{\eta(2\tau)} \cdot 
\vartheta_{10}(2\tau, 2z)
\,\ + \,\ 
\frac{\eta(2\tau)^5}{\eta(\tau)^2\eta(4\tau)^2} \cdot 
\vartheta_{00}(2\tau, 2z)
\bigg\}
\\[2mm]
&=&
-i \,\  \frac{\vartheta_{11}(\tau, z) \, \vartheta_{10}(\tau, z)
}{\vartheta_{01}(\tau, z) \, \vartheta_{00}(\tau, z)} \, 
\bigg\{
2 \, \frac{\eta(4\tau)^2}{\eta(2\tau)} \cdot 
\vartheta_{10}(2\tau, 2z)
\,\ + \,\ 
\frac{\eta(2\tau)^5}{\eta(\tau)^2\eta(4\tau)^2} \cdot 
\vartheta_{00}(2\tau, 2z)
\bigg\}
\end{eqnarray*}}

\noindent
Then, by Corollary \ref{n3:cor:2022-112a}, we have
{\allowdisplaybreaks
\begin{eqnarray*}
& & \hspace{-10mm}
\widetilde{\rm ch}^{(-)}_{H(\Lambda^{[K(4),1]})}(\tau, z) \,\ = \,\ 
\frac{1}{\overset{N=3}{R}{}^{(-)}(\tau, z)} \cdot 
\frac12 \cdot \big[
\overset{\circ}{A}{}^{[4]}_3 \, + \, \overset{\circ}{A}{}^{[4]}_4\big](\tau, z)
\\[3mm]
&=&- \, 
\frac{i}{2} \,\ 
\frac{\eta(\frac{\tau}{2})}{\eta(\tau)^3} \cdot 
\frac{\vartheta_{01}(\tau, z)}{\vartheta_{11}(\tau, z)} \cdot
\frac{\vartheta_{11}(\tau, z) \, \vartheta_{10}(\tau, z)
}{\vartheta_{01}(\tau, z) \, \vartheta_{00}(\tau, z)}
\\[1mm]
& &
\times \,\ \bigg\{
2 \, \frac{\eta(4\tau)^2}{\eta(2\tau)} \cdot 
\vartheta_{10}(2\tau, 2z)
\,\ + \,\ 
\frac{\eta(2\tau)^5}{\eta(\tau)^2\eta(4\tau)^2} \cdot 
\vartheta_{00}(2\tau, 2z)
\bigg\}
\\[3mm]
&=&- \, 
\frac{i}{2} \,\ 
\frac{\eta(\frac{\tau}{2})}{\eta(\tau)^3} \cdot 
\frac{\vartheta_{10}(\tau, z)}{\vartheta_{00}(\tau, z)} \, 
\bigg\{
2 \, \frac{\eta(4\tau)^2}{\eta(2\tau)} \cdot 
\vartheta_{10}(2\tau, 2z)
\,\ + \,\ 
\frac{\eta(2\tau)^5}{\eta(\tau)^2\eta(4\tau)^2} \cdot 
\vartheta_{00}(2\tau, 2z)
\bigg\}
\end{eqnarray*}}
proving 2).
\end{proof}

%(line=3450)
%(label=n3:cor:2022-205a)
\begin{cor} 
\label{n3:cor:2022-205a}
When $m=4$, the modified character and supercharacter \\
$\widetilde{\rm ch}^{(\pm)}_{H(\Lambda^{[K(4),1]})}(\tau, z)$ 
are holomorphic functions of $(\tau, z) \in \ccc_+ \times \ccc$.
\end{cor}

The modified character 
$\widetilde{\rm ch}^{(+)}_{H(\Lambda^{[K(4),0]})}(\tau, z)$ cannot 
be obtained by the argument in this section, but 
further analysis via another approach suggests the following:

\begin{conj}
\label{n3:conj:2022-625a}
$\widetilde{\rm ch}^{(+)}_{H(\Lambda^{[K(4),0]})}(\tau, z)$
will be written in the form
\begin{equation}
\widetilde{\rm ch}^{(+)}_{H(\Lambda^{[K(4),0]})}(\tau, z)
= \dfrac{1}{\eta(\frac{\tau}{2})\eta(2\tau)} \, \Big\{
C_1(\tau)\vartheta_{01}(\tau,z)^2
+
C_2(\tau)\vartheta_{10}(\tau,z)^2\Big\}
\label{n3:eqn:2022-625a}
\end{equation} 

\noindent
where $C_i(\tau)$ are holomorphic functions in $\tau \in \ccc_+$ 
satisfying $C_1(-\frac{1}{\tau}) \, = \, - \, C_2(\tau)$.
\end{conj}

\section{Honest characters in the case $m=4$}
\label{sec:m4:honest}
%(label=sec:m4:honest)
%(line=3487)

For $m \in \nnn$ and $s \in \frac12 \zzz$, we define the functions 
$P^{[m,s]}(\tau, z)$ and $Q^{[m,s]}(\tau, z)$ as follows:
%(label=n3:eqn:2022-202f1)
%(label=n3:eqn:2022-202f2)
\begin{subequations}
{\allowdisplaybreaks
\begin{eqnarray}
P^{[m,s]}(\tau, z) &:=& 
\Phi^{[\frac{m}{2},s]}\Big(2\tau, \,\ 
z+\frac{\tau}{2}, \,\ 
z-\frac{\tau}{2}, \,\ 0 \Big)
\label{n3:eqn:2022-202f1}
\\[1mm]
Q^{[m,s]}(\tau, z) &:=& 
\Phi^{[\frac{m}{2},s]}\Big(2\tau, \,\ 
z+\frac{\tau}{2}-\frac12, \,\ 
z-\frac{\tau}{2}+\frac12, \,\ 0 \Big)
\label{n3:eqn:2022-202f2}
\end{eqnarray}}
\end{subequations}

%(line=3764)
%(label=n3:lemma:2022-202c)
\begin{lemma} \quad 
\label{n3:lemma:2022-202c}
Let $m \in \nnn$ and $s \in \frac12 \zzz$. Then the following formulas hold:
\begin{enumerate}
\item[{\rm 1)}] $
P^{[m,s+1]}(\tau, z) - P^{[m,s]}(\tau, z) 
\, = \, - \, 
q^{\frac{s}{2}-\frac{s^2}{m}}
\, \big[\theta_{s, \, \frac{m}{2}}-\theta_{-s, \, \frac{m}{2}}
\big](2\tau, 2z)$
\item[{\rm 2)}] $
Q^{[m,s+1]}(\tau, z) \, - \, Q^{[m,s]}(\tau, z) 
\, = \, - \, 
e^{-\pi is} \, q^{\frac{s}{2}-\frac{s^2}{m}}
\, \big[\theta_{s, \, \frac{m}{2}}-\theta_{-s, \, \frac{m}{2}}
\big](2\tau, 2z)$
\end{enumerate}
\end{lemma}

\begin{proof} To prove these formulas we use the following formula
which is obtained from Lemma \ref{n3:lemma:2021-1213a} :
{\allowdisplaybreaks
\begin{eqnarray}
\lefteqn{
\Phi^{[\frac{m}{2},s]}(2\tau, z_1, z_2,0) 
\, - \, \Phi^{[\frac{m}{2},s+1]}(2\tau, z_1, z_2,0)} 
\nonumber
\\[1mm]
&=&
e^{\pi is(z_1-z_2)} \, q^{-\frac{s^2}{m}}
\, \big[\theta_{s, \, \frac{m}{2}}-\theta_{-s, \, \frac{m}{2}}
\big](2\tau, \, z_1+z_2)
\label{n3:eqn:2022-204a}
\end{eqnarray}}
Putting $z_1 = z+\frac{\tau}{2}$ and $z_2 = z-\frac{\tau}{2}$, this formula
\eqref{n3:eqn:2022-204a} gives 
{\allowdisplaybreaks
\begin{eqnarray*}
& & \hspace{-10mm}
\Phi^{[\frac{m}{2};s]}
\Big(2\tau, \, z+\frac{\tau}{2}, \, z-\frac{\tau}{2}, \, 0\Big) 
\, - \, 
\Phi^{[\frac{m}{2};s+1]}
\Big(2\tau, \, z+\frac{\tau}{2}, \, z-\frac{\tau}{2}, \, 0\Big) 
\\[1mm]
&=&
e^{\pi is\tau} \, q^{-\frac{s^2}{m}}
\, \big[\theta_{s, \, \frac{m}{2}}-\theta_{-s, \, \frac{m}{2}}
\big](2\tau, \, 2z)
\end{eqnarray*}}
proving 1). 

\medskip

Next, putting $z_1 = z+\frac{\tau}{2}-\frac12$ and 
$z_2 = z-\frac{\tau}{2}+\frac12$, the formula
\eqref{n3:eqn:2022-204a} gives 
{\allowdisplaybreaks
\begin{eqnarray*}
& & \hspace{-15mm}
\lefteqn{
\Phi^{[\frac{m}{2};s]}\Big(2\tau, \, 
z+\frac{\tau}{2}-\frac12, \, 
z-\frac{\tau}{2}+\frac12, \, 0\Big)
}
\\[1mm]
& &- \,\ 
\Phi^{[\frac{m}{2};s+1]}\Big(2\tau, \, 
z+\frac{\tau}{2}-\frac12, \, 
z-\frac{\tau}{2}+\frac12, \, 0\Big) 
\\[1mm]
&=&
e^{-\pi is}q^{\frac{s}{2}} \, q^{-\frac{s^2}{m}}
\, \big[\theta_{s, \, \frac{m}{2}}-\theta_{-s, \, \frac{m}{2}}
\big](2\tau, \, 2z)
\end{eqnarray*}}
proving 2). 
\end{proof}

%(line=3846)
%(label=n3:lemma:2022-202a)
\begin{lemma} \quad 
\label{n3:lemma:2022-202a}
Let $m \in \nnn$ and $s \in \frac12 \zzz$. Then

\begin{enumerate}
\item[{\rm 1)}] \quad $2 \, P^{[m,s]}(2\tau, 2z) \,\ = \,\ 
P^{[2m,2s]}(\tau,z) \,\ + \,\ e^{-2\pi is} \, Q^{[2m,2s]}(\tau,z)$

\item[{\rm 2)}] 
\begin{enumerate}
\item[{\rm (i)}] \quad $P^{[2m,2s]}(\tau,z) \,\ = \,\ 
P^{[m,s]}(2\tau, 2z) \, + \, P^{[m,s+\frac12]}(2\tau, 2z)$
\item[{\rm (ii)}] \quad $Q^{[2m,2s]}(\tau,z) \,\ = \,\ 
e^{2\pi is} \, \big\{
P^{[m,s]}(2\tau, 2z) \, - \, P^{[m,s+\frac12]}(2\tau, 2z)\big\}$
\end{enumerate}
\end{enumerate}
\end{lemma}

\begin{proof} 1) 
By \eqref{n3:eqn:2022-202f1} and Lemma \ref{n3:lemma:2022-111c}, we have
{\allowdisplaybreaks
\begin{eqnarray*}
\lefteqn{
2P^{[m,s]}(2\tau, \, 2z) \,\ = \,\ 2 \, 
\Phi^{[\frac{m}{2},s]}(4\tau, \,\ 2z+\tau, \,\ 2z-\tau, \,\ 0)}
\\[1mm]
&=&
\underbrace{\Phi^{[m,2s]}
\Big(2\tau, z+\frac{\tau}{2}, z-\frac{\tau}{2}, 0\Big)}_{
\substack{|| \\[1mm] {\displaystyle P^{[2m,2s]}(\tau,z)
}}}
+ 
e^{-2\pi is} 
\underbrace{\Phi^{[m,2s]}
\Big(2\tau, z+\frac{\tau}{2}+\frac12, z-\frac{\tau}{2}-\frac12, 0\Big)
}_{\substack{\hspace{6mm} || \,\ put \\[1mm] {\rm (I)}
}}
\end{eqnarray*}}
where (I) is computed by using Lemma \ref{n3:lemma:2022-202b} as follows :
{\allowdisplaybreaks
\begin{eqnarray*}
{\rm (I)} &=& \Phi^{[m,2s]}
\Big(2\tau, 
\Big(z+\frac{\tau}{2}+\frac12\Big)-1,  
\Big(z-\frac{\tau}{2}-\frac12\Big)+1, 0\Big)
\\[1mm]
&=&
\Phi^{(+)[m,2s]}\Big(2\tau, 
z+\frac{\tau}{2}-\frac12,  
z-\frac{\tau}{2}+\frac12,  0\Big) \,\ = \,\ Q^{[2m,2s]}(\tau,z)
\end{eqnarray*}}
Thus we have
%(label=n3:eqn:2022-202g1)
\begin{subequations}
\begin{equation}
2 \, P^{[m,s]}(2\tau, 2z) \,\ = \,\ 
P^{[2m,2s]}(\tau,z) \,\ + \,\ e^{-2\pi is} \, Q^{[2m,2s]}(\tau,z) \, ,
\label{n3:eqn:2022-202g1}
\end{equation}
proving 1). In order to prove 2), we let $s \rightarrow s+\frac12$ in  
\eqref{n3:eqn:2022-202g1}. Then, by using Lemma \ref{n3:lemma:2022-202c}, 
we have
{\allowdisplaybreaks
\begin{eqnarray}
& & \hspace{-10mm}
2 \, P^{[m,s+\frac12]}(2\tau, 2z) \,\ = \,\ 
P^{[2m,2s+1]}(\tau,z) \,\ + \,\ 
e^{-2\pi i(s+\frac12)} \, Q^{[2m,2s+1]}(\tau,z)
\nonumber
\\[1mm]
&= &
\Big\{P^{[2m,2s]}(\tau, z) \, - \, 
q^{\frac{2s}{2}-\frac{(2s)^2}{4m}} \, 
\big[\theta_{2s, \, 2m}-\theta_{-2s, \, 2m}\big](2\tau, \, 2z)\Big\}
\nonumber
\\[0mm]
& &
- \, e^{-2\pi is} \, \Big\{Q^{[2m,2s]}(\tau, z) \, - \, 
e^{-2\pi is} \, 
q^{\frac{2s}{2}-\frac{(2s)^2}{4m}} \, 
\big[\theta_{2s, \, 2m}-\theta_{-2s, \, 2m}\big](2\tau, \, 2z)\Big\}
\nonumber
\\[1mm]
&=&
P^{[2m,2s]}(\tau, z) \, - \, e^{-2\pi is} \, Q^{[2m,2s]}(\tau, z)
\label{n3:eqn:2022-202g2}
\end{eqnarray}}
\end{subequations}
Then by making 
${\rm \eqref{n3:eqn:2022-202g1}} \pm {\rm \eqref{n3:eqn:2022-202g2}}$,
we have
$$\left\{
\begin{array}{ccr}
2 \, P^{[m,s]}(2\tau, 2z) \, + \, 2 \, P^{[m,s+\frac12]}(2\tau, 2z)
&=& 2 \, P^{[2m,2s]}(\tau,z)  \\[3mm]
2 \, P^{[m,s]}(2\tau, 2z) \, - \, 2 \, P^{[m,s+\frac12]}(2\tau, 2z)
&=& 
2 \, e^{-2\pi is} \, Q^{[2m,2s]}(\tau,z)
\end{array}\right. \, ,
$$
so we have
$$
\left\{
\begin{array}{ccr}
P^{[2m,2s]}(\tau,z) &=& P^{[m,s]}(2\tau, 2z) \, + \, P^{[m,s+\frac12]}(2\tau, 2z)
\hspace{2mm}
\\[1mm]
Q^{[2m,2s]}(\tau,z) &=& e^{2\pi is} \, \big\{
P^{[m,s]}(2\tau, 2z) \, - \, P^{[m,s+\frac12]}(2\tau, 2z)\big\}
\end{array}\right. \, 
$$
proving 2).
\end{proof}

In order to compute the honest characters in the case $m=4$, we 
need to know $P^{[4, s]}(\tau, z)$ and $Q^{[4, s]}(\tau, z)$, 
which are obtained from Lemma \ref{n3:lemma:2022-202a} as follows:

%(line=3968)
%(label=n3:lemma:2022-205a)
\begin{lemma} \,\
\label{n3:lemma:2022-205a}
\begin{enumerate}
\item[{\rm 1)}] $P^{[2,s]}(\tau,z)$ \,\ 
$(s \in \frac12\zzz \,\ \text{such that} \,\ 
\frac12 \leq s \leq \frac32)$ are as follows:

\begin{enumerate}
\item[{\rm (i)}] $P^{[2,\frac12]}(\tau, z) 
\,\ = \,\ 
- \, \dfrac{i}{2} \, q^{\frac18} \, \Big\{
g_3^{(-)}(\tau, z) \, + \, \vartheta_{11}(\tau, z)\Big\}$
$$
= \,\ - \, \dfrac{i}{2} \, q^{\frac18} \, \bigg\{
\frac{\eta(\tau)^5}{\eta(\frac{\tau}{2})^2 \, \eta(2\tau)^2} \cdot 
\frac{\vartheta_{11}(\tau,z) \, \vartheta_{00}(\tau,z)
}{\vartheta_{01}(\tau,z)}
\,\ + \,\ \vartheta_{11}(\tau, z)\bigg\}
$$
\item[{\rm (ii)}] $P^{[2,1]}(\tau, z) 
\, = \, 
- i q^{\frac14} \, g_1^{(+)}(\tau,z)
\, = \, 
- i q^{\frac14} \, \dfrac{\eta(2\tau)^2}{\eta(\tau)} \cdot 
\dfrac{\vartheta_{11}(\tau,z) \, \vartheta_{10}(\tau,z)
}{\vartheta_{01}(\tau,z)}$
\item[{\rm (iii)}] $P^{[2,\frac32]}(\tau, z) 
\,\ = \,\ 
- \, \dfrac{i}{2} \, q^{\frac18} \, \Big\{
g_3^{(-)}(\tau, z) \, - \, \vartheta_{11}(\tau, z)\Big\}$
$$
= \,\ - \, \dfrac{i}{2} \, q^{\frac18} \, \bigg\{
\frac{\eta(\tau)^5}{\eta(\frac{\tau}{2})^2 \, \eta(2\tau)^2} \cdot 
\frac{\vartheta_{11}(\tau,z) \, \vartheta_{00}(\tau,z)
}{\vartheta_{01}(\tau,z)}
\,\ - \,\ \vartheta_{11}(\tau, z)\bigg\}
$$
\end{enumerate}

\item[{\rm 2)}] $P^{[4,s]}(\tau,z)$ \quad $(s = 1, 2)$ \,\ 
are as follows :

\begin{enumerate}
\item[{\rm (i)}] $P^{[4,1]}(\tau, z) 
\,\ = \,\ 
- \, \dfrac{i}{2} \, q^{\frac14} \, \Big\{
g^{(-)}_3(2\tau, 2z)
\,\ + \,\ 2 \, g_1^{(+)}(2\tau, 2z)
\,\ + \,\ \vartheta_{11}(2\tau, 2z)\Big\}$
$$ \hspace{-5mm}
= \, 
-  \frac{i}{2} \, q^{\frac14} \, \vartheta_{11}(2\tau, 2z)\bigg\{
\frac{\eta(2\tau)^5}{\eta(\tau)^2 \eta(4\tau)^2} 
\cdot 
\frac{\vartheta_{00}(2\tau,2z)}{\vartheta_{01}(2\tau,2z)}
+ 2 
\frac{\eta(4\tau)^2}{\eta(2\tau)} \cdot 
\frac{\vartheta_{10}(2\tau,2z)}{\vartheta_{01}(2\tau,2z)} 
+ 1\bigg\}
$$
\item[{\rm (ii)}] $P^{[4,2]}(\tau, z) 
\,\ = \,\ 
- \, \dfrac{i}{2} \, q^{\frac14} \, \Big\{
g^{(-)}_3(2\tau, 2z)
\,\ + \,\ 2 \, g_1^{(+)}(2\tau, 2z)
\,\ - \,\ \vartheta_{11}(2\tau, 2z)\Big\}$
$$ \hspace{-5mm}
= \,\ 
- \frac{i}{2} \, q^{\frac14} \, \vartheta_{11}(2\tau, 2z)\bigg\{
\frac{\eta(2\tau)^5}{\eta(\tau)^2 \eta(4\tau)^2} 
\cdot 
\frac{\vartheta_{00}(2\tau,2z)}{\vartheta_{01}(2\tau,2z)}
+ 2 \, 
\frac{\eta(4\tau)^2}{\eta(2\tau)} \cdot 
\frac{\vartheta_{10}(2\tau,2z)}{\vartheta_{01}(2\tau,2z)} 
- 1\bigg\}
$$
\end{enumerate}

\item[{\rm 3)}] $Q^{[4,s]}(\tau,z)$ \quad $(s =1, 2)$ \,\ 
are as follows :

\begin{enumerate}
\item[{\rm (i)}] $Q^{[4,1]}(\tau, z) 
\,\ = \,\ 
\dfrac{i}{2} \, q^{\frac14} \, \Big\{
g^{(-)}_3(2\tau, 2z)
\,\ - \,\ 2 \, g_1^{(+)}(2\tau, 2z)
\,\ + \,\ \vartheta_{11}(2\tau, 2z)\Big\}$
$$ \hspace{-5mm}
= \, 
\frac{i}{2} \, q^{\frac14} \, \vartheta_{11}(2\tau, 2z)\bigg\{
\frac{\eta(2\tau)^5}{\eta(\tau)^2 \eta(4\tau)^2} 
\cdot 
\frac{\vartheta_{00}(2\tau,2z)}{\vartheta_{01}(2\tau,2z)}
- 2 
\frac{\eta(4\tau)^2}{\eta(2\tau)} \cdot 
\frac{\vartheta_{10}(2\tau,2z)}{\vartheta_{01}(2\tau,2z)} 
+ 1\bigg\}
$$
\item[{\rm (ii)}] $Q^{[4,2]}(\tau, z) 
\,\ = \,\ 
\dfrac{i}{2} \, q^{\frac14} \, \Big\{
g^{(-)}_3(2\tau, 2z)
\,\ - \,\ 2 \, g_1^{(+)}(2\tau, 2z)
\,\ - \,\ \vartheta_{11}(2\tau, 2z)\Big\}$
$$ \hspace{-5mm}
= \, 
\frac{i}{2} \, q^{\frac14} \, \vartheta_{11}(2\tau, 2z)\bigg\{
\frac{\eta(2\tau)^5}{\eta(\tau)^2 \eta(4\tau)^2} 
\cdot 
\frac{\vartheta_{00}(2\tau,2z)}{\vartheta_{01}(2\tau,2z)}
- 2  
\frac{\eta(4\tau)^2}{\eta(2\tau)} \cdot 
\frac{\vartheta_{10}(2\tau,2z)}{\vartheta_{01}(2\tau,2z)} 
- 1\bigg\}
$$
\end{enumerate}
\end{enumerate}
\end{lemma}

\begin{proof} 
1) follows from definition \eqref{n3:eqn:2022-202f1} 
of $P^{[m,s]}(\tau, z)$ and Proposition \ref{n3:prop:2022-116b} 
and the formula \eqref{n3:eqn:2022-119j2}.
2) and 3) follow from 1) and Lemma \ref{n3:lemma:2022-202a}.
\end{proof}

\medskip

Then, in the case $m=4$, we obtain some of honest characters as follows:

%(line=3852)
%(label=n3:prop:2022-204a)
\begin{prop} 
\label{n3:prop:2022-204a}
In the case $m=4$, the honest (super)characters of $H(\Lambda^{[K(4), m_2]})$
\,\ $(m_2 = 1,3)$ \,\ are given by the following formulas:

\begin{enumerate}
\item[{\rm 1)}]
\begin{enumerate}
\item[{\rm (i)}] \, ${\rm ch}^{(+)}_{H(\Lambda^{[K(4), 1]})}(\tau, z)
\,\ = \,\ 
\dfrac{i}{2} \cdot \dfrac{1}{\eta(\frac{\tau}{2}) \, \eta(2\tau)} \cdot 
\dfrac{\vartheta_{10}(\tau, z)}{\vartheta_{01}(\tau, z)} $
$$
\times \,\ \bigg\{
\frac{\eta(2\tau)^5}{\eta(\tau)^2 \eta(4\tau)^2} \, 
\vartheta_{00}(2\tau,2z)
\,\ - \,\ 2 \, 
\frac{\eta(4\tau)^2}{\eta(2\tau)} \, \vartheta_{10}(2\tau,2z) 
\,\ + \,\ \vartheta_{01}(2\tau,2z)\bigg\}
$$
\item[{\rm (ii)}] \, ${\rm ch}^{(+)}_{H(\Lambda^{[K(4), 3]})}(\tau, z)
\,\ = \,\ 
\dfrac{i}{2} \cdot \dfrac{1}{\eta(\frac{\tau}{2}) \, \eta(2\tau)} \cdot 
\dfrac{\vartheta_{10}(\tau, z)}{\vartheta_{01}(\tau, z)}$
$$
\times \,\ \bigg\{
\frac{\eta(2\tau)^5}{\eta(\tau)^2 \eta(4\tau)^2} \, 
\vartheta_{00}(2\tau,2z)
\,\ - \,\ 2 \, 
\frac{\eta(4\tau)^2}{\eta(2\tau)} \, \vartheta_{10}(2\tau,2z) 
\,\ - \,\ \vartheta_{01}(2\tau,2z)\bigg\}
$$
\end{enumerate}

\item[{\rm 2)}]
\begin{enumerate}
\item[{\rm (i)}] \, ${\rm ch}^{(-)}_{H(\Lambda^{[K(4), 1]})}(\tau, z)
\,\ = \,\ 
- \, \dfrac{i}{2} \cdot \dfrac{\eta(\frac{\tau}{2})}{\eta(\tau)^3} \cdot 
\dfrac{\vartheta_{10}(\tau, z)}{\vartheta_{00}(\tau, z)}$
$$
\times \,\ \bigg\{
\frac{\eta(2\tau)^5}{\eta(\tau)^2 \eta(4\tau)^2} \, 
\vartheta_{00}(2\tau,2z)
\,\ + \,\ 2 \, 
\frac{\eta(4\tau)^2}{\eta(2\tau)} \, \vartheta_{10}(2\tau,2z) 
\,\ + \, \vartheta_{01}(2\tau,2z)\bigg\}
$$
\item[{\rm (ii)}] \, ${\rm ch}^{(-)}_{H(\Lambda^{[K(4), 3]})}(\tau, z)
\,\ = \,\ 
- \, \dfrac{i}{2} \cdot \dfrac{\eta(\frac{\tau}{2})}{\eta(\tau)^3} \cdot 
\dfrac{\vartheta_{10}(\tau, z)}{\vartheta_{00}(\tau, z)}$
$$
\times \,\ \bigg\{
\frac{\eta(2\tau)^5}{\eta(\tau)^2 \eta(4\tau)^2} \, 
\vartheta_{00}(2\tau,2z)
\,\ + \,\ 2 \, 
\frac{\eta(4\tau)^2}{\eta(2\tau)} \, \vartheta_{10}(2\tau,2z) 
\,\ - \, \vartheta_{01}(2\tau,2z)\bigg\}
$$
\end{enumerate}
\end{enumerate}
\end{prop}

\begin{proof} By Proposition \ref{n3:prop:2022-111a} and the formulas 
\eqref{n3:eqn:2022-202f1} and \eqref{n3:eqn:2022-202f2}, the 
numerators of the (super)characters of N=3 module 
$H(\Lambda^{[K(m), m_2]})$ are given by the following formulas:
{\allowdisplaybreaks
\begin{eqnarray*}
\big[\overset{N=3}{R}{}^{(+)} \, {\rm ch}^{(+)}_{H(\Lambda^{[K(m), m_2]})}
\big] (\tau, z) &=& q^{-\frac{m}{16}} \, 
Q^{[m, \, \frac{m_2+1}{2}]}(\tau, z)
\\[1mm]
\big[\overset{N=3}{R}{}^{(-)} \, {\rm ch}^{(-)}_{H(\Lambda^{[K(m), m_2]})}
\big] (\tau, z) &=& q^{-\frac{m}{16}} \, 
P^{[m, \, \frac{m_2+1}{2}]}(\tau, z)
\end{eqnarray*}}
Then this proposition follows from Lemma \ref{n3:lemma:2022-205a} and
the formula \eqref{n3:eqn:2022-115a} for the N=3 denominators.
\end{proof}

From Proposition \ref{n3:prop:2022-204b} and 
Proposition \ref{n3:prop:2022-204a}, we see the following:

%(line=3939)
%(label=n3:cor:2022-205b)
\begin{cor}
\label{n3:cor:2022-205b}
In the case $m=4$,  the modified character 
$\widetilde{\rm ch}^{(+)}_{H(\Lambda^{[K(4), 1]})}(\tau,z)$ 
is a sum of honest characters, namely,
$$
\sum_{m_2=1,3}{\rm ch}^{(+)}_{H(\Lambda^{[K(4), m_2]})}(\tau,z)
\,\ = \,\ 2 \cdot 
\widetilde{\rm ch}^{(+)}_{H(\Lambda^{[K(4), 1]})}(\tau,z)
$$
\end{cor}

\medskip

Below we will show that the honest characters are holomorphic functions.

%(line=4207)
%(label=n3:lemma:2022-116a)
\begin{lemma}
\label{n3:lemma:2022-116a}
For $m \in \frac12 \nnn$ and $s \in \frac12 \zzz$, we put 
$$
f(\tau, z) := \Phi^{[m,s]}
\left(2\tau, z+\frac{\tau}{2}, z-\frac{\tau}{2}, 0\right)
$$
Then \quad $f(\tau, n+p\tau) =0$ \quad for \,\ 
${}^{\forall} n, {}^{\forall}p \in \zzz$.
\end{lemma}

\begin{proof} By the definition \eqref{n3:eqn:2022-111c} of $\Phi^{[m,s]}$,
the function $f(\tau, z)$ is written as follows:
$$
f(\tau, z) = f_1(\tau, z) -f_2(\tau, z) 
$$
where
{\allowdisplaybreaks
\begin{eqnarray*}
f_1(\tau, z) &=&\sum_{j \in \zzz}\frac{
e^{4\pi imjz}e^{2\pi isz}q^{2mj^2+2sj+\frac{s}{2}}}{
1-e^{2\pi iz} q^{2j+\frac12}}
\\[1mm]
f_2(\tau, z) &=&\sum_{j \in \zzz}\frac{
e^{-4\pi imjz}e^{-2\pi isz}q^{2mj^2+2sj+\frac{s}{2}}}{
1-e^{-2\pi iz} q^{2j+\frac12}}
\end{eqnarray*}}
Then we have
\begin{subequations}
{\allowdisplaybreaks
\begin{eqnarray}
f_1(\tau, n+p\tau) &=&
(-1)^{2sn}\sum_{j \in \zzz}\frac{
q^{2mj(j+p)+s(2j+p+\frac12)}}{1-q^{2j+p+\frac12}}
\label{n3:eqn:2022-116a}
\\[1mm]
f_2(\tau, n+p\tau) &=&
(-1)^{2sn}\sum_{j \in \zzz}\frac{
q^{2mj(j-p)+s(2j-p+\frac12)}}{1-q^{2j-p+\frac12}}
\label{n3:eqn:2022-116b}
\end{eqnarray}}
Putting $j=k+p$, this equation \eqref{n3:eqn:2022-116b} is rewitten
as follows:
\begin{equation}
f_2(\tau, n+p\tau) =
(-1)^{2sn}\sum_{k \in \zzz}\frac{
q^{2mk(k+p)+s(2k+p+\frac12)}}{1-q^{2k+p+\frac12}}
\label{n3:eqn:2022-116c}
\end{equation}
\end{subequations}
Then by \eqref{n3:eqn:2022-116a} and \eqref{n3:eqn:2022-116c}, 
one has
$$
f(\tau, n+p\tau) = f_1(\tau, n+p\tau)-f_2(\tau, n+p\tau)=0 \, ,
$$
proving lemma.
\end{proof}

From this lemma, we have the following:

%(line=4287)
%(label=n3:lemma:2022-116b)
\begin{lemma}
\label{n3:lemma:2022-116b}
Let $m$ be a positive integer and $m_2$ be a non-negative integer
such that $m_2 \leq m$. Then 
$$
(\overset{N=3}{R}{}^{(-)} \, 
{\rm ch}^{(-)}_{H(\Lambda^{[K(m), m_2]})})(\tau, n+p\tau)=0 
\qquad \text{for} \,\ {}^{\forall} n, {}^{\forall}p \in \zzz
$$
\end{lemma}

\begin{proof} By Proposition \ref{n3:prop:2022-111a}, we have
$$
(\overset{N=3}{R}{}^{(-)} \, 
{\rm ch}^{(-)}_{H(\Lambda^{[K(m), m_2]})})(\tau, z)
=
q^{-\frac{m}{16}}
\Phi^{[\frac{m}{2}, \frac{m_2+1}{2}]}
\left(2\tau, z+\frac{\tau}{2}, z-\frac{\tau}{2}, 0\right) \, .
$$
Then the claim follows from Lemma \ref{n3:lemma:2022-116a}.
\end{proof}

%(line=4294)
%(label=n3:prop:2022-116a)
\begin{prop}
\label{n3:prop:2022-116a}
Let $m$ be a positive integer and $m_2$ be a non-negative integer
such that $m_2 \leq m$. Then the (super)characters 
${\rm ch}^{(\pm)}_{H(\Lambda^{[K(m), m_2]})}(\tau, z)$ 
are holomorphic in the domain $(\tau, z) \in \ccc_+ \times \ccc$.
\end{prop}

\begin{proof} First we consider the super-character. By 
\eqref{n3:eqn:2022-115a} and Proposition \ref{n3:prop:2022-111a}, we have
\begin{equation}
{\rm ch}^{(-)}_{H(\Lambda^{[K(m), m_2]})}(\tau, z) =
\frac{\eta(\frac{\tau}{2})}{\eta(\tau)^3} \cdot 
\frac{\vartheta_{01}(\tau,z)}{\vartheta_{11}(\tau,z)} \cdot 
q^{-\frac{m}{16}}
\Phi^{[\frac{m}{2}, \frac{m_2+1}{2}]}
\left(2\tau, z+\frac{\tau}{2}, z-\frac{\tau}{2}, 0\right) 
\label{n3:eqn:2022-116d}
\end{equation}
where
\begin{eqnarray*}
& &
\Phi^{[\frac{m}{2}, \frac{m_2+1}{2}]}
\left(2\tau, z+\frac{\tau}{2}, z-\frac{\tau}{2}, 0\right)
=
\sum_{j \in \zzz}\frac{
e^{2\pi i mjz+\pi i (m_2+1)z} 
q^{mj^2+(m_2+1)(j+\frac14)}}{1-e^{2\pi iz}q^{2j+\frac12}}
\\[1mm]
& & \hspace{30mm}
-\sum_{j \in \zzz}\frac{
e^{2\pi i mjz-\pi i (m_2+1)z} 
q^{mj^2+(m_2+1)(j+\frac14)}}{1-e^{-2\pi iz}q^{2j+\frac12}}
\end{eqnarray*}
From these equations, we see the following:
\begin{enumerate}
\item[(i)] $\Phi^{[\frac{m}{2}, \frac{m_2+1}{2}]}
\left(2\tau, z+\frac{\tau}{2}, z-\frac{\tau}{2}, 0\right)$ has 
singularities at $e^{\pm 2\pi iz}q^{2j+\frac12}=1$, but these singularities 
are cancelled by $\vartheta_{01}(\tau,z)$.

\item[(ii)] $\dfrac{1}{\vartheta_{11}(\tau,z)}$ has 
singularities at $e^{\pm 2\pi iz}q^j=1$, but these singularities 
are cancelled by  $\Phi^{[\frac{m}{2}, \frac{m_2+1}{2}]}
\left(2\tau, z+\frac{\tau}{2}, z-\frac{\tau}{2}, 0\right)$ 
due to Lemma \ref{n3:lemma:2022-116b}.
\end{enumerate}

Thus all of singularities in the RHS of \eqref{n3:eqn:2022-116d} 
disappear, so ${\rm ch}^{(-)}_{H(\Lambda^{[K(m), m_2]})}(\tau, z)$ 
is holomorphic. Then ${\rm ch}^{(+)}_{H(\Lambda^{[K(m), m_2]})}(\tau, z)$ 
is holomorphic, since it is a scalar multiple of 
${\rm ch}^{(-)}_{H(\Lambda^{[K(m), m_2]})}(\tau+1, z)$. 
\end{proof}

From Corollary \ref{n3:cor:2022-116a} and Corollary \ref{n3:cor:2022-205b},
we conjecture the following:

%(line=4104)
%(label=n3:conj:2022-205a)
\begin{conj} 
\label{n3:conj:2022-205a}
For $m \in \nnn$ and $j \in \{0,1\}$, there exists 
$\big(\mu^{[m]}(m_2)\big)_{m_2 \in \zzz^{[m]}_j}$
such that
\begin{enumerate}
\item[{\rm (i)}] \quad $\mu^{[m]}(m_2) \, \in \, \nnn$ \quad for each 
\,\ $m_2 \in \zzz^{[m]}_j$
\item[{\rm (ii)}] \,\ $
\sum\limits_{m_2 \in \zzz^{[m]}_j} \mu^{[m]}(m_2) \cdot 
{\rm ch}^{(+)}_{H(\Lambda^{[K(m), m_2]})}(\tau, z)$
$$= \,\ 
\Big(\sum\limits_{m_2 \in \zzz^{[m]}_j}\mu^{[m]}(m_2)\Big) \, \times \, 
\widetilde{\rm ch}^{(+)}_{H(\Lambda^{[K(m),j]})}(\tau, z)$$
\end{enumerate}
where \quad $
\zzz^{[m]}_j \,\ := \,\ \big\{m_2 \in \zzz \,\ ; \,\ 
0 \leq m_2 \leq m \,\ \text{and} \,\ m_2 -j \in 2 \zzz\big\} \, .$
\end{conj}

%(line=4126)
%(label=n3:conj:2022-205b)
\begin{conj} 
\label{n3:conj:2022-205b}
For $m \in \nnn$, modified characters are holomorphic functions 
of $(\tau, z) \in \ccc_+ \times \ccc$.
\end{conj}

We note that, if Conjecture \ref{n3:conj:2022-205a} is true, then 
Conjecture \ref{n3:conj:2022-205b} is true by 
Proposition \ref{n3:prop:2022-116a}.
Also note that, if Conjecture \ref{n3:conj:2022-205a} is true, then 
the honest characters ${\rm ch}^{(+)}_{H(\Lambda^{[K(m), m_2]})}(\tau, z)$
\,\ $(m_2 \in \zzz^{[m]}_j)$ can be written explicitly by the modified 
character $\widetilde{\rm ch}{}^{(+)}_{H(\Lambda^{[K(m),j]})}(\tau, z)$,
as is seen in the following Remark \ref{n3:rem:2022-205a}
in the case $m=4$.

%(line=4144)
%(label=n3:rem:2022-205a)
\begin{rem}
\label{n3:rem:2022-205a}
Assume that Conjecture \ref{n3:conj:2022-205a} is true for $m=4$ and 
$j=0$ with $(\mu^{[4]}(0), \mu^{[4]}(2), \mu^{[4]}(4)) = (1,2,1)$.
Then the honest characters ${\rm ch}^{(+)}_{H(\Lambda^{[K(4), m_2]})}(\tau, z)$
\,\ $(m_2 \in \{0,2,4\})$ are expected to be as follows
by using the formula \eqref{n3:eqn:2022-625a} for the modified character 
$\widetilde{\rm ch}^{(+)}_{H(\Lambda^{[K(4),0]})}(\tau, z)$ 
in Conjecture \ref{n3:conj:2022-625a}:

\begin{enumerate}
\item[{\rm 1)}] ${\rm ch}^{(+)}_{H(\Lambda^{[K(4),0]})}(\tau, z)
\,\ = \,\ 
\widetilde{\rm ch}^{(+)}_{H(\Lambda^{[K(4),0]})}(\tau, z)$
$$
+ \, \frac{i}{4} \, q^{-\frac{1}{16}} \frac{1}{\eta(\frac{\tau}{2})\eta(2\tau)} \cdot 
\frac{\vartheta_{00}(\tau,z)}{\vartheta_{11}(\tau,z)} \, \big[
-3(\theta_{\frac12,2}-\theta_{-\frac12,2})+(\theta_{\frac32,2}-\theta_{-\frac32,2})
\big](2\tau,2z)
$$

\item[{\rm 2)}] ${\rm ch}^{(+)}_{H(\Lambda^{[K(4),2]})}(\tau, z)
\,\ = \,\ 
\widetilde{\rm ch}^{(+)}_{H(\Lambda^{[K(4),0]})}(\tau, z)$
$$
+ \, \frac{i}{4} \, q^{-\frac{1}{16}} \frac{1}{\eta(\frac{\tau}{2})\eta(2\tau)} \cdot 
\frac{\vartheta_{00}(\tau,z)}{\vartheta_{11}(\tau,z)} \, \big[
(\theta_{\frac12,2}-\theta_{-\frac12,2})+(\theta_{\frac32,2}-\theta_{-\frac32,2})
\big](2\tau,2z)
$$

\item[{\rm 3)}] ${\rm ch}^{(+)}_{H(\Lambda^{[K(4),4]})}(\tau, z)
\,\ = \,\ 
\widetilde{\rm ch}^{(+)}_{H(\Lambda^{[K(4),0]})}(\tau, z)$
$$
+ \, \frac{i}{4} \, q^{-\frac{1}{16}} \frac{1}{\eta(\frac{\tau}{2})\eta(2\tau)} \cdot 
\frac{\vartheta_{00}(\tau,z)}{\vartheta_{11}(\tau,z)} \, \big[
(\theta_{\frac12,2}-\theta_{-\frac12,2})-3(\theta_{\frac32,2}-\theta_{-\frac32,2})
\big](2\tau,2z)
$$
\end{enumerate}
\end{rem}

\section{$\vartheta$-relations}
\label{sec:theta-relation}

Using Lemma \ref{n3:lemma:2022-111e} and the formulas for 
$\overset{\circ}{A}{}^{[2]}_j(\tau,z)$ in Lemma \ref{n3:lemma:2022-108e}, 
we obtain the following formulas for $\vartheta_{ab}$:

%(line=4470)
%(label=n3:prop:2022-120b)
\begin{prop} \,\ 
\label{n3:prop:2022-120b}
\begin{enumerate}
\item[{\rm 1)}] \quad $
\dfrac{\vartheta_{00}(2\tau, z)}{\vartheta_{10}(2\tau, z)}
\,\ + \,\ 
\dfrac{\vartheta_{10}(2\tau, z)}{\vartheta_{00}(2\tau, z)}
\,\ = \,\ 
\dfrac{\eta(\tau)^6}{\eta(\frac{\tau}{2})^2 \, \eta(2\tau)^4} \cdot 
\dfrac{\vartheta_{00}(\tau, z)}{\vartheta_{10}(\tau, z)} $

\item[{\rm 2)}] \quad $
\dfrac{\vartheta_{00}(2\tau, z)}{\vartheta_{10}(2\tau, z)}
\,\ - \,\ 
\dfrac{\vartheta_{10}(2\tau, z)}{\vartheta_{00}(2\tau, z)}
\,\ = \,\ 
\dfrac{\eta(\frac{\tau}{2})^2}{\eta(2\tau)^2} \cdot 
\dfrac{\vartheta_{01}(\tau, z)}{\vartheta_{10}(\tau, z)} $
\end{enumerate}
\end{prop}

\begin{proof}
Letting $m=2$ in the formulas \eqref{n3:eqn:2022-117d5} 
and \eqref{n3:eqn:2022-117d6} and using Lemma \ref{n3:lemma:2022-108e}
for the explicit formulas for $\overset{\circ}{A}{}^{[2]}_5(\tau, z)$
and $\overset{\circ}{A}{}^{[2]}_6(\tau, z)$, we have 
%(label=n3:eqn:2022-120e1)
%(label=n3:eqn:2022-120e2)
\begin{subequations}
{\allowdisplaybreaks
\begin{eqnarray}
& &\hspace{-10mm}
\widetilde{\Phi}^{[1,0]}\Big(2\tau, \,  z-\frac12, \, z+\frac12, \, 0\Big) 
\, = \, 
\overset{\circ}{A}{}^{[2]}_5(\tau, z)
\,\ = \,\ 
\frac{i}{2} \, \big\{g^{(+)}_3(\tau,z)+g^{(+)}_2(\tau,z)\big\}
\nonumber
\\[1mm]
&=&
\frac{i}{2} \cdot \frac{\vartheta_{11}(\tau,z)}{\vartheta_{10}(\tau,z)}\left\{
\frac{\eta(\tau)^5}{\eta(\frac{\tau}{2})^2 \eta(2\tau)^2} 
\cdot \vartheta_{00}(\tau,z)
+
\frac{\eta(\frac{\tau}{2})^2}{\eta(\tau)} \cdot 
\vartheta_{01}(\tau,z)\right\} \hspace{10mm}
\label{n3:eqn:2022-120e1}
\\[1mm]
& & \hspace{-10mm}
\widetilde{\Phi}^{[1,0]}\Big(2\tau, \,  z+\tau-\frac12, \, 
z-\tau+\frac12, \, \frac{\tau}{2}\Big) 
\, = \, 
\overset{\circ}{A}{}^{[2]}_6(\tau, z)
\nonumber
\\[1mm]
&=&
\frac{i}{2} \, \big\{
g^{(+)}_3(\tau,z)-g^{(+)}_2(\tau,z)\big\}
\nonumber
\\[1mm]
&=&
\frac{i}{2} \cdot \frac{\vartheta_{11}(\tau,z)}{\vartheta_{10}(\tau,z)}\left\{
\frac{\eta(\tau)^5}{\eta(\frac{\tau}{2})^2 \eta(2\tau)^2} 
\cdot \vartheta_{00}(\tau,z)
-
\frac{\eta(\frac{\tau}{2})^2}{\eta(\tau)} \cdot 
\vartheta_{01}(\tau,z)\right\}
\label{n3:eqn:2022-120e2}
\end{eqnarray}}
\end{subequations}

We compute the LHS of \eqref{n3:eqn:2022-120e1} and \eqref{n3:eqn:2022-120e2}.
By Lemma \ref{n3:lemma:2022-111e}, these are rewritten as follows:
%(label=n3:eqn:2022-120f1)
%(label=n3:eqn:2022-120f2)
\begin{subequations}
{\allowdisplaybreaks
\begin{eqnarray}
& & \hspace{-10mm}
\text{LHS of \eqref{n3:eqn:2022-120e1}} \, = \, 
- \, i \, \eta(2\tau)^3 \, 
\frac{\vartheta_{11}(2\tau, \, 2z)
}{\vartheta_{11}(2\tau, \, z-\frac12) \, \vartheta_{11}(2\tau, \, z+\frac12)}
\nonumber
\\[1mm]
&=& i \, \frac{\eta(2\tau)^2}{\eta(\tau)} \cdot 
\frac{\vartheta_{00}(2\tau, z)}{\vartheta_{10}(2\tau, z)} \cdot 
\vartheta_{11}(\tau, z)
\label{n3:eqn:2022-120f1}
\\[1mm]
& & \hspace{-10mm}
\text{LHS of \eqref{n3:eqn:2022-120e2}}
\, = \, 
- i \, q^{-\frac12} \eta(2\tau)^3 
\frac{\vartheta_{11}(2\tau, \, 2z)}{\vartheta_{11}(2\tau, z+\tau-\frac12) 
\vartheta_{11}(2\tau, z-\tau+\frac12)}
\nonumber
\\[1mm]
&=& i \, \frac{\eta(2\tau)^2}{\eta(\tau)} \cdot 
\frac{\vartheta_{10}(2\tau, z)}{\vartheta_{00}(2\tau, z)} \cdot 
\vartheta_{11}(\tau, z)
\label{n3:eqn:2022-120f2}
\end{eqnarray}}
\end{subequations}
where we used the formula \eqref{n3:eqn:2022-120g1} and the following 
formulas:
{\allowdisplaybreaks
\begin{eqnarray*}
& &
\vartheta_{11}\Big(2\tau, \, z-\frac12\Big)
\vartheta_{11}\Big(2\tau, \, z+\frac12\Big)
\,\ = \,\ - \, \vartheta_{10}(2\tau, z)^2
\\[1mm]
& &
\vartheta_{11}\Big(2\tau, \, z+\tau-\frac12\Big)
\vartheta_{11}\Big(2\tau, \, z-\tau+\frac12\Big)
\,\ = \,\ - \, q^{-\frac12} \, \vartheta_{00}(2\tau, z)^2 
\end{eqnarray*}}
and
$$
\left\{
\begin{array}{ccc}
\dfrac{\vartheta_{10}(\tau, z)}{\vartheta_{10}(2\tau, z)} &=&
\dfrac{\eta(\tau)}{\eta(2\tau)^2} \,\ \vartheta_{00}(2\tau, z)
\\[4mm]
\dfrac{\vartheta_{10}(\tau, z)}{\vartheta_{00}(2\tau, z)} &=&
\dfrac{\eta(\tau)}{\eta(2\tau)^2} \,\ \vartheta_{10}(2\tau, z)
\end{array}\right.
$$
Then we have, by \eqref{n3:eqn:2022-120e1} and \eqref{n3:eqn:2022-120f1},
\begin{subequations}
\begin{equation}
\frac{\eta(2\tau)^2}{\eta(\tau)} \cdot 
\frac{\vartheta_{00}(2\tau, z)}{\vartheta_{10}(2\tau, z)} 
=
\frac12 \cdot \frac{1}{\vartheta_{10}(\tau,z)}\left\{
\frac{\eta(\tau)^5}{\eta(\frac{\tau}{2})^2 \eta(2\tau)^2} 
\vartheta_{00}(\tau,z)
+
\frac{\eta(\frac{\tau}{2})^2}{\eta(\tau)} 
\vartheta_{01}(\tau,z)\right\}
\label{n3:eqn:2022-120h1}
\end{equation}
and, by \eqref{n3:eqn:2022-120e2} and \eqref{n3:eqn:2022-120f2},
\begin{equation}
\frac{\eta(2\tau)^2}{\eta(\tau)} \cdot 
\frac{\vartheta_{10}(2\tau, z)}{\vartheta_{00}(2\tau, z)} 
=
\frac12 \cdot \frac{1}{\vartheta_{10}(\tau,z)}\left\{
\frac{\eta(\tau)^5}{\eta(\frac{\tau}{2})^2 \eta(2\tau)^2} 
\vartheta_{00}(\tau,z)
-
\frac{\eta(\frac{\tau}{2})^2}{\eta(\tau)}
\vartheta_{01}(\tau,z)\right\} .
\label{n3:eqn:2022-120h2}
\end{equation}
\end{subequations}
Now, by making $\text{\eqref{n3:eqn:2022-120h1}} \pm \text{\eqref{n3:eqn:2022-120h2}}$,
we have
\begin{eqnarray*}
\frac{\eta(2\tau)^2}{\eta(\tau)} \, \left\{
\frac{\vartheta_{00}(2\tau, z)}{\vartheta_{10}(2\tau, z)} 
\, + \, 
\frac{\vartheta_{10}(2\tau, z)}{\vartheta_{00}(2\tau, z)}\right\}
&=&
\frac{\eta(\tau)^5}{\eta(\frac{\tau}{2})^2 \eta(2\tau)^2} \cdot
\frac{\vartheta_{00}(\tau,z)}{\vartheta_{10}(\tau,z)}
\\[1mm]
\frac{\eta(2\tau)^2}{\eta(\tau)} \, \left\{
\frac{\vartheta_{00}(2\tau, z)}{\vartheta_{10}(2\tau, z)} 
\, - \, 
\frac{\vartheta_{10}(2\tau, z)}{\vartheta_{00}(2\tau, z)}\right\}
&=&
\frac{\eta(\frac{\tau}{2})^2}{\eta(\tau)} \cdot 
\frac{\vartheta_{01}(\tau,z)}{\vartheta_{10}(\tau,z)}
\end{eqnarray*}
namely
\begin{eqnarray*}
\frac{\vartheta_{00}(2\tau, z)}{\vartheta_{10}(2\tau, z)} 
\, + \, 
\frac{\vartheta_{10}(2\tau, z)}{\vartheta_{00}(2\tau, z)}
&=&
\frac{\eta(\tau)^6}{\eta(\frac{\tau}{2})^2 \eta(2\tau)^4} \cdot
\frac{\vartheta_{00}(\tau,z)}{\vartheta_{10}(\tau,z)}
\\[1mm]
\frac{\vartheta_{00}(2\tau, z)}{\vartheta_{10}(2\tau, z)} 
\, - \, 
\frac{\vartheta_{10}(2\tau, z)}{\vartheta_{00}(2\tau, z)}
&=&
\frac{\eta(\frac{\tau}{2})^2}{\eta(2\tau)^2} \cdot 
\frac{\vartheta_{01}(\tau,z)}{\vartheta_{10}(\tau,z)} \, ,
\end{eqnarray*}
proving proposition.
\end{proof}

\end{document}